\documentclass[reqno,12pt]{amsart}
\usepackage{graphicx} 
\usepackage[margin=1in]{geometry}

\usepackage{mathtools, amssymb, tikz, amsthm, amsmath}
\usetikzlibrary{calc}
\usepackage[pdfencoding=auto,psdextra]{hyperref}
\DeclareMathAlphabet{\mathbbold}{U}{bbold}{m}{n}
\newcommand{\one}{\mathbbold{1}}
\usepackage{faktor}
\usepackage{enumitem}
\usepackage[dvipsnames]{xcolor}

\usepackage[
    backend=biber,
    isbn=false,
    url=false,
    sorting=nyt,
    maxbibnames=8,
]{biblatex}
\theoremstyle{plain}
\newtheorem{introtheorem}{Theorem}

\newtheorem{theorem}{Theorem}[section]

\newtheorem{proposition}[theorem]{Proposition}
\newtheorem{lemma}[theorem]{Lemma}
\newtheorem{corollary}[theorem]{Corollary}

\theoremstyle{definition}
\newtheorem{definition}[theorem]{Definition}
\newtheorem*{definition*}{Definition}
\newtheorem{remark}[theorem]{Remark}
\newtheorem*{remark*}{Remark}

\newtheorem{example}[theorem]{Example}

  \newcommand{\ce}{\mathcal{E}}
  \newcommand{\cg}{\mathcal{G}}
  \newcommand{\cl}{\mathcal{L}}
  
  \newcommand{\cv}{\mathcal{V}}

 \newcommand{\C}{\mathbb{C}}
 \newcommand{\R}{\mathbb{R}}

 \newcommand{\ds}{\displaystyle}
 \newcommand{\poly}{\C[x,y]}

 \newcommand{\splines}[1]{\mathcal{M}_{#1}} 
 \newcommand{\Ms}{\splines{H}} 
 \newcommand{\lqot}[1]{\mathrm{L}_{#1}} 
 \newcommand{\rqot}[1]{\mathrm{R}_{#1}} 
 \newcommand{\spl}[1]{\overline{#1}} 
 \newcommand{\hinv}{\mathrm{inv}_H} 
 \newcommand{\hlen}{\ell_H}
 
\newcommand{\supp}{\mathrm{supp}} 

\DeclareMathOperator{\Hilb}{Hilb}
\DeclareMathOperator{\spann}{span}
\DeclareMathOperator{\sgn}{sgn}

 \newcommand{\deftobe}{\coloneq} 
 \DeclareMathOperator{\GL}{GL}

  \DeclarePairedDelimiter{\lrverts}{\lvert}{\rvert}
  \newcommand{\abs}{\lrverts*} 
  \newcommand{\card}{\lrverts*} 
  \newcommand{\llangle}{\left\langle}
  \newcommand{\rrangle}{\right\rangle}

  \newcommand{\nst}[1][]{\mathchoice{\;}{\;}{\,}{\,}
    #1\vert 
    \allowbreak\mathchoice{\;}{\;}{\,}{\,}\mathopen{}}%
    \newcommand{\st}{}
    \DeclarePairedDelimiterX{\nsetof}[1]{\lbrace}{\rbrace}{\renewcommand{\st}{\nst[\delimsize]}#1}
    \newcommand{\setof}{\nsetof*}

  \DeclarePairedDelimiterX{\ngenby}[1]{\langle}{\rangle}{\renewcommand{\st}{\nst[\delimsize]}#1}
  \newcommand{\genby}{\ngenby*}

 \definecolor{darkgreen}{rgb}{0, .6, 0}

\usepackage[normalem]{ulem}

\title{Dot actions and Hessenberg spaces for dihedral groups}

\author[Boretsky]{Jonathan Boretsky}
\address{Department of Mathematics and Statistics, McGill University, Montreal, QC}
\email{\href{mailto:jonathan.boretsky@mcgill.ca}{jonathan.boretsky@mcgill.ca}}

\author[Hodges]{Reuven Hodges}
\address{Department of Mathematics, University of Kansas, Lawrence, KS 66045}
\email{\href{mailto:rmhodges@ku.edu}{rmhodges@ku.edu}}

\author[Lesnevich]{Nathan Lesnevich}
\address{Department of Mathematics, Oklahoma State University, Stillwater, OK 74078}
\email{\href{mailto:nlesnev@okstate.edu}{nlesnev@okstate.edu}}

\author[McAllister]{Tyrrell McAllister}
\address{Department of Mathematics and Statistics, University of Wyoming, Laramie, WY, 82071}
\email{\href{mailto:tmcallis@uwyo.edu}{tmcallis@uwyo.edu}}

\author[Pappe]{Joseph Pappe}
\address{Department of Mathematics, University of Virginia, Charlottesville, VA 22904}
\email{\href{mailto:yyz8xu@virginia.edu}{yyz8xu@virginia.edu}}


\begin{document}

\begin{abstract}
Splines on Weyl groups model the equivariant cohomology of regular semisimple Hessenberg varieties, and the natural group action on them encodes the monodromy action. The resulting representations are of great interest in algebraic combinatorics for their connections to Hecke algebra characters, LLT polynomials, and chromatic symmetric functions. We study the analogous construction for dihedral groups, which in general have no underlying geometry. Using purely algebraic methods, we show that these spline modules are free, construct explicit bases, and compute their Hilbert series, recovering the familiar properties of the geometric setting. We also determine the associated dihedral group representations completely. Unlike what is conjectured for the symmetric group case, these need not be permutation representations.
\end{abstract}

\maketitle

\setcounter{tocdepth}{1}
\tableofcontents

\section{Introduction}

Let $I_2(n)$, with $n\geq 3$, be the dihedral group of order $2n$, with root system $\Phi$ embedded in $\C^2$. A \emph{Hessenberg space} $H$ is a lower order ideal in the poset of positive roots $\Phi^+$. We insist that $H$ contain the simple roots. Each root $\gamma \in \Phi$ is naturally associated to both a linear form $t_\gamma$ in the polynomials $\C[x,y]$ and a reflection $s_\gamma$ in $I_2(n)$. As such, the \emph{module of splines} given by $H$ is defined to be the set 
\[
\Ms = \bigoplus_{i \geq 0} (\Ms)_i \coloneqq \left\{\left.\spl{\rho} \in \!\!\!\prod_{w \in I_2(n)}\!\!\poly\;\right|\;\spl{\rho}(w) - \spl{\rho}(ws_\gamma) \in \llangle t_{w(\gamma)} \rrangle \text{ for all }w \in I_2(n),\, \gamma \in H\right\}
\]
with (graded) $I_2(n)$-module structure $(w\cdot \spl{\rho})(v) = w(\spl{\rho}(w^{-1}v))$ and (graded) $\poly$-module structure given by multiplication.

The module of splines is of particular interest in geometric combinatorics for its connections to chromatic symmetric functions and certain subvarieties of the flag variety. Indeed, if $W$ is a Weyl group and $H$ a Hessenberg space, then the corresponding module of splines is isomorphic to the $T$-equivariant cohomology of the associated \emph{regular semisimple Hessenberg variety}. The connection between this combinatorial module and this geometric object is called \emph{GKM theory} \cite{GKM_theory}. In this Weyl group case, the ring structure is the cohomology ring structure, the module structure is over the equivariant cohomology of a point, and the $W$-module structure is equivalent to the \emph{monodromy action} \cite{BalibanuCrooks_sheaves,brosnan_chow_dotactn_is_chromsym}.

Now the dihedral group $I_2(n)$ is not \emph{crystallographic} in general, and thus does not have an associated flag variety, let alone Hessenberg varieties. As such, our methods are entirely algebraic. In spite of this, the module $\Ms$ displays remarkable similarities to its geometric counterpart. We prove for example that $\Ms$ is a free module over the polynomial ring whose Hilbert series is given by an inversion-type statistic on $I_2(n)$, precisely as it is in the geometric case.
\begin{introtheorem}\label{intthm:poincare_hilbert}
    The module $\Ms$ is a free module over $\poly$. The rank generating function for $\Ms$ is palindromic and given by \[P_H(q) = \sum_{w \in I_2(n)} q^{\hlen(w)},\]
    where $\hlen(w)$ is the number of roots in $H$ sent to $\Phi^-$ by $w$.    Furthermore, the Hilbert series of $\Ms$ is 
    \[\Hilb(\Ms;q) = \frac{P_H(q)}{(1-q)^2}.\]
\end{introtheorem}
Theorem \ref{intthm:poincare_hilbert} is Corollary \ref{cor:rank_gf} below.

The freeness of this module in the case that $W$ is a Weyl group is intrinsic to the geometry, and the Hilbert series are given by a Betti number computation. Moreover, this Betti number calculation is given by an \emph{affine paving} of the associated Hessenberg variety \cite{Precup_paving, tymoczko2006linear}. The associated GKM classes to these pavings are largely unknown. It is an open question in all Weyl types to give explicit cohomology classes that correspond to this paving. While these classes can be described geometrically \cite{BialynickiBirula1973,BialynickiBirula1976,ChoHongLee_Bases}, explicit formulas for their localizations, i.e. for the corresponding splines, are known only in special cases \cite{chohonglee_second_cohom,lesnevich2025signedpermsdot}.  We give an explicit homogeneous module basis for $\Ms$ that satisfies precisely the properties that would follow from a paving.

\begin{introtheorem}\label{intthm:basis}
    There exists a homogeneous basis
    $\{\spl{\rho}_w \mid w \in I_2(n)\}$ of $\Ms$, with
    $\deg \spl{\rho}_w = \hlen(w)$, such that each
    $\spl{\rho}_w$ is either a Schubert spline or has the form
    \begin{equation*}
        \spl{\rho}_w(v) =
        \begin{dcases*}
            c^w(v)
            \prod_{\substack{\alpha \in H, \\
                            vs_\alpha \notin G_w}}
            t_{v(\alpha)}, & if $v \in G_w$,\\[6pt]
            0, & if $v \notin G_w$,
        \end{dcases*}
    \end{equation*}
    where $G_w$ contains $w$ and is contained in its Bruhat
    upper ideal, and $c^w(v) \in \C^\times$ for every
    $v \in G_w$.
\end{introtheorem}
Theorem \ref{thm:flowup} constructs this basis recursively,
Proposition \ref{prop:flowup_explicit_setup} describes its supports,
and Corollary \ref{cor:rationally_smooth} gives the displayed
localization formula. The Schubert splines satisfy an analogous
product formula in the full reflection graph, given in
Theorem \ref{thm:schubert_rational}.

Finally, we completely determine the $I_2(n)$-module structure of $\Ms$. Specifically, we compute the (graded) left and right \emph{dot action representations}
\[\lqot{H} = \bigoplus_{i=0}^{\abs{H}}\, (\lqot{H})_i \qquad \qquad \rqot{H} = \bigoplus_{i=0}^{\abs{H}}\, (\rqot{H})_i,\]
each derived from an equivariant quotient of $\Ms$. Determining this structure for Weyl groups is a subject of significant research. For the symmetric group, the left and right quotients correspond to passing from equivariant to ordinary cohomology of the regular semisimple Hessenberg variety and its twin, respectively. There, the left representation is, under the Frobenius characteristic map and up to the standard involution $\omega$ on symmetric functions, a \emph{chromatic quasisymmetric function} \cite{brosnan_chow_dotactn_is_chromsym,guaypaquet2016shar_wachs_conj,SW2016chromaticquasisymmetric}, and the right representation is similarly a \emph{unicellular LLT polynomial} \cite{ALEXANDERSSON2018LLTchromsym,Ayzenberg2018isospectral, guaypaquet2016shar_wachs_conj}. For $W = S_n$,  the decomposition of $\lqot{H}$ into irreducibles is known \cite{SW2016chromaticquasisymmetric}. It is a major open problem to determine whether these graded representations are permutation representations of a particular form, i.e.~to prove or disprove the \emph{graded Stanley--Stembridge Conjecture} \cite{brosnan_chow_dotactn_is_chromsym,guaypaquet2016shar_wachs_conj, SW2016chromaticquasisymmetric, STANLEY1995chromsym}, whose ungraded statement was proven by Hikita \cite{hikita2024SSconjecture} and later reproven by Griffin, et al.~\cite{Griffin_Mellit_et_al}. A proof of the graded statement has recently been claimed by OpenAI \cite{OpenAI2026ChromaticPositivity}. 

In other Weyl types, in particular the classical types B/C/D, the $W$-module structure has received recent attention for its deep connections to Hecke algebras and other geometric and representation-theoretic objects \cite{Abreu2025heckealg, BalibanuCrooks_sheaves, precup2022perverse,skandera2025hyperoctahedral,}. In these cases, there are no combinatorial formulae that compute the left or right representations in full, and it is an active area of research to find ones \cite{BalibanuCrooks_sheaves,lesnevich2025signedpermsdot, precup2022perverse}. It is also unknown whether these left representations are permutation representations of a particular type \cite{lesnevich2025signedpermsdot}. We settle each of these questions in their entirety for $I_2(n)$.
\begin{introtheorem}\label{intthm:dotaction}
    Let $H \subsetneq \Phi^+$ be a Hessenberg space containing the simple roots $\{\alpha,\beta\}$ and let $h = \lfloor \abs{H} / 2 \rfloor$. If $\abs{H}$ is odd, let $\alpha$ be the $H$-dominant simple root. Then, as an element of the graded representation ring of $I_2(n)$,
    \begin{equation*}
        \lqot{H}  \cong
        \begin{dcases*}
            (1+q)(1+q+\cdots+q^{2h-1})\one
            \oplus
            q^h\left(C_{\alpha} \oplus C_\beta - \mathsf{P}_h - \mathsf{P}_{h-1}\right)
                & if $\abs{H}=2h$,\\[10pt]
            (1+q)(1+q+\cdots+q^{2h})\one
            \oplus
            (q^h+q^{h+1})\left(C_\alpha - \mathsf{P}_h\right)
                & if $\abs{H}=2h+1$
        \end{dcases*}
    \end{equation*}
    and
    \[
    \rqot{H}  \cong
    \begin{cases}
        \ds
        \mathsf{Co}_H
        \oplus
        q^h\left(C_{\alpha} \oplus C_\beta - \mathsf{P}_h - \mathsf{P}_{h-1}\right)
            &\text{if } \abs{H}=2h,\\[10pt]
        \ds
        \mathsf{Co}_H
        \oplus
        q^h\left(C_\alpha - \mathsf{P}_h\right)
        \oplus 
        q^{h+1}\left(C_\beta - \mathsf{P}_h\right)
            &\text{if } \abs{H}=2h+1.
    \end{cases}
    \]
\end{introtheorem}
The definition of an $H$-dominant simple root is Definition~\ref{def:H_dominant}. The irreducible representations $\one$ and $V_i$ as well as the permutation representations $C_\gamma$ are defined in Subsection~\ref{ssec:representations}, whereas $\mathsf{P}_h$ is defined in Section~\ref{sec:left_rep} and $\mathsf{Co}_H$ is defined in Lemma~\ref{lem:coinv_right_subreps}. Theorem~\ref{intthm:dotaction} is Theorems~\ref{thm:left_representation} and~\ref{thm:right_representation} below. We note that the left representation is not necessarily a permutation representation in Corollary~\ref{cor:not_permutation}.

This paper is structured as follows. Section~\ref{sec:background} provides the necessary background. Section~\ref{sec:schubert} treats the case where $H$ is the set of positive roots. This is the \emph{Schubert case}, corresponding to the Schubert calculus of flag varieties. Much of the content of this section is known in greater generality~\cite{MR3391071,KajiEquivariantSchubertCoxeter, Kumar_KacMoody}, but our treatment is entirely algebraic and from first principles. Section~\ref{sec:flowup} gives an iterative construction for the bases from Theorem~\ref{intthm:basis}. Section~\ref{sec:properties} computes the properties of this basis, both completing the description of Theorem~\ref{intthm:basis} and setting up the necessary calculations for the following sections. Finally, Section~\ref{sec:left_rep} computes the left dot action representation and Section~\ref{sec:right_rep} computes the right dot action representation, which together give Theorem~\ref{intthm:dotaction}.

\subsection*{Acknowledgments}
This collaboration began at the Graduate Research Workshop in Combinatorics (GRWC) 2025. The workshop was supported in part by NSF DMS-2152490, Barbara Jansons Professorship, the Combinatorics Foundation, and Iowa State University. The authors are grateful for the opportunity to collaborate despite coming from different areas of combinatorics, which would not have materialized without the environment and support provided by the GRWC. We also thank Sheila Sundaram for fruitful discussions and for helping to accept this project into the GRWC. JB thanks NSERC for their support through grant PDF 587781. JB remercie le CRSNG pour son soutien dans le cadre de la bourse PDF 587781. The authors used GPT Sol 6.1 and Claude Opus 5.5 to assist with a
Lean~4~\cite{deMouraUllrich2021} formalization of all results proved
in this paper~\cite{HodgesLean2026}, using the Mathlib
library~\cite{mathlib2020}. The authors wrote the arguments presented in this paper
and take full responsibility for its mathematical content.


\section{Background}\label{sec:background}
\subsection{Vector Algebra}
Fix a complex vector space $V \cong \C^n$ with symmetric bilinear form $\langle\cdot,\cdot\rangle$ given by the bilinear extension of the standard inner product on $\R^n$. If $\alpha \in V$ satisfies $\langle\alpha,\alpha\rangle\neq 0$, then the reflection $s_\alpha$ across its orthogonal hyperplane is given by
\[s_\alpha(v) = v - \frac{2\llangle v,\alpha\rrangle}{\llangle \alpha,\alpha \rrangle} \alpha. \]
Let $e_1,\ldots,e_n$ be the standard basis for $V$. There is a natural map from $V$ to the linear polynomials in $\C[t_1,\ldots,t_n]$ defined as the linear extension of $e_i \mapsto t_i$. For a particular vector $\alpha=\sum c_ie_i \in V$, we denote by $t_\alpha = \sum c_it_i \in \C[t_1,\ldots,t_n]$ the linear form associated to it by this map. Via this identification, any linear map $X$ on $V$ may be extended to the polynomial ring $\C[t_1,\ldots,t_n]$ by acting on each variable individually: $X$ maps the monomial $\prod_i t_i^{m_i}$ to $\prod_i X(t_i)^{m_i}$, and extends linearly.

\begin{example}
    If $V = \C^2$ and $\ds X = \begin{bmatrix*}[r]
        -1 & 1 \\ 0 & 2
    \end{bmatrix*}$, then 
    \begin{align*}
        X(t_1) &= -t_1     &&\text{since }X(e_1) = -e_1,\\
        X(t_2) &= t_1+2t_2 &&\text{since }X(e_2) = e_1+2e_2.
    \end{align*}
    This action extends multiplicatively to higher degree monomials; for example $X(t_1t_2^2) = X(t_1)X(t_2)^2 = (-t_1)(t_1+2t_2)^2$. It also extends linearly to polynomials; for example, $X(t_1+t_2)=X(t_1)+X(t_2)=2t_2$.
\end{example}
\begin{lemma}\label{lem:reflect_divisible}
Let $s_\alpha$ be a reflection on $\C^n$. If $f$ is any polynomial in $\C[t_1,\ldots,t_n]$, then $f-s_\alpha(f)$ is divisible by $t_\alpha$.
\end{lemma}
\begin{proof}
    Consider the action of an arbitrary $X = [c_{ij}] \in \GL_n(\C)$ on $\C[t_1,\ldots,t_n]$. This action interacts with evaluation on $v \in \C^n$ via:
    \[Xf(v) = f(X^Tv).\]
    It suffices to verify this identity on the generators $t_j$, and certainly 
    \[Xt_j(v) = \sum_{i=1}^n c_{ij}t_i(v) = \sum_{i=1}^n c_{ij}v_i = t_j(X^Tv).\]
     Since reflections $s_\alpha$ are symmetric linear operators, the vanishing set of $f - s_\alpha(f)$ contains any point on the reflecting hyperplane $H_\alpha = \mathcal{V}(t_\alpha)$. Formally, if $v \in H_\alpha$ then 
     \[
     s_\alpha(f)(v) = f(s_\alpha ^T v) = f(s_\alpha v) = f(v).
     \]
     Since $t_\alpha$ is irreducible (it is linear), by nullstellensatz $f - s_\alpha(f)$ is a multiple of $t_\alpha$.
\end{proof}

The remainder of this work considers the dihedral group $I_2(n)$, as naturally extended from its typical $\R^2$ embeddings to $V = \C^2$, so the polynomial ring is $\C[x,y]$. We continue to use $t_\alpha$ to denote the linear polynomial corresponding to the vector $\alpha$.

\subsection{The Dihedral Group and Root Systems}
The \emph{dihedral group} 
\begin{equation*}
    I_2(n) 
        \deftobe
        \genby{
            s_1,s_2 
            \st 
            s_1^2 = s_2^2 = (s_1s_2)^n = e
        }
\end{equation*}
is the group of symmetries of a regular $n$-gon. We employ here the definition that gives $I_2(n)$ as the finite rank-two Coxeter group of order $2n$.

The reflections $s_1 \coloneqq s_{\alpha_1}$ and $s_2 \coloneqq s_{\alpha_2}$ form a minimal generating set. For each $w \in I_2(n)$, there is a corresponding set of  \textit{reduced words}, that is, minimal ways of expressing $w$ as a product of $s_1$ and $s_2$. The \emph{length} of an element $w$, denoted $\ell(w)$, is the length of any reduced word for $w$. Note that reduced words are alternating sequences of $s_1$ and $s_2$ since $s_1^2 = s_2^2 = e$. Every element of the dihedral group has only one reduced word except for the longest element~$w_0$, which is of length $n$ and can start with either $s_1$ or $s_2$ and alternate between them $n$ times, since $(s_1 s_2)^n = e$. The \emph{Bruhat order} $\leq$ on $I_2(n)$ is defined letting $w \leq v$ when $w$ appears as a subword of a reduced word for $v$.

The dihedral group is also an example of a \textit{real reflection group}. Such groups are generated by reflections $s_i = s_{\alpha_{i}}$ across particular hyperplanes defined by normal vectors $\alpha_i \in \R^2$, called \emph{roots}. We take all roots to be unit length.
A real reflection group corresponds to a \emph{real root system} $\Phi$, which is a finite set of real-valued vectors which satisfies the following:
\begin{itemize}
    \item if $\alpha \in \Phi$ then $-\alpha \in \Phi$,
    \item if $\alpha,\beta \in \Phi$ and $\alpha = c\beta$ then $c = \pm 1$, and
    \item if $\alpha,\beta \in \Phi$ then $s_\alpha(\beta) \in \Phi$.
\end{itemize}

The correspondence between real reflection groups and real root systems is that the set of reflections $\{s_\alpha\mid \alpha\in \Phi\}$ generates a group, which is a real reflection group. A generic hyperplane will partition $\Phi$ into ``positive'' roots $\Phi^+$ and ``negative'' roots $\Phi^-$, so that there is a one-to-one correspondence between $\Phi^+$ and the reflections. Moreover, the \emph{simple roots} are the unique pair of positive roots closest to the generic hyperplane, and are always the positive roots that correspond to the generators $s_1,s_2$. We write $\Delta$ for the set of simple roots.

We index the positive roots of $I_2(n)$ so that $s_1 = s_{\alpha_1}$, $s_2 = s_{\alpha_2}$, and more generally 
\[(s_1s_2)^ks_1 = s_{\alpha_{2k+1}} \qquad\text{and}\qquad (s_2s_1)^ks_2 = s_{\alpha_{2k+2}}\]
for $2k+1\leq n$ and $2k+2\leq n$, respectively. Our proofs will not rely on a particular embedding, but in examples we will often place $\alpha_{2k+1}$ at angle $k\pi/n$ and $\alpha_{2k+2}$ at angle $(n-1-k)\pi/n$.

\begin{example} \label{ex:reflection}
    Below, the red lines show reflection symmetries and the arrows show a choice of positive roots for $I_2(3)$, $I_2(4)$, and $I_2(5)$.
    
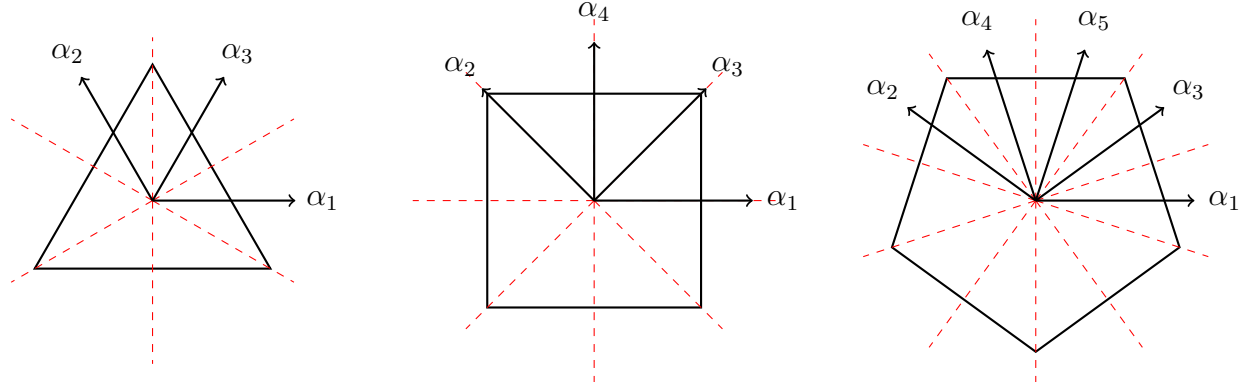
\begin{center}
        \begin{tikzpicture}
            \begin{scope}[scale=0.9]
                \def\angle{360/3}
                \def\radius{2}
                
                \foreach \i in {0,1,2} {
                    \path (\i*\angle-30:\radius) coordinate (A\i);
                }
                \foreach \i in {0,1,2,3,4,5} {
                    \path (\i*\angle/2-30:\radius+0.4) coordinate (B\i);
                }
                \draw[thick] (A0)--(A1) -- (A2) -- cycle;
                \foreach \i in {0,1,2,3,4,5} {
                    \draw[dashed, red] (0,0) -- (B\i);
                }
                \foreach \i in {0,1,2} {
                    \draw[thick,->] (0,0) -- (\i*\angle/2:\radius+0.1);
                }
                \draw (0*\angle/2:\radius+0.5) node (Al0) {$\alpha_{1}$};     
                \draw (1*\angle/2:\radius+0.5) node (Al1) {$\alpha_{3}$};
                \draw (2*\angle/2:\radius+0.5) node (Al2) {$\alpha_{2}$};

            \end{scope}
            \begin{scope}[scale=1,xshift=2.3in]
                \def\angle{360/4}
                \def\radius{2}
                
                \foreach \i in {0,1,2,3} {
                    \path (\i*\angle-45:\radius) coordinate (A\i);
                }
                \foreach \i in {0,1,2,3,4,5,6,7} {
                    \path (\i*\angle/2-45:\radius+0.4) coordinate (B\i);
                }
                \draw[thick] (A0)--(A1) -- (A2) -- (A3) -- cycle;
                \foreach \i in {0,1,2,3,4,5,6,7} {
                    \draw[dashed, red] (0,0) -- (B\i);
                }
                \foreach \i in {0,1,2,3} {
                    \draw[thick,->] (0,0) -- (\i*\angle/2:\radius+0.1);
                }
                \draw (0*\angle/2:\radius+0.5) node (Al0) {$\alpha_{1}$};     
                \draw (1*\angle/2:\radius+0.5) node (Al1) {$\alpha_{3}$};
                \draw (2*\angle/2:\radius+0.5) node (Al2) {$\alpha_{4}$};
                \draw (3*\angle/2:\radius+0.5) node (Al3) {$\alpha_{2}$};

            \end{scope}
            \begin{scope}[scale=1,xshift=4.6in]
                \def\angle{360/5}
                \def\radius{2}
                
                \foreach \i in {0,1,2,3,4} {
                    \path (\i*\angle-18:\radius) coordinate (A\i);
                }
                \foreach \i in {0,1,2,3,4,5,6,7,8,9} {
                    \path (\i*\angle/2-18:\radius+0.4) coordinate (B\i);
                }
                \draw[thick] (A0)--(A1) -- (A2) -- (A3) -- (A4)-- cycle;
                \foreach \i in {0,1,2,3,4,5,6,7,8,9} {
                    \draw[dashed, red] (0,0) -- (B\i);
                }
                \foreach \i in {0,1,2,3,4} {
                    \draw[thick,->] (0,0) -- (\i*\angle/2:\radius+0.1);
                }
                \draw (0*\angle/2:\radius+0.5) node (Al0) {$\alpha_{1}$};   
                \draw (1*\angle/2:\radius+0.5) node (Al1) {$\alpha_{3}$};
                \draw (2*\angle/2:\radius+0.5) node (Al2) {$\alpha_{5}$};
                \draw (3*\angle/2:\radius+0.5) node (Al3) {$\alpha_{4}$};
                \draw (4*\angle/2:\radius+0.5) node (Al4) {$\alpha_{2}$};

            \end{scope}
            
        \end{tikzpicture}
    \end{center}
\end{example}

For $w\in I_2(n)$, let
\begin{equation}\label{eq:left_inversion_set}
    N_L(w)\deftobe
    \setof{\alpha\in\Phi^+ \st \ell(s_\alpha w)<\ell(w)}
\end{equation}
be the left inversion set of $w$. 
Similarly, let
\begin{equation}\label{eq:right_inversion_set}
    N_R(w)\deftobe
    \setof{\alpha\in\Phi^+ \st \ell(w s_\alpha)<\ell(w)}
\end{equation}
be the right inversion set of $w$.

\begin{lemma}\label{lem:rank_two_bruhat_facts}
The following hold for all $w \in I_2(n)$.

\begin{enumerate}
    \item If $w\notin\{e,w_0\}$, then $w$ has a unique simple left descent
    and a unique simple left ascent. That is, exactly one of
    $s_{\alpha_1}w,s_{\alpha_2}w$ has length $\ell(w)-1$, and exactly one
    has length $\ell(w)+1$.

    \item If $0<\ell(u)<n$, then $u\leq v$ if and only if either
    $v=u$ or $\ell(v)>\ell(u)$.

    \item For every $w\in I_2(n)$, one has $\card{N_L(w)}=\ell(w)$. If
    $\alpha\in\{\alpha_1,\alpha_2\}$ and
    $\ell(s_\alpha w)=\ell(w)+1$, then
    \begin{align*}
         N_L(s_\alpha w)&=\{\alpha\}\sqcup s_\alpha N_L(w), \\
         N_R(s_\alpha w)&=\{w^{-1}(\alpha)\}\sqcup  N_R(w).
    \end{align*}
\end{enumerate}
\end{lemma}

\begin{proof}
Every reduced word in $I_2(n)$ is
alternating, and for each $1\leq k<n$ the two alternating words of length
$k$ represent distinct elements. The two alternating words of length $n$
represent the longest element $w_0$. Hence every element other than $e$
and $w_0$ has a unique nonempty reduced word.

Now let $0<\ell(w)<n$, and write the unique reduced word for $w$ as
$w=s_{\alpha_i}s_{\alpha_j}\cdots$. Left multiplication by $s_{\alpha_i}$
cancels the first letter and lowers length by one. Left multiplication by the
other simple reflection gives an alternating word of length
$\ell(w)+1\leq n$, so it raises length by one. This proves (1).

For (2), use the subword criterion for Bruhat order. If
$u\leq v$, then $\ell(u)\leq\ell(v)$, and equality forces $u=v$. Thus,
if $u\leq v$, then either $v=u$ or $\ell(v)>\ell(u)$. Conversely,
suppose $\ell(v)>\ell(u)$. Every alternating reduced word of length
$\ell(v)$ contains both alternating words of length $\ell(u)$ as subwords.
Thus a reduced word for $u$ occurs as a subword of a reduced word for $v$,
so $u \leq v$.

For (3), use the usual rank-two inversion criterion:
$\ell(s_\beta w)<\ell(w)$ if and only if $w^{-1}(\beta)\in\Phi^-$. Thus
$N_L(w)$ is the set of positive roots sent negative by $w^{-1}$. A minimal
gallery from the fundamental chamber $C$ to the chamber $wC$ crosses exactly
$\ell(w)$ reflecting hyperplanes, so $\card{N_L(w)}=\ell(w)$. Finally, if
$\ell(s_\alpha w)=\ell(w)+1$, then $w^{-1}(\alpha)\in\Phi^+$. Since a
simple reflection sends only its own simple root from $\Phi^+$ to $\Phi^-$,
the positive roots sent negative by $(s_\alpha w)^{-1}=w^{-1}s_\alpha$ are
precisely $\alpha$, together with the roots $s_\alpha\beta$ for
$\beta\in N_L(w)$. This is the first line of (3). The same logic applies for $N_R(w)$.
\end{proof}

\subsection{Representations}\label{ssec:representations}
The representation theory of the dihedral group is very well understood. We record several standard facts here. These can be found, for example, in \cite[Chapter 5, \S 5.3]{Serre}, or follow from direct computation.

All irreducible representations of $I_2(n)$ are either one- or two-dimensional. If $n$ is odd, then $I_2(n)$ has $2$ one-dimensional irreducible representations and $(n-1)/2$ two-dimensional irreducible representations. If $n$ is even, then $I_2(n)$ has $4$ one-dimensional irreducible representations and $(n-2)/2$ two-dimensional irreducible representations. 

The trivial representation is $\one(w) \coloneqq (1)$ and the sign representation is $\sgn(w) \coloneqq (-1)^{\ell(w)}$. 

Let $r=s_1s_2$. When $n$ is even, let $\eta_0$ and $\eta_1$ denote the one-dimensional
representations determined by
\[
    \eta_0(r)=-1,
    \qquad
    \eta_0(s_1)=1,
    \qquad
    \eta_1(r)=-1,
    \qquad
    \eta_1(s_1)=-1.
\]

For each integer $j$ with $1\leq j<n/2$, let $V_j$
denote the two-dimensional irreducible representation of $I_2(n)$ whose
character is given by
\[
    \chi_{V_j}(r^a)
        =
    2\cos\left(\frac{2\pi ja}{n}\right),
    \qquad
    \chi_{V_j}(r^as_1)
        =
    0
\]
for $a\in\mathbb Z$. In particular, $V_1$ is the defining representation, which sends $w$ to the element of $\GL_2(\C)$ that defines it as a symmetry of the $n$-gon. These, together with the two one-dimensional representations for $n$ odd or the four one-dimensional representations for $n$ even, constitute all of the irreducible representations of $I_2(n)$.

There is one more family of representations, indexed by positive roots, that will be critical for our results. For $\gamma\in\Phi^+$, when $n$ is even let $\eta_\gamma$ be the unique representation in
$\{\eta_0,\eta_1\}$ satisfying $\eta_\gamma(s_\gamma)=1$. Let
\[
    C_\gamma
        \deftobe
    \operatorname{Ind}_{\genby{s_\gamma}}^{I_2(n)}\one
\]
be the permutation representation on the cosets of
$\genby{s_\gamma}$. Its irreducible decomposition is
\begin{equation}
\label{eq:reflection_coset_decomposition}
    C_\gamma
    \cong
    \begin{cases}
        \displaystyle
        \one
        \oplus
        \bigoplus_{j=1}^{(n-1)/2}V_j,
            & n \text{ is odd},\\[12pt]
        \displaystyle
        \one
        \oplus
        \eta_\gamma
        \oplus
        \bigoplus_{j=1}^{n/2-1}V_j,
            & n \text{ is even}.
    \end{cases}
\end{equation}

\begin{example}
        Consider $I_2(4)$, generated by $s_1\coloneqq s_{\alpha_1}$ and $s_2\coloneqq s_{\alpha_2}$. The longest element is $s_1s_2s_1s_2=s_2s_1s_2s_1$. The trivial representation always evaluates to $1$, so $\one(s_1)=\one(s_2)=1$. The sign representation measures the parity of the lengths of words as expressions in $s_1$ and $s_2$, so $\sgn(s_1)=\sgn(s_2)=-1$. Also, $\sgn(s_{\alpha_3})=-1$ since $s_{\alpha_3}=s_1s_2s_1$.

        Finally, $V_1(s_1)=\begin{pmatrix}
            -1&0\\0&1
        \end{pmatrix}$, since the reflection associated to $\alpha_1$ is the reflection across the $y$-axis, and $V_1(s_2)=\begin{pmatrix}
            0 & 1\\ 1&0
        \end{pmatrix}$, since the reflection associated to $\alpha_2$ is the reflection across $y=x$, which swaps $x$ and $y$.
    \end{example}

\subsection{The Root Poset}
For each choice of simple roots in the root system for $I_2(n)$, the positive roots $\Phi^+$ have a natural partial order.
\begin{definition}\label{def:root_poset}
    Let $\alpha_1,\alpha_2$ be simple roots. The root poset of $I_2(n)$ is the partial order on $\Phi^+$ given by $\alpha \leq \beta$ whenever $\beta - \alpha \in \mathrm{Cone}(\alpha_1,\alpha_{2}) \coloneqq \{c_1\alpha_1 + c_2\alpha_2 \mid c_1,c_2 \in \R_{\geq 0}\}$ \cite{ChenKriloff_noncrystal}.
\end{definition}

The Hasse diagrams for these posets are given by the following:
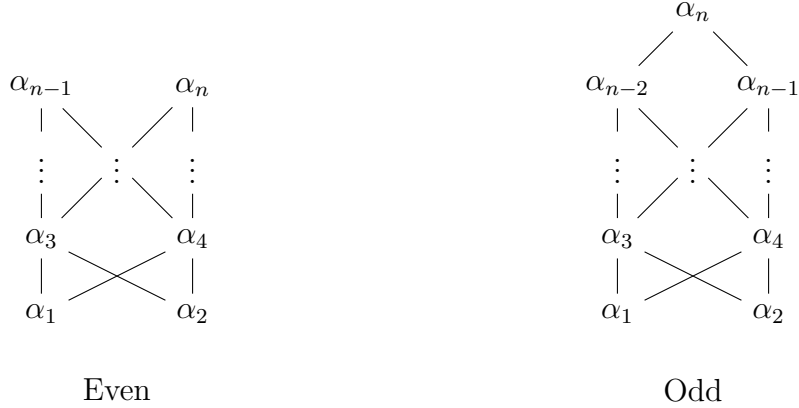
\begin{center}
\begin{tikzpicture}
\begin{scope}
    \draw (-1,0) node (a0) {$\alpha_{1}$};
    \draw (1,0) node (an1) {$\alpha_{2}$};
    \draw (-1,1) node (a1) {$\alpha_{3}$};
    \draw (1,1) node (an2) {$\alpha_{4}$};
    \draw (-1,2) node (dtsl) {$\vdots$};
    \draw (0,2) node (dtsc) {$\vdots$};
    \draw (1,2) node (dtsr) {$\vdots$};
    \draw (-1,3) node (al) {$\alpha_{n-1}$};
    \draw (1,3) node (ar) {$\alpha_{n}$};

    \draw (a0)--(a1)--(dtsl)--(al);
    \draw (an1)--(an2)--(dtsr)--(ar);
    \draw (a0)--(an2)--(dtsc)--(al);
    \draw (an1)--(a1)--(dtsc)--(ar);
    \draw (0,-1) node (name) {Even};
\end{scope}  
\begin{scope}[xshift = 3in]
    \draw (-1,0) node (a0) {$\alpha_{1}$};
    \draw (1,0) node (an1) {$\alpha_{2}$};
    \draw (-1,1) node (a1) {$\alpha_{3}$};
    \draw (1,1) node (an2) {$\alpha_{4}$};
    \draw (-1,2) node (dtsl) {$\vdots$};
    \draw (0,2) node (dtsc) {$\vdots$};
    \draw (1,2) node (dtsr) {$\vdots$};
    \draw (-1,3) node (al) {$\alpha_{n-2}$};
    \draw (1,3) node (ar) {$\alpha_{n-1}$};
    \draw (0,4) node (a) {$\alpha_{n}$};

    \draw (a0)--(a1)--(dtsl)--(al)--(a);
    \draw (an1)--(an2)--(dtsr)--(ar)--(a);
    \draw (a0)--(an2)--(dtsc)--(al);
    \draw (an1)--(a1)--(dtsc)--(ar);
    \draw (0,-1) node (name) {Odd};
\end{scope} 
\end{tikzpicture}
\end{center}
\begin{remark}
    One notable (and standard) property of these root posets is that $\alpha_i < \alpha_j$ in the root poset if and only if the corresponding reflections satisfy $\ell(s_{\alpha_i}) < \ell(s_{\alpha_j})$, which is equivalent to $s_{\alpha_i} < s_{\alpha_j}$ in the strong Bruhat order on $I_2(n)$.
\end{remark}
We consider $H \subseteq \Phi^+$ where $H$ is a lower order ideal in the root poset. We always insist that $H$ contain the simple roots $\alpha_1$ and $\alpha_{2}$. We call such $H$ a \emph{Hessenberg space}. As such the possible $H$ are:
\[
\begin{matrix}
    \{\alpha_1,\alpha_2, \ldots, \alpha_{j-1},\alpha_j\}\qquad\text{and}\qquad\{\alpha_1,\alpha_2 \ldots, \alpha_{2k-2}, \alpha_{2k} \}
\end{matrix}
\]
for $j \in \{2,\ldots,n\}$ and $k \in \{2,\ldots,\lfloor n/2\rfloor\}$.

\begin{remark}
    Our characterization covers more subsets of roots than it may seem. In particular, any convex subset of positive roots can be written as a Hessenberg space up to a different choice of simple roots. For example in $I_2(6)$, up to sign, the convex subset $\{\alpha_1,\alpha_3,\alpha_5\}$ of positive roots can be made into a Hessenberg space by declaring $\alpha_3$ and $-\alpha_5$ to be the simple roots. See Figure~\ref{fig:CharacterizingConvexSubsets}.
    \begin{figure}
        \begin{tikzpicture}[scale = 0.8]
            \begin{scope}[xshift=0in]
                \def\angle{360/6}
                \def\radius{2}

                \foreach \i in {0,1,2,3,4,5} {
                    \path (\i*\angle:\radius) coordinate (A\i);
                }
                \foreach \i in {0,1,2,3,4,5,6,7,8,9,10,11} {
                    \path (\i*\angle/2:\radius+0.4) coordinate (B\i);
                }
                \draw[thick] (A0) -- (A1) -- (A2) -- (A3) -- (A4) -- (A5) -- cycle;
                \foreach \i in {0,1,2,3,4,5,6,7,8,9,10,11} {
                    \draw[dashed, red] (0,0) -- (B\i);
                }
                \foreach \i in {0,1,2,3,4,5} {
                    \draw[thick,->] (0,0) -- (\i*\angle/2:\radius+0.1);
                }
                \draw (0*\angle/2:\radius+0.5) node (Al0) {$\alpha_{1}$};
                \draw (1*\angle/2:\radius+0.5) node (Al1) {$\alpha_{3}$};
                \draw (2*\angle/2:\radius+0.5) node (Al2) {$\alpha_{5}$};
                \draw (3*\angle/2:\radius+0.5) node (Al3) {$\alpha_{6}$};
                \draw (4*\angle/2:\radius+0.5) node (Al4) {$\alpha_{4}$};
                \draw (5*\angle/2:\radius+0.5) node (Al5) {$\alpha_{2}$};

                \draw[blue, thick, rounded corners=6pt, fill=blue, fill opacity=0.08]
                    (210:0.35) -- (-12:\radius+0.85) arc (-12:72:\radius+0.85) -- cycle;
            \end{scope}

            \begin{scope}[xshift=3in]
                \def\angle{360/6}
                \def\radius{2}

                \foreach \i in {0,1,2,3,4,5} {
                    \path (\i*\angle:\radius) coordinate (A\i);
                }
                \foreach \i in {0,1,2,3,4,5,6,7,8,9,10,11} {
                    \path (\i*\angle/2:\radius+0.4) coordinate (B\i);
                }
                \draw[thick] (A0) -- (A1) -- (A2) -- (A3) -- (A4) -- (A5) -- cycle;
                \foreach \i in {0,1,2,3,4,5,6,7,8,9,10,11} {
                    \draw[dashed, red] (0,0) -- (B\i);
                }
                \foreach \i in {0,1,-1,-2,-3,-4} {
                    \draw[thick,->] (0,0) -- (\i*\angle/2:\radius+0.1);
                }
                \draw (0*\angle/2:\radius+0.5) node (Al0) {$\alpha_{1}$};
                \draw (1*\angle/2:\radius+0.5) node (Al1) {$\alpha_{3}$};
                \draw (2*\angle/2:-\radius-0.5) node (Al2) {$-\alpha_{5}$};
                \draw (3*\angle/2:-\radius-0.5) node (Al3) {$-\alpha_{6}$};
                \draw (4*\angle/2:-\radius-0.5) node (Al4) {$-\alpha_{4}$};
                \draw (5*\angle/2:-\radius-0.5) node (Al5) {$-\alpha_{2}$};

                \draw[blue, thick, rounded corners=10pt, fill=blue, fill opacity=0.08]
                    (-12:\radius+0.95) arc (-12:42:\radius+0.95) -- (135:0.35)
                    -- (228:\radius+0.95) arc (228:252:\radius+0.95) -- (300:0.35) -- cycle;
            \end{scope}
        \end{tikzpicture} 
        \caption{The set $\setof{\alpha_{1}, \alpha_{3}, \alpha_{5}}$ in $I_2(6)$ as a Hessenberg space.}
        \label{fig:CharacterizingConvexSubsets}
    \end{figure}
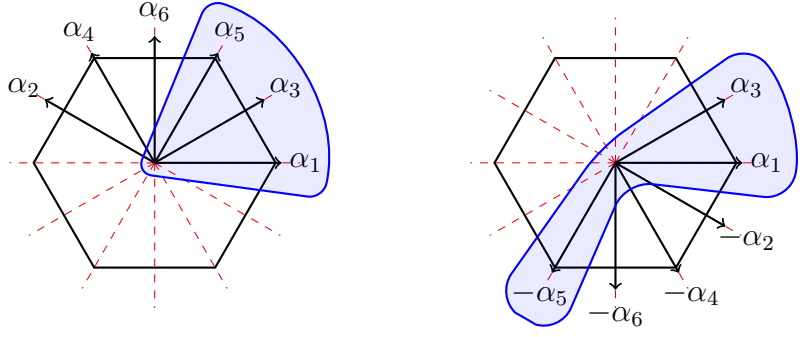 
\end{remark}

\begin{remark}
    Since $I_2(n)$ is not crystallographic in general, there are several reasonable ways to define a root poset that do not necessarily coincide. This root poset does not \emph{precisely} align with the special cases $G_2 \cong I_2(6)$,  $A_2 \cong I_2(3)$ and $B_2,C_2 \cong I_2(4)$. For $A_2$ the root poset is identical. The root posets $B_2$ and $C_2$ are different, but every lower order ideal for $I_2(4)$ is a lower order ideal for the $B_2$ or $C_2$ root system, and so for our purposes are the same. The root poset for $G_2$ is entirely different, as neither family of ideals contains the other, but we do not explore this exceptional group here.
\end{remark}
\begin{definition} \label{def:inversion_set}
    Given $H\subseteq\Phi^+$, the \emph{$H$-inversion set} of
    $w\in I_2(n)$ is
    \[
        \hinv(w)
            \coloneqq
        \{\alpha\in H\mid \ell(ws_\alpha)<\ell(w)\} = H \cap N_R(w).
    \]
    The map $\alpha\mapsto-w(\alpha)$ identifies $\hinv(w)$ with
    $w(-H)\cap\Phi^+$. The $H$-length is 
    \[\hlen(w) = \abs{\hinv(w)}.\]
\end{definition}

Given $H \subseteq \Phi^+$, the associated \emph{Cayley graph} is the graph with $I_2(n)$ as its vertices and $\{(w,ws_\alpha) \mid \alpha \in H\}$ as its edges. We impose two additional structures: first and foremost the \emph{labeled Cayley graph} $\cg_H = (\cv,\ce_H,\cl)$ has vertices $\cv = I_2(n)$, edges $\ce_H = \{(w,ws_\alpha) \mid w \in I_2(n),\; \alpha \in H\}$ and edge labeling $\cl\colon \ce_H \to \C[x,y]$ given by $\cl(w,ws_\alpha) \coloneqq t_{w(\alpha)}$ (up to multiplication by a unit).

\begin{example} \label{ex:cayley}
    Below we have the labeled Cayley graphs for $I_2(3)$ with $H = \{\alpha_1,\alpha_2,\alpha_3\}$ and $I_2(4)$ with $H = \{\alpha_1,\alpha_2\}$. For $I_2(3)$, we kept the labels in terms of the corresponding roots $w(\alpha)$, and for $I_2(4)$ we translated the labels into $\C[x,y]$ for the embedding given in Example~\ref{ex:reflection}, and scaled by a unit for readability. On the left, we have $\hinv(s_2s_1) = \{\alpha_1, \alpha_3\}$.
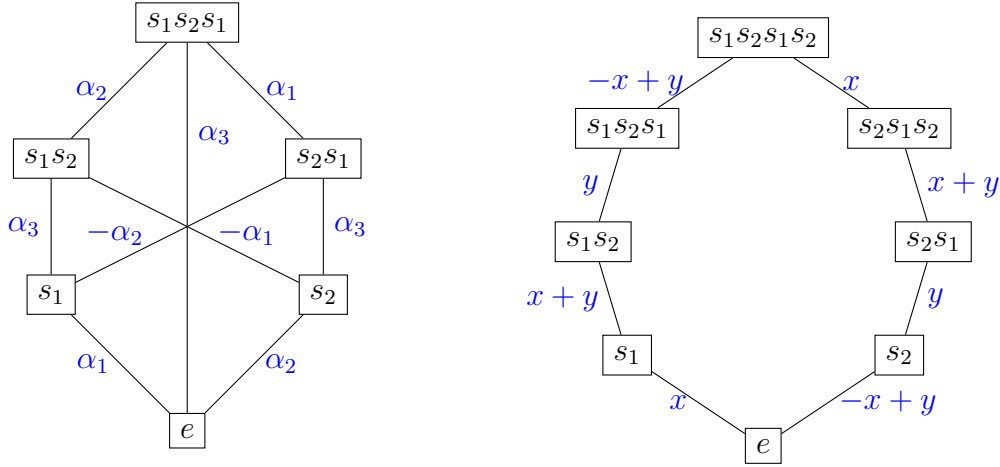
\begin{center}
\begin{tikzpicture}
    \begin{scope}[yshift=0.5cm, scale=1.8]
        \draw  (0,0) node[draw] (v0) {$e$};
        \draw  (1,1) node[draw]  (v1) {$s_2$};
        \draw  (-1,1) node[draw] (v2) {$s_1$};
        \draw  (-1,2) node[draw] (v3) {$s_1s_2$};
        \draw  (1,2) node[draw] (v4) {$s_2s_1$};
        \draw (0,3)  node[draw]  (v5) {$s_1s_2s_1$};
        
        \draw (v0)--(v1)node[midway,right] {\textcolor{blue}{ $\alpha_2$}};
        \draw (v0)--(v2)node[midway,left] {\textcolor{blue}{ $\alpha_1$}};
        \draw (v1)--(v4)node[midway,right] {\textcolor{blue}{ $\alpha_3$}};
        \draw (v2)--(v3)node[midway,left] {\textcolor{blue}{ $\alpha_3$}};
        \draw (v3)--(v5)node[midway,left] {\textcolor{blue}{ $\alpha_2$}};
        \draw (v4)--(v5)node[midway,right] {\textcolor{blue}{ $\alpha_1$}};
        \draw (v0)--(v5)node[near end,right] {\textcolor{blue}{ $\alpha_3$}};
        \draw (v1)--(v3)node[near start,above] {\textcolor{blue}{ $-\alpha_1$}};
        \draw (v2)--(v4)node[near start,above] {\textcolor{blue}{ $-\alpha_2\;\;\;$}};
        
    \end{scope}
    \begin{scope}[xshift=3in, scale=1.5]
        \tikzstyle{every node}=[shape=rectangle, inner sep=4pt]
        \draw  (0,0.2) node[draw] (v0) {$e$};
        \draw  (1.2,1) node[draw]  (v1) {$s_2$};
        \draw  (-1.2,1) node[draw] (v2) {$s_1$};
        \draw  (-1.5,2) node[draw] (v3) {$s_1 s_2$};
        \draw  (1.5,2) node[draw] (v4) {$s_2 s_1$};
        \draw (1.2,3) node[draw] (v5) {$s_2 s_1 s_2$};
        \draw (-1.2,3) node[draw] (v6) {$s_1 s_2s_1$};
        \draw (0,3.8)  node[draw]  (v7) {$s_1 s_2s_1s_2$};
        
        \draw (v0)--(v1) node[midway,right] {\textcolor{blue}{$-x+y$}};
        \draw (v0)--(v2) node[midway,left] {\textcolor{blue}{$x$}};
        \draw (v1)--(v4)node[midway,right] {\textcolor{blue}{$y$}};
        \draw (v2)--(v3) node[midway,left] {\textcolor{blue}{$x+y$}};
        \draw (v3)--(v6) node[midway,left] {\textcolor{blue}{$y$}};
        \draw (v4)--(v5)node[midway,right] {\textcolor{blue}{$x+y$}};
        \draw (v5)--(v7)node[midway,right] {\textcolor{blue}{$x$}};
        \draw (v6)--(v7)node[midway,left] {\textcolor{blue}{$-x+y$}};

    \end{scope}
\end{tikzpicture}
\end{center}
\end{example}

Second, the Cayley graph has a natural acyclic orientation by insisting that edges are directed from elements of lesser length to elements of greater length. In this case $\hlen(w)$ is the in-degree of $w$ in this directed graph. 

As such, there is a natural partial order $\prec_H$ given by the transitive closure of $w\prec_H ws_\alpha$ for all $\alpha \in H$ where $\ell(w) < \ell(ws_\alpha)$. 
\begin{example}
    When $H = \{\alpha_1,\alpha_2\}$, this is the right weak Bruhat order on $I_2(n)$. When $H = \Phi^+$, this is the strong Bruhat order on $I_2(n)$.
\end{example}
\begin{example}
    If $n = 5$ and $H = \{\alpha_1,\alpha_2,\alpha_3\}$, the directed graph and Hasse diagram for $\prec_H$ are:
    
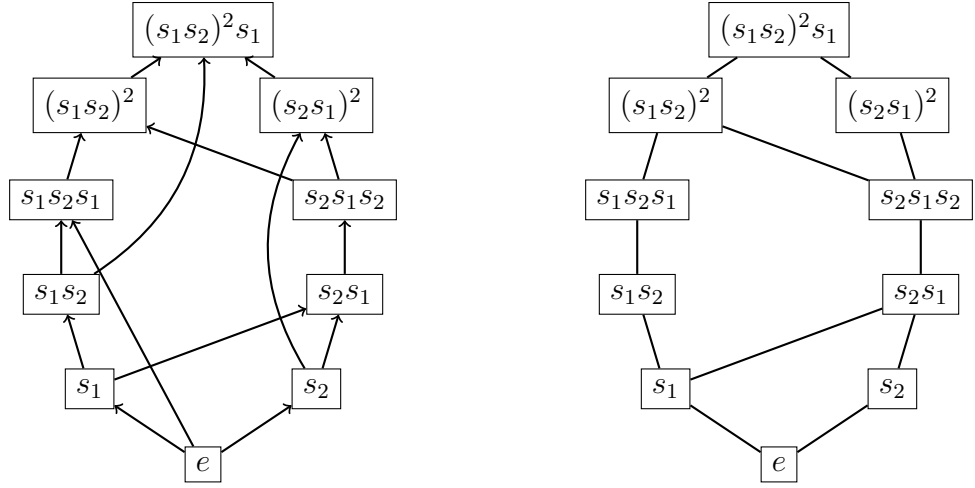
\begin{center}
        \begin{tikzpicture}
    \begin{scope}[xshift=0in, scale=1.25]
        \tikzstyle{every node}=[shape=rectangle, inner sep=4pt]
        \draw  (0,0.2) node[draw] (v0) {$e$};
        \draw  (1.2,1) node[draw]  (v1) {$s_2$};
        \draw  (-1.2,1) node[draw] (v2) {$s_1$};
        \draw  (-1.5,2) node[draw] (v3) {$s_1 s_2$};
        \draw  (1.5,2) node[draw] (v4) {$s_2 s_1$};
        \draw (1.5,3) node[draw] (v5) {$s_2 s_1 s_2$};
        \draw (-1.5,3) node[draw] (v6) {$s_1 s_2s_1$};
        \draw (-1.2,4)  node[draw]  (v7) {$(s_1 s_2)^2$};
        \draw (1.2,4)  node[draw]  (v8) {$(s_2 s_1)^2$};
        \draw (0,4.8)  node[draw]  (v9) {$(s_1 s_2)^2s_1$};

        \draw[->, thick] (v0)--(v1);
        \draw[->, thick] (v1)--(v4);
        \draw[->, thick] (v4)--(v5);
        \draw[->, thick] (v5)--(v8);
        \draw[->, thick] (v8)--(v9);
        \draw[<-, thick] (v9)--(v7);
        \draw[<-, thick] (v7)--(v6);
        \draw[<-, thick] (v6)--(v3);
        \draw[<-, thick] (v3)--(v2);
        \draw[<-, thick] (v2)--(v0);
        \draw[->, thick] (v0)--(v6);
        \draw[->, thick] (v1) to[bend left=30] (v8);
        \draw[->, thick] (v2) to (v4);
        \draw[->, thick] (v3) to[bend right=30] (v9);
        \draw[->, thick] (v5) to (v7);
        
    \end{scope}
    \begin{scope}[xshift=3in, scale=1.25]
        \tikzstyle{every node}=[shape=rectangle, inner sep=4pt]
        \draw  (0,0.2) node[draw] (v0) {$e$};
        \draw  (1.2,1) node[draw]  (v1) {$s_2$};
        \draw  (-1.2,1) node[draw] (v2) {$s_1$};
        \draw  (-1.5,2) node[draw] (v3) {$s_1 s_2$};
        \draw  (1.5,2) node[draw] (v4) {$s_2 s_1$};
        \draw (1.5,3) node[draw] (v5) {$s_2 s_1 s_2$};
        \draw (-1.5,3) node[draw] (v6) {$s_1 s_2s_1$};
        \draw (-1.2,4)  node[draw]  (v7) {$(s_1 s_2)^2$};
        \draw (1.2,4)  node[draw]  (v8) {$(s_2 s_1)^2$};
        \draw (0,4.8)  node[draw]  (v9) {$(s_1 s_2)^2s_1$};

        \draw[thick] (v0)--(v1)--(v4)--(v5)--(v8)--(v9)--(v7)--(v6)--(v3)--(v2)--(v0);
        \draw[thick] (v2) to (v4);
        \draw[thick] (v5) to (v7);
        
    \end{scope}
\end{tikzpicture}
    \end{center}
\end{example}

\subsection{Splines and their Algebraic Properties}
Given a subset $H\subseteq \Phi^+ $ of roots whose associated reflections generate $I_2(n)$, the  \emph{set of splines} is:
\[\Ms \coloneqq \left\{\spl{\rho} \colon I_2(n) \to \C[x,y] \mid \spl{\rho}(w) - \spl{\rho}(ws_\alpha) \in \llangle t_{w(\alpha)} \rrangle,\text{ for all } w \in I_2(n),\;\;\alpha \in H   \right\}.\]
We will always insist that $H$ is a Hessenberg space. This construction is more easily seen as a \emph{generalized spline} on the labeled Cayley graph $\cg_H$. An element of $\Ms$ is an assignment of polynomials to vertices of $\cg_H$ such that the difference of two polynomials along an edge is divisible by the label on that edge.
The following are some splines on $I_2(3)$ with $H = \{\alpha_1,\alpha_2,\alpha_3\}$. The labels are given by the positive root $\alpha$ instead of the polynomials $t_\alpha$ for readability.
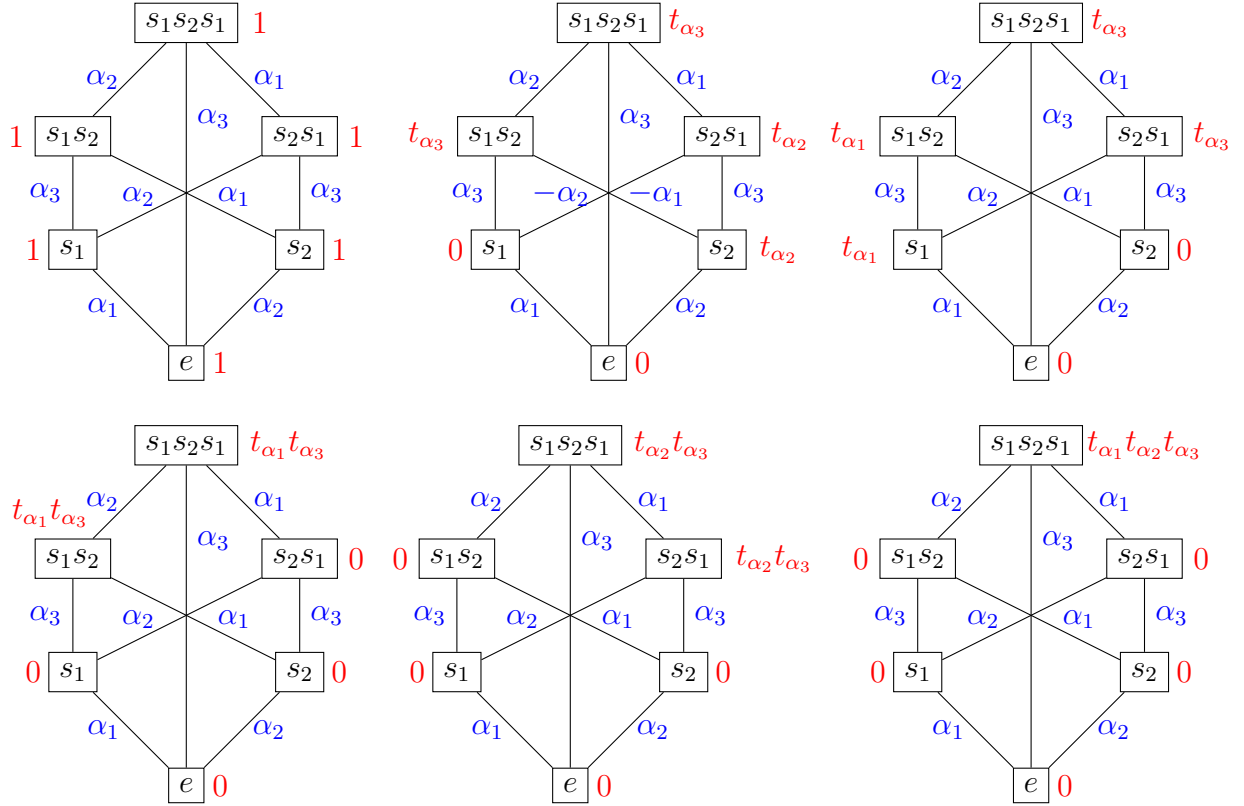
\begin{center}
\begin{tikzpicture}
\begin{scope}[scale=1.5]
    \draw  (0,0) node[draw] (v0) {$e$};
    \draw  (1,1) node[draw]  (v1) {$s_2$};
    \draw  (-1,1) node[draw] (v2) {$s_1$};
    \draw  (-1,2) node[draw] (v3) {$s_1s_2$};
    \draw  (1,2) node[draw] (v4) {$s_2s_1$};
    \draw (0,3)  node[draw]  (v5) {$s_1s_2s_1$};

    \draw (0.3,0)  node  (l0) {\textcolor{red}{$1$}};
    \draw  (1.35,1) node  (l1) {\textcolor{red}{$1$}};
    \draw  (-1.35,1) node (l2) {\textcolor{red}{$1$}};
    \draw  (-1.5,2) node (l3) {\textcolor{red}{$1$}};
    \draw  (1.5,2) node  (l4) {\textcolor{red}{$1$}};
    \draw (0.65,3)  node  (l5) {\textcolor{red}{$1$}};

    
    \draw (v0)--(v1)node[midway,right] {\textcolor{blue}{ $\alpha_2$}};
    \draw (v0)--(v2)node[midway,left] {\textcolor{blue}{ $\alpha_1$}};
    \draw (v1)--(v4)node[midway,right] {\textcolor{blue}{ $\alpha_3$}};
    \draw (v2)--(v3)node[midway,left] {\textcolor{blue}{ $\alpha_3$}};
    \draw (v3)--(v5)node[midway,left] {\textcolor{blue}{ $\alpha_2$}};
    \draw (v4)--(v5)node[midway,right] {\textcolor{blue}{ $\alpha_1$}};
    \draw (v0)--(v5)node[near end,right] {\textcolor{blue}{ $\alpha_3$}};
    \draw (v1)--(v3)node[near start,above] {\textcolor{blue}{ $\alpha_1$}};
    \draw (v2)--(v4)node[near start,above] {\textcolor{blue}{ $\alpha_2$}};
    
\end{scope}
\begin{scope}[xshift=2.2in, scale=1.5]
    \draw  (0,0) node[draw] (v0) {$e$};
    \draw  (1,1) node[draw]  (v1) {$s_2$};
    \draw  (-1,1) node[draw] (v2) {$s_1$};
    \draw  (-1,2) node[draw] (v3) {$s_1s_2$};
    \draw  (1,2) node[draw] (v4) {$s_2s_1$};
    \draw (0,3)  node[draw]  (v5) {$s_1s_2s_1$};

    \draw (0.3,0)  node  (l0)  {\textcolor{red}{$0$}};
    \draw  (1.5,1) node  (l1) {\textcolor{red}{$t_{\alpha_2}$}};
    \draw  (-1.35,1) node (l2) {\textcolor{red}{$0$}};
    \draw  (-1.6,2) node (l3)  {\textcolor{red}{$t_{\alpha_3}$}};
    \draw  (1.6,2) node  (l4)  {\textcolor{red}{$t_{\alpha_2}$}};
    \draw (0.7,3)  node  (l5) {\textcolor{red}{$t_{\alpha_3}$}};
    
    \draw (v0)--(v1)node[midway,right] {\textcolor{blue}{ $\alpha_2$}};
    \draw (v0)--(v2)node[midway,left] {\textcolor{blue}{ $\alpha_1$}};
    \draw (v1)--(v4)node[midway,right] {\textcolor{blue}{ $\alpha_3$}};
    \draw (v2)--(v3)node[midway,left] {\textcolor{blue}{ $\alpha_3$}};
    \draw (v3)--(v5)node[midway,left] {\textcolor{blue}{ $\alpha_2$}};
    \draw (v4)--(v5)node[midway,right] {\textcolor{blue}{ $\alpha_1$}};
    \draw (v0)--(v5)node[near end,right] {\textcolor{blue}{ $\alpha_3$}};
    \draw (v1)--(v3)node[near start,above] {\textcolor{blue}{ $-\alpha_1$}};
    \draw (v2)--(v4)node[near start,above] {\textcolor{blue}{ $-\alpha_2$}};
    
\end{scope}
\begin{scope}[xshift=4.4in, scale=1.5]
    \draw  (0,0) node[draw] (v0) {$e$};
    \draw  (1,1) node[draw]  (v1) {$s_2$};
    \draw  (-1,1) node[draw] (v2) {$s_1$};
    \draw  (-1,2) node[draw] (v3) {$s_1s_2$};
    \draw  (1,2) node[draw] (v4) {$s_2s_1$};
    \draw (0,3)  node[draw]  (v5) {$s_1s_2s_1$};

    \draw (0.3,0)  node  (l0)  {\textcolor{red}{$0$}};
    \draw  (1.35,1) node  (l1) {\textcolor{red}{$0$}};
    \draw  (-1.5,1) node (l2) {\textcolor{red}{$t_{\alpha_1}$}};
    \draw  (-1.6,2) node (l3)  {\textcolor{red}{$t_{\alpha_1}$}};
    \draw  (1.6,2) node  (l4)  {\textcolor{red}{$t_{\alpha_3}$}};
    \draw (0.7,3)  node  (l5) {\textcolor{red}{$t_{\alpha_3}$}};
    
    \draw (v0)--(v1)node[midway,right] {\textcolor{blue}{ $\alpha_2$}};
    \draw (v0)--(v2)node[midway,left] {\textcolor{blue}{ $\alpha_1$}};
    \draw (v1)--(v4)node[midway,right] {\textcolor{blue}{ $\alpha_3$}};
    \draw (v2)--(v3)node[midway,left] {\textcolor{blue}{ $\alpha_3$}};
    \draw (v3)--(v5)node[midway,left] {\textcolor{blue}{ $\alpha_2$}};
    \draw (v4)--(v5)node[midway,right] {\textcolor{blue}{ $\alpha_1$}};
    \draw (v0)--(v5)node[near end,right] {\textcolor{blue}{ $\alpha_3$}};
    \draw (v1)--(v3)node[near start,above] {\textcolor{blue}{ $\alpha_1$}};
    \draw (v2)--(v4)node[near start,above] {\textcolor{blue}{ $\alpha_2$}};
    
\end{scope}
\begin{scope}[yshift = -2.2in, xshift=0in, scale=1.5]
    \draw  (0,0) node[draw] (v0) {$e$};
    \draw  (1,1) node[draw]  (v1) {$s_2$};
    \draw  (-1,1) node[draw] (v2) {$s_1$};
    \draw  (-1,2) node[draw] (v3) {$s_1s_2$};
    \draw  (1,2) node[draw] (v4) {$s_2s_1$};
    \draw (0,3)  node[draw]  (v5) {$s_1s_2s_1$};

    \draw (0.3,0)  node  (l0)  {\textcolor{red}{$0$}};
    \draw  (1.35,1) node  (l1) {\textcolor{red}{$0$}};
    \draw  (-1.35,1) node (l2) {\textcolor{red}{$0$}};
    \draw  (-1.2,2.4) node (l3)  {\textcolor{red}{$t_{\alpha_1}t_{\alpha_3}$}};
    \draw  (1.5,2) node  (l4)  {\textcolor{red}{$0$}};
    \draw (0.9,3)  node  (l5) {\textcolor{red}{$t_{\alpha_1}t_{\alpha_3}$}};
    
    \draw (v0)--(v1)node[midway,right] {\textcolor{blue}{ $\alpha_2$}};
    \draw (v0)--(v2)node[midway,left] {\textcolor{blue}{ $\alpha_1$}};
    \draw (v1)--(v4)node[midway,right] {\textcolor{blue}{ $\alpha_3$}};
    \draw (v2)--(v3)node[midway,left] {\textcolor{blue}{ $\alpha_3$}};
    \draw (v3)--(v5)node[midway,left] {\textcolor{blue}{ $\alpha_2$}};
    \draw (v4)--(v5)node[midway,right] {\textcolor{blue}{ $\alpha_1$}};
    \draw (v0)--(v5)node[near end,right] {\textcolor{blue}{ $\alpha_3$}};
    \draw (v1)--(v3)node[near start,above] {\textcolor{blue}{ $\alpha_1$}};
    \draw (v2)--(v4)node[near start,above] {\textcolor{blue}{ $\alpha_2$}};
    
\end{scope}
\begin{scope}[yshift = -2.2in, xshift=2in, scale=1.5]
    \draw  (0,0) node[draw] (v0) {$e$};
    \draw  (1,1) node[draw]  (v1) {$s_2$};
    \draw  (-1,1) node[draw] (v2) {$s_1$};
    \draw  (-1,2) node[draw] (v3) {$s_1s_2$};
    \draw  (1,2) node[draw] (v4) {$s_2s_1$};
    \draw (0,3)  node[draw]  (v5) {$s_1s_2s_1$};

    \draw (0.3,0)  node  (l0)  {\textcolor{red}{$0$}};
    \draw  (1.35,1) node  (l1) {\textcolor{red}{$0$}};
    \draw  (-1.35,1) node (l2) {\textcolor{red}{$0$}};
    \draw  (-1.5,2) node (l3)  {\textcolor{red}{$0$}};
    \draw  (1.8,2) node  (l4)  {\textcolor{red}{$t_{\alpha_2}t_{\alpha_3}$}};
    \draw (0.9,3)  node  (l5) {\textcolor{red}{$t_{\alpha_2}t_{\alpha_3}$}};
    
    \draw (v0)--(v1)node[midway,right] {\textcolor{blue}{ $\alpha_2$}};
    \draw (v0)--(v2)node[midway,left] {\textcolor{blue}{ $\alpha_1$}};
    \draw (v1)--(v4)node[midway,right] {\textcolor{blue}{ $\alpha_3$}};
    \draw (v2)--(v3)node[midway,left] {\textcolor{blue}{ $\alpha_3$}};
    \draw (v3)--(v5)node[midway,left] {\textcolor{blue}{ $\alpha_2$}};
    \draw (v4)--(v5)node[midway,right] {\textcolor{blue}{ $\alpha_1$}};
    \draw (v0)--(v5)node[near end,right] {\textcolor{blue}{ $\alpha_3$}};
    \draw (v1)--(v3)node[near start,above] {\textcolor{blue}{ $\alpha_1$}};
    \draw (v2)--(v4)node[near start,above] {\textcolor{blue}{ $\alpha_2$}};
    
\end{scope}
\begin{scope}[yshift = -2.2in, xshift=4.4in, scale=1.5]
    \draw  (0,0) node[draw] (v0) {$e$};
    \draw  (1,1) node[draw]  (v1) {$s_2$};
    \draw  (-1,1) node[draw] (v2) {$s_1$};
    \draw  (-1,2) node[draw] (v3) {$s_1s_2$};
    \draw  (1,2) node[draw] (v4) {$s_2s_1$};
    \draw (0,3)  node[draw]  (v5) {$s_1s_2s_1$};

    \draw (0.3,0)  node  (l0)  {\textcolor{red}{$0$}};
    \draw  (1.35,1) node  (l1) {\textcolor{red}{$0$}};
    \draw  (-1.35,1) node (l2) {\textcolor{red}{$0$}};
    \draw  (-1.5,2) node (l3)  {\textcolor{red}{$0$}};
    \draw  (1.5,2) node  (l4)  {\textcolor{red}{$0$}};
    \draw (1,3)  node  (l5) {\textcolor{red}{$t_{\alpha_1}t_{\alpha_2}t_{\alpha_3}$}};
    
    \draw (v0)--(v1)node[midway,right] {\textcolor{blue}{ $\alpha_2$}};
    \draw (v0)--(v2)node[midway,left] {\textcolor{blue}{ $\alpha_1$}};
    \draw (v1)--(v4)node[midway,right] {\textcolor{blue}{ $\alpha_3$}};
    \draw (v2)--(v3)node[midway,left] {\textcolor{blue}{ $\alpha_3$}};
    \draw (v3)--(v5)node[midway,left] {\textcolor{blue}{ $\alpha_2$}};
    \draw (v4)--(v5)node[midway,right] {\textcolor{blue}{ $\alpha_1$}};
    \draw (v0)--(v5)node[near end,right] {\textcolor{blue}{ $\alpha_3$}};
    \draw (v1)--(v3)node[near start,above] {\textcolor{blue}{ $\alpha_1$}};
    \draw (v2)--(v4)node[near start,above] {\textcolor{blue}{ $\alpha_2$}};
    
\end{scope}
\end{tikzpicture}
\end{center}
The following are some splines on $I_2(4)$ with $H = \{\alpha_1,\alpha_2\}$.
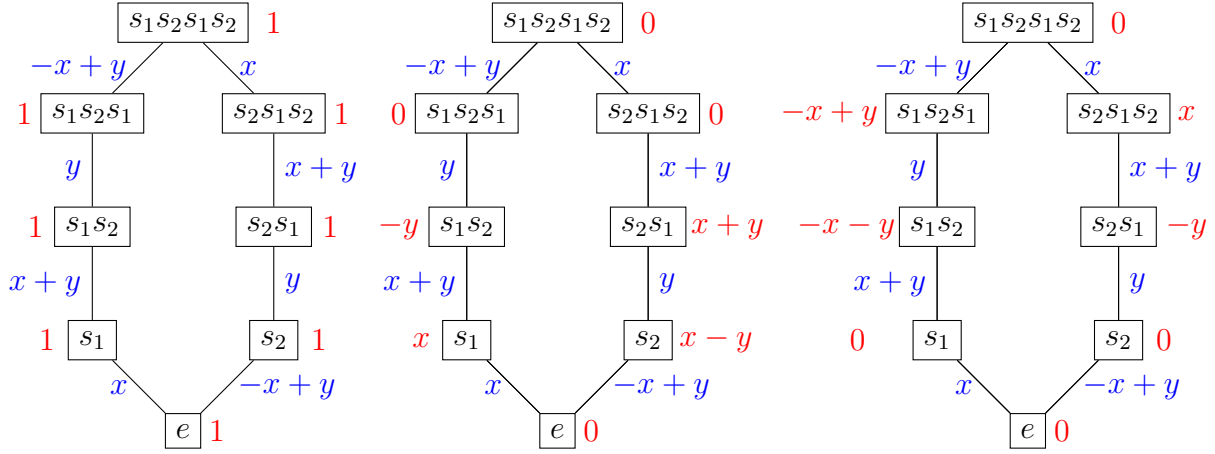
\begin{center}
\begin{tikzpicture}
\begin{scope}[xshift=0in, scale=1.5]
    \draw  (0,0.2) node[draw] (v0) {$e$};
    \draw  (0.8,1) node[draw]  (v1) {$s_2$};
    \draw  (-0.8,1) node[draw] (v2) {$s_1$};
    \draw  (-0.8,2) node[draw] (v3) {$s_1 s_2$};
    \draw  (0.8,2) node[draw] (v4) {$s_2 s_1$};
    \draw (0.8,3) node[draw] (v5) {$s_2 s_1 s_2$};
    \draw (-0.8,3) node[draw] (v6) {$s_1 s_2s_1$};
    \draw (0,3.8)  node[draw]  (v7) {$s_1 s_2s_1s_2$};

    \draw  (0.3,0.2) node   (l0) {\textcolor{red}{$1$}};
    \draw  (1.2,1) node   (l1) {\textcolor{red}{$1$}};
    \draw  (-1.2,1) node  (l2) {\textcolor{red}{$1$}};
    \draw  (-1.3,2) node  (l3) {\textcolor{red}{$1$}};
    \draw  (1.3,2) node   (l4) {\textcolor{red}{$1$}};
    \draw (1.4,3) node    (l5) {\textcolor{red}{$1$}};
    \draw (-1.4,3) node   (l6) {\textcolor{red}{$1$}};
    \draw (0.8,3.8)  node (l7) {\textcolor{red}{$1$}};
    

    \draw (v0)--(v1) node[midway,right] {\textcolor{blue}{$-x+y$}};
    \draw (v0)--(v2) node[midway,left] {\textcolor{blue}{$x$}};
    \draw (v1)--(v4)node[midway,right] {\textcolor{blue}{$y$}};
    \draw (v2)--(v3) node[midway,left] {\textcolor{blue}{$x+y$}};
    \draw (v3)--(v6) node[midway,left] {\textcolor{blue}{$y$}};
    \draw (v4)--(v5)node[midway,right] {\textcolor{blue}{$x+y$}};
    \draw (v5)--(v7)node[midway,right] {\textcolor{blue}{$x$}};
    \draw (v6)--(v7)node[midway,left] {\textcolor{blue}{$-x+y$}};
\end{scope}
\begin{scope}[xshift=1.95in, scale=1.5]
    \draw  (0,0.2) node[draw] (v0) {$e$};
    \draw  (0.8,1) node[draw]  (v1) {$s_2$};
    \draw  (-0.8,1) node[draw] (v2) {$s_1$};
    \draw  (-0.8,2) node[draw] (v3) {$s_1 s_2$};
    \draw  (0.8,2) node[draw] (v4) {$s_2 s_1$};
    \draw (0.8,3) node[draw] (v5) {$s_2 s_1 s_2$};
    \draw (-0.8,3) node[draw] (v6) {$s_1 s_2s_1$};
    \draw (0,3.8)  node[draw]  (v7) {$s_1 s_2s_1s_2$};

    \draw  (0.3,0.2) node (l0) {\textcolor{red}{$0$}};
    \draw  (1.4,1) node   (l1) {\textcolor{red}{$x-y$}};
    \draw  (-1.2,1) node  (l2) {\textcolor{red}{$x$}};
    \draw  (-1.4,2) node  (l3) {\textcolor{red}{$-y$}};
    \draw  (1.5,2) node   (l4) {\textcolor{red}{$x+y$}};
    \draw (1.4,3) node    (l5) {\textcolor{red}{$0$}};
    \draw (-1.4,3) node   (l6) {\textcolor{red}{$0$}};
    \draw (0.8,3.8)  node (l7) {\textcolor{red}{$0$}};
    
    \draw (v0)--(v1);
    \draw (v0)--(v2);
    \draw (v1)--(v4);
    \draw (v2)--(v3);
    \draw (v3)--(v6) ;
    \draw (v4)--(v5);
    \draw (v5)--(v7);
    \draw (v6)--(v7);
    \draw (v0)--(v1) node[midway,right] {\textcolor{blue}{$-x+y$}};
    \draw (v0)--(v2) node[midway,left] {\textcolor{blue}{$x$}};
    \draw (v1)--(v4)node[midway,right] {\textcolor{blue}{$y$}};
    \draw (v2)--(v3) node[midway,left] {\textcolor{blue}{$x+y$}};
    \draw (v3)--(v6) node[midway,left] {\textcolor{blue}{$y$}};
    \draw (v4)--(v5)node[midway,right] {\textcolor{blue}{$x+y$}};
    \draw (v5)--(v7)node[midway,right] {\textcolor{blue}{$x$}};
    \draw (v6)--(v7)node[midway,left] {\textcolor{blue}{$-x+y$}};
\end{scope}
\begin{scope}[xshift=4.4in, scale=1.5]
    \draw  (0,0.2) node[draw] (v0) {$e$};
    \draw  (0.8,1) node[draw]  (v1) {$s_2$};
    \draw  (-0.8,1) node[draw] (v2) {$s_1$};
    \draw  (-0.8,2) node[draw] (v3) {$s_1 s_2$};
    \draw  (0.8,2) node[draw] (v4) {$s_2 s_1$};
    \draw (0.8,3) node[draw] (v5) {$s_2 s_1 s_2$};
    \draw (-0.8,3) node[draw] (v6) {$s_1 s_2s_1$};
    \draw (0,3.8)  node[draw]  (v7) {$s_1 s_2s_1s_2$};

    \draw  (0.3,0.2) node (l0) {\textcolor{red}{$0$}};
    \draw  (1.2,1) node   (l1) {\textcolor{red}{$0$}};
    \draw  (-1.5,1) node  (l2) {\textcolor{red}{$0$}};
    \draw  (-1.6,2) node  (l3) {\textcolor{red}{$-x-y$}};
    \draw  (1.4,2) node   (l4) {\textcolor{red}{$-y$}};
    \draw (1.4,3) node    (l5) {\textcolor{red}{$x$}};
    \draw (-1.75,3) node   (l6) {\textcolor{red}{$-x+y$}};
    \draw (0.8,3.8)  node (l7) {\textcolor{red}{$0$}};
    
    \draw (v0)--(v1);
    \draw (v0)--(v2);
    \draw (v1)--(v4);
    \draw (v2)--(v3);
    \draw (v3)--(v6) ;
    \draw (v4)--(v5);
    \draw (v5)--(v7);
    \draw (v6)--(v7);
    \draw (v0)--(v1) node[midway,right] {\textcolor{blue}{$-x+y$}};
    \draw (v0)--(v2) node[midway,left] {\textcolor{blue}{$x$}};
    \draw (v1)--(v4)node[midway,right] {\textcolor{blue}{$y$}};
    \draw (v2)--(v3) node[midway,left] {\textcolor{blue}{$x+y$}};
    \draw (v3)--(v6) node[midway,left] {\textcolor{blue}{$y$}};
    \draw (v4)--(v5)node[midway,right] {\textcolor{blue}{$x+y$}};
    \draw (v5)--(v7)node[midway,right] {\textcolor{blue}{$x$}};
    \draw (v6)--(v7)node[midway,left] {\textcolor{blue}{$-x+y$}};
\end{scope}
\end{tikzpicture}
\end{center}
Note that $\cl(w,ws_\alpha) = t_\beta$ precisely when $s_\beta w = ws_\alpha$.

\begin{lemma}
    The set of splines is a ring with pointwise multiplication and addition.
\end{lemma}
\begin{proof}
    We may simply check this for two $\spl{\rho},\spl{\sigma} \in \Ms$ at an arbitrary pair $(w,ws_\alpha)$. For $\spl{\rho} + \spl{\sigma}$, this is straightforward. As for the product $\spl{\rho}\spl{\sigma}$, 
    \begin{align*}
        \spl{\rho}(w)\spl{\sigma}(w)- \spl{\rho}(ws_\alpha)\spl{\sigma}(ws_\alpha)&= \spl{\rho}(w)\spl{\sigma}(w) - \spl{\rho}(w)\spl{\sigma}(ws_\alpha) + \spl{\rho}(w)\spl{\sigma}(ws_\alpha) - \spl{\rho}(ws_\alpha)\spl{\sigma}(ws_\alpha) \\
        &= \spl{\rho}(w)\left(\spl{\sigma}(w) - \spl{\sigma}(ws_\alpha)\right) +\spl{\sigma}(ws_\alpha)\left(\spl{\rho}(w)   - \spl{\rho}(ws_\alpha)\right),
    \end{align*}
    and since $\spl{\rho},\spl{\sigma} \in \Ms$, this lies in $\llangle t_{w(\alpha)} \rrangle$.
\end{proof}
\begin{lemma}
    The ring $\Ms$ is graded by degree.
\end{lemma}
\begin{proof}
    Each label given by $\cl$ is homogeneous of degree $1$. Thus each polynomial ideal in the definition of $\splines{H}$ is homogeneous, and so containment in each polynomial ideal is given individually by each graded piece.
\end{proof}
Some functions from $I_2(n)$ to $\C[x,y]$ are elements of $\Ms$ for all $H$. Let 
\begin{center} 
\begin{tabular}{lll}
    $\spl{\one} \colon I_2(n) \to \poly$ &be $\spl{\one}(w)\coloneqq 1$ & for all $w \in I_2(n)$,\\[3pt]
    $\spl{x}_l \colon I_2(n) \to \C[x,y]$ &be $\spl{x}_l(w) \coloneqq x$ & for all $w\in I_2(n)$,\\[3pt]
    $\spl{y}_l \colon I_2(n) \to \C[x,y]$ &be $\spl{y}_l(w) \coloneqq y$ & for all $w\in I_2(n)$,\\[3pt]
    $\spl{x}_r \colon I_2(n) \to \C[x,y]$ &be $\spl{x}_r(w) \coloneqq w(x)$ & for all $w\in I_2(n)$, and\\[3pt]
    $\spl{y}_r \colon I_2(n) \to \C[x,y]$ &be $\spl{y}_r(w) \coloneqq w(y)$ & for all $w\in I_2(n)$.
\end{tabular}
\end{center}
Checking that $\spl{\one}$, $\spl{x}_{l}$, and $\spl{y}_l$ are splines is trivial, as they are each examples of constant splines. 
As for $\spl{x}_r$ and $\spl{y}_r$, we can check that \[\spl{x}_r(w) - \spl{x}_r(ws_\alpha) = w(x)-ws_\alpha(x) = w(x-s_\alpha(x)).\]
Since $t_\alpha$ divides $x-s_\alpha(x)$ by Lemma~\ref{lem:reflect_divisible}, $t_{w(\alpha)}$ divides $w(x-s_\alpha(x))$. The same occurs for $\spl{y}_r$, and in fact these are special cases of the spline which evaluates to $w(p)$ for a fixed $p\in \mathbb{C}[x,y]$.\par  
The ring $\Ms$ is an infinite-dimensional $\C$-vector space in the natural way, and can also be viewed as a finitely generated graded $\poly$-module in two ways via the following module actions:
\begin{equation}\label{eqn:left_action}
    f\left(x,y\right).\spl{\rho} = f\left(\spl{x}_l,\spl{y}_l\right)\spl{\rho}
\end{equation}
and
\begin{equation}\label{eqn:right_action}
    f\left(x,y\right).\spl{\rho} = f\left(\spl{x}_r,\spl{y}_r\right)\spl{\rho},
\end{equation}
where the right-hand side of both (\ref{eqn:left_action}) and (\ref{eqn:right_action}) work by substituting splines for variables into the polynomial $f$, and  then multiplying as in the ring structure of $\Ms$. For both actions of the polynomial $f(x,y)$, the constant $c = f(0,0)$ is naturally mapped to $c\spl{\one}$. Since $\Ms$ is a $\poly$-submodule of the free module $\prod_{w \in I_2(n)} \poly$ for either module action, it is finitely generated. We call the module action (\ref{eqn:left_action}) the \emph{left action} and the module action (\ref{eqn:right_action}) the \emph{right action} of $\poly$ on $\Ms$. Both the left and right actions are naturally compatible with the grading on $\Ms$. \par
\begin{lemma}
    Any minimal homogeneous (module) generating set for $\Ms$ has the same number of elements in each degree.
\end{lemma}
\begin{proof}
    This is a straightforward consequence of the graded Nakayama Lemma \cite{eisenbud}. One may also use the quotient defined at the end of the section to prove this directly, as in \cite[Lemma 3.1]{lesnevich2025splines}. Additionally, we will show later that this module is a free module, from which this also follows directly. 
\end{proof}
So even if $\Ms$ is not a free module over $\C[x,y]$, we can speak meaningfully about its rank-generating function. That being said, by Theorem \ref{thm:flowup} below, $\Ms$ is in fact a free module for all $H$, and has a very particular type of basis.  
\begin{definition}\label{def:upper_triangular}
    Given a labeled graph $\cg$ endowed with a partial order $\preceq$ on its vertices, a spline $\spl{\rho}$ on $\cg$ is \emph{upper triangular} from $w \in \cg$ if $w$ is in the support of $\spl{\rho}$ and $\spl{\rho}(v) = 0$ whenever $w \not\preceq v$. In other words, $\spl{\rho}$ is supported at $w$ and only elements above $w$.

    A set $\{\spl{\rho}_w \mid w \in \cg\}$ of splines on $\cg$ is \emph{poset upper-triangular} if each $\spl{\rho}_w$ is upper triangular from $w$.
\end{definition}
\begin{remark}
    If one takes a linear extension of the poset and writes each spline $\spl{\rho}_w$ in a poset upper-triangular set as the rows in a matrix ordered by that linear extension, then the resulting matrix is an upper-triangular matrix.
\end{remark}
The following is a well-known property for splines in several settings, and is a consequence of Proposition 2.6 in \cite{Harada_Tymoczko_pinball}. We include a proof for completeness.
\begin{lemma}\label{lem:upper_triangular}
    Let $H \subseteq \Phi^+$ be a Hessenberg space. If $R = \{\spl{\rho}_w \mid w \in I_2(n)\}$ is 
    \begin{enumerate}
        \item poset upper-triangular with respect to Bruhat order, and 
        \item $\deg(\spl{\rho}_w) = \hlen(w)$ for all $w \in I_2(n)$,
    \end{enumerate}
    then for both the left and right actions, $R$ is a minimal module generating set for $\Ms$. Moreover, $\Ms$ is free with $R$ as a basis, and the rank polynomial is $\ds \sum_{w \in I_2(n)} q^{\hlen(w)}$.
\end{lemma}
\begin{proof}
First we prove linear independence. Suppose
$\sum_w f_w\spl{\rho}_w=0$, with $f_w\in\poly$, and choose $w$
minimal in Bruhat order among those with $f_w\neq 0$. Evaluating at $w$,
every term except $f_w\spl{\rho}_w(w)$ vanishes. Indeed, if
$v< w$, then $f_v=0$ by minimality, while if $v\nleq w$, then
$\spl{\rho}_v(w)=0$ by assumption. Thus
$f_w\spl{\rho}_w(w)=0$. Since $\poly$ is a domain and $\spl{\rho}_w(w)\neq 0$, this forces $f_w=0$, a contradiction.
For the right action, evaluation at $w$ instead gives $w(f_w)\spl{\rho}_w(w)=0$; since $w$ acts on $\poly$ by an automorphism, this likewise forces $f_w=0$.

Now let $\spl{\rho}\in\Ms$. We express $\spl{\rho}$ as a
$\poly$-linear combination of the $\spl{\rho}_w$ by induction along a
linear extension of Bruhat order, from smaller elements to larger elements.
At the step indexed by $w$, assume that the current residual spline vanishes
at every $v< w$. For each $\gamma\in H$, the graph $\cg_H$ contains the edge $(w, w s_\gamma)$ with label
$t_{w(\gamma)}$. If $ws_\gamma < w$, the spline condition gives
$t_{w(\gamma)}\mid \spl{\rho}(w)$. The forms $t_{w(\gamma)}$, for
$\gamma \in H$, are pairwise relatively prime, so
$\prod_{\gamma \in \hinv(w)}t_{w(\gamma)}$ divides $\spl{\rho}(w)$. The same argument applies to $\spl{\rho}_w(w)$, and since this is also the same degree as $\spl{\rho}_w$, this product is a constant multiple of
$\spl{\rho}_w(w)$. Hence $\spl{\rho}_w(w)$ divides
$\spl{\rho}(w)$, so $g_w=\spl{\rho}(w)/\spl{\rho}_w(w)$ lies in $\poly$.

Replace the current residual spline $\spl{\rho}$ by
$\spl{\rho}-g_w\spl{\rho}_w$ for the left action or $\spl{\rho}-w^{-1}(g_w)\spl{\rho}_w$ for the right action. The new residual vanishes at $w$. It also
still vanishes at every $v< w$, because
$\spl{\rho}_w(v)=0$ unless $w\leq v$. Continuing through the finite
poset expresses the original spline as a $\poly$-linear combination of $R$.

The rank-polynomial calculation follows directly from the degree of each spline in $R$.
\end{proof}
\begin{definition}\label{def:flowup}
    A basis of splines satisfying the conditions in Lemma \ref{lem:upper_triangular} is called a \emph{flow-up basis}. The elements of such a basis are called \emph{flow-up splines}. If $w$ is the minimal element in the support of a flow-up spline, then we call it a flow-up spline at $w$.
\end{definition}
Furthermore, $\Ms$ has the structure of an $I_2(n)$-module. This was defined for the symmetric group by Tymoczko in \cite{tymoczko2008permutation}.
\begin{definition}
    Let $\spl{\rho} \in \Ms$. The \emph{dot action} of $I_2(n)$ on $\Ms$ is given by 
    \[
    (w\cdot \spl{\rho})(v) \coloneqq w(\spl{\rho}(w^{-1}v))
    \]
    for $w,v \in I_2(n)$.
\end{definition}
\begin{lemma}\label{lem:dot_action_preserves_splines} If $\spl{\rho}$ is a spline in $\Ms$ then $w \cdot \spl{\rho}$ is as well. In other words, the dot action is well-defined.
\end{lemma}
\begin{proof}
    For all $(v,vs_\alpha) \in \ce_H$, 
\begin{align*}
(w \cdot \spl{\rho})(v) - (w \cdot \spl{\rho})(vs_\alpha) &= w\spl{\rho}(w^{-1}v)- w\spl{\rho}(w^{-1}v s_\alpha) \\
&= w(\spl{\rho}(w^{-1}v)- \spl{\rho}(w^{-1} vs_\alpha)) \\
&\in w\llangle t_{w^{-1}v(\alpha)} \rrangle \\
&= \llangle t_{v(\alpha)}\rrangle \qedhere.
\end{align*}
\end{proof}

Finally, consider the irrelevant ideal $J = \llangle x,y \rrangle \subset \C[x,y]$. The submodule $J\Ms$ is homogeneous and $I_2(n)$-equivariant. So 
\begin{equation}\label{eqndef:left_quotient}
    \lqot{H} = \bigoplus_{i\geq 0} \left(\lqot{H}\right)_i \coloneqq \faktor{\splines{H}}{J\Ms} = \faktor{\Ms}{\llangle \spl{x}_l,\spl{y}_l\rrangle}
\end{equation}
with the left action and
\begin{equation}\label{eqndef:right_quotient}
    \rqot{H} = \bigoplus_{i\geq 0} \left(\rqot{H}\right)_i \coloneqq \faktor{\splines{H}}{J \Ms} = \faktor{\Ms}{\llangle \spl{x}_r,\spl{y}_r\rrangle}
\end{equation}
with the right action are each graded $\C[x,y]/J \cong \C$-vector spaces. Since $J\Ms$ is $I_2(n)$-equivariant, each of these vector spaces carries an $I_2(n)$ representation given by $w.[\spl{\rho}] = [w \cdot \spl{\rho}]$.
This representation is different with respect to the left and right module action. This is relatively easy to see in an example, as $\spl{x}_l \sim 0$ in the quotient under the left action but not the right, but $\spl{x}_r \sim 0$ in the quotient under the right action but not the left. The dot action fixes $\spl{x}_r$ but does not fix $\spl{x}_l$, and so the left and right quotients interact differently with the dot action.

\section{The Schubert case}\label{sec:schubert}

In this section we specialize to the full reflection graph, so $H=\Phi^+$.
The goal is to construct a Schubert basis for $\splines{\Phi^+}$, prove a
Borel presentation for this spline module, and then use that presentation to
identify the left and right quotient representations. We define the Schubert
splines using divided differences, prove their flow-up triangularity directly
from the rank-two Bruhat order, and then derive the basis theorem by
triangular elimination.

The defining spline condition was stated using right reflection edges. For
the Schubert construction it is more convenient to use left reflection edges.
The following lemma records that, in the full graph, these two descriptions are
equivalent; the only difference is a possible sign in the edge label, which
does not change the corresponding principal ideal.

\begin{lemma}\label{lem:left_edges_full_graph}
For purposes of the spline conditions, the full reflection graph may be read
using left reflection edges $w\leftrightarrow s_\alpha w$, with label
$t_\alpha$. Equivalently,
\begin{equation}\label{eq:left_spline_condition_full}
    \splines{\Phi^+}
    =
    \setof{
        \spl{\rho}\colon I_2(n)\to\poly
        \st
        \spl{\rho}(w)-\spl{\rho}(s_\alpha w)
            \in \llangle t_\alpha\rrangle
        \text{ for all } w\in I_2(n),\ \alpha\in\Phi^+
    }.
\end{equation}
\end{lemma}

\begin{proof}
The right-hand side of \eqref{eq:left_spline_condition_full} uses left
reflection edges. We compare these with the right reflection edges in the
definition of $\splines{\Phi^+}$.

Fix $w\in I_2(n)$ and $\alpha\in\Phi^+$. Choose the positive root
$\beta$ such that $w(\beta)=\pm\alpha$. Since reflections depend only on
the root line, $s_\alpha w=ws_\beta$. Thus the left edge
$w\leftrightarrow s_\alpha w$ is the same unoriented edge as the right edge
$w\leftrightarrow ws_\beta$. The label on this right edge is
$t_{w(\beta)}=\pm t_\alpha$, so it generates the same
ideal as $t_\alpha$. Hence the
right-edge spline condition for $w\leftrightarrow ws_\beta$ is exactly the
left-edge condition
$\spl{\rho}(w)-\spl{\rho}(s_\alpha w)\in\llangle t_\alpha\rrangle$.

The same argument also converts every right reflection edge into a left
reflection edge with the same label up to sign. Therefore the two collections of
divisibility conditions are identical.
\end{proof}

\begin{lemma}\label{lem:divided_difference_spline}
Let $\spl{\rho}\in\splines{\Phi^+}$ and let $\alpha\in\Phi^+$. Define
\begin{equation}\label{eq:divided_difference_spline}
    \partial_\alpha\spl{\rho}(v)
        \deftobe
        \frac{\spl{\rho}(v)-s_\alpha\bigl(\spl{\rho}(s_\alpha v)\bigr)}
             {t_\alpha}.
\end{equation}
Then $\partial_\alpha\spl{\rho}\in\splines{\Phi^+}$.
\end{lemma}

\begin{proof}
    This operation is called the \emph{divided difference} operator and is well known to preserve $\splines{\Phi^+}$ for $I_2(n)$ and other finite Coxeter groups (including Weyl groups). We prove a slightly more general version of this in Lemma \ref{lem:generalized_divided_difference}.
\end{proof}

We now define the Schubert splines. Set
$\spl{\sigma}_e=\spl{\one}$. At the top of Bruhat order, set
\begin{equation}\label{eq:top_schubert_spline}
    \spl{\sigma}_{w_0}(v)
        \deftobe
        \begin{dcases*}
            \prod_{\alpha\in\Phi^+} t_\alpha, & if \(v=w_0\),\\
            0, & otherwise.
        \end{dcases*}
\end{equation}
This is a spline: all edge differences away from $w_0$ are zero, and every
edge incident to $w_0$ has label $t_\alpha$ for some $\alpha\in\Phi^+$,
which divides the product in \eqref{eq:top_schubert_spline}.

For $w\notin\{e,w_0\}$, let $\alpha\in\{\alpha_1,\alpha_2\}$ be the unique
simple root such that $\ell(s_\alpha w)=\ell(w)+1$, whose existence and
uniqueness are given by Lemma~\ref{lem:rank_two_bruhat_facts}(1), and define
\begin{equation}\label{eq:schubert_recursion}
    \spl{\sigma}_w
        \deftobe
        \partial_\alpha \bigl(\spl{\sigma}_{s_\alpha w}\bigr) .
\end{equation}
This is a downward recursion from $w_0$ for all elements except $e$. The
identity class is defined separately as the constant spline, so no
well-definedness issue arises at the bottom rank. By
Lemma~\ref{lem:divided_difference_spline}, every $\spl{\sigma}_w$ is a
spline. The operator in \eqref{eq:divided_difference_spline} lowers polynomial
degree by one, so the recursion is compatible with the expected grading; the
next proposition proves the nonvanishing and triangularity needed for the basis
theorem.

\begin{proposition}\label{prop:schubert_flowup}
For every $w\in I_2(n)$, the spline $\spl{\sigma}_w$ is homogeneous of
degree $\ell(w)$. It is supported on the Bruhat upper order ideal of $w$,
and its diagonal value is
\begin{equation}\label{eq:schubert_diagonal_value}
    \spl{\sigma}_w(w)=\prod_{\alpha\in N_L(w)}t_\alpha .
\end{equation}
Equivalently, $\spl{\sigma}_w(v)=0$ unless $w\leq v$.
\end{proposition}

\begin{proof}
The case $w=e$ is immediate. For the remaining elements, we prove the result
by downward induction on $\ell(w)$. The base case $w=w_0$ is immediate:
the support and degree follow from \eqref{eq:top_schubert_spline}, and
$N_L(w_0)=\Phi^+$.

Now let $w\notin\{e,w_0\}$, and let $\alpha\in\{\alpha_1,\alpha_2\}$ be
the unique simple root such that $\ell(s_\alpha w)=\ell(w)+1$, as in
Lemma~\ref{lem:rank_two_bruhat_facts}(1). The recursion gives
$\spl{\sigma}_w=\partial_\alpha(\spl{\sigma}_{s_\alpha w})$, so homogeneity follows
from the induction hypothesis because the divided-difference operator lowers
polynomial degree by one.

We next prove the support statement. Let $v\in I_2(n)$ and suppose
$w\nleq v$. Since $0<\ell(w)<n$, Lemma~\ref{lem:rank_two_bruhat_facts}(2)
shows that either $\ell(v)<\ell(w)$, or $\ell(v)=\ell(w)$ and $v\neq w$.

If $\ell(v)<\ell(w)$, then $s_\alpha w\nleq v$ and
$s_\alpha w\nleq s_\alpha v$, since
$\ell(s_\alpha v)\leq \ell(w)<\ell(s_\alpha w)$. If
$\ell(v)=\ell(w)$ and $v\neq w$, then again $s_\alpha w\nleq v$. Also $s_\alpha w\nleq s_\alpha v$. Indeed, if
$\ell(s_\alpha v)<\ell(s_\alpha w)$, this follows because Bruhat order is
graded by length. If $\ell(s_\alpha v)=\ell(s_\alpha w)$, then any
comparability would force equality, and $s_\alpha v=s_\alpha w$ would force
$v=w$, contrary to the case assumption. Thus in all cases $s_\alpha w\nleq v$ and
$s_\alpha w\nleq s_\alpha v$. Since
$\ell(s_\alpha w)=\ell(w)+1$, the induction hypothesis applies to
$\spl{\sigma}_{s_\alpha w}$, and gives
$\spl{\sigma}_{s_\alpha w}(v)=
\spl{\sigma}_{s_\alpha w}(s_\alpha v)=0$. Since
$\spl{\sigma}_w=\partial_\alpha(\spl{\sigma}_{s_\alpha w})$, the definition of the
divided-difference operator in \eqref{eq:divided_difference_spline} gives
$\spl{\sigma}_w(v)=0$.

It remains to compute the diagonal value. Since $s_\alpha w\nleq w$, the
support statement for $\spl{\sigma}_{s_\alpha w}$ gives
$\spl{\sigma}_{s_\alpha w}(w)=0$. Hence
\[
    \spl{\sigma}_w(w)
    =
    -\frac{s_\alpha(\spl{\sigma}_{s_\alpha w}(s_\alpha w))}{t_\alpha}.
\]
By the induction hypothesis and Lemma \ref{lem:rank_two_bruhat_facts}(3),
\[
    \spl{\sigma}_{s_\alpha w}(s_\alpha w)
        =
    t_\alpha\prod_{\beta\in N_L(w)}t_{s_\alpha\beta}.
\]
Applying $s_\alpha$ gives
$-t_\alpha\prod_{\beta\in N_L(w)}t_\beta$, and division by $-t_\alpha$
proves \eqref{eq:schubert_diagonal_value}.
\end{proof}
The Schubert splines for $I_2(n)$ have a particularly nice formula that, for geometric cases, is a consequence of a geometric condition called \emph{rational smoothness}. This is distinct from Billey's formula for Schubert splines, which is a sum over reduced subwords of a fixed reduced expression \cite{billey1999schub_formula}. The following is easily obtained by the same logic that appears in \cite{Carrell_Pete,Richmond_Slofstra,Elias_Williamson} in much greater generality, and agrees with the geometric cases in full Kac-Moody generality \cite{Kumar_KacMoody}, but could also be proven using our Corollary  \ref{cor:rationally_smooth} and has an even easier direct proof, which we include here.

\begin{theorem}\label{thm:schubert_rational}
    The Schubert splines are given explicitly by 
    \[
    \spl{\sigma}_w(v) = c^w(v) \prod_{\substack{s_\gamma v < v \\ s_\gamma v \not\geq w}} t_\gamma
    \]
    where $c^w(v) \in \C^*$ if $w \leq v$ and $c^w(v) = 0$ otherwise.
\end{theorem}

\begin{proof}
    The cases $w=e$ and $w=w_0$ follow from the nonzero constant spline and the top Schubert spline, respectively. Assume $0<\ell(w)<n$. By Proposition
    \ref{prop:schubert_flowup}, $\spl{\sigma}_w$ has degree $\ell(w)$, its support is contained in the upper order ideal generated by $w$, and $\spl{\sigma}_w(w)\neq 0$. In the strong Bruhat graph $\cg_{\Phi^+}$, the pair $(w,v)$ is an edge whenever $\abs{\ell(w) - \ell(v)}$ is an odd number (all odd-length elements of $I_2(n)$ are reflections). The upper order ideal $\{v\mid v\geq w\}$ contains $n-\ell(w)$ vertices of each length parity. Thus each vertex in this ideal has $n-\ell(w)$ neighbors within
    the ideal, out of $n$ neighbors in the full reflection graph, leaving exactly $\ell(w)$ edges to vertices outside the ideal.
    
    Thus, the set $\{\gamma \in \Phi^+ \mid vs_\gamma < v,\; vs_\gamma \not\geq w\}$ is size $\ell(w)$ for all $v \geq w$. Since $\spl{\sigma}_w(v') = 0$ for all $v' \not \geq w$, it follows from the divisibility conditions that, whenever $v \geq w$,
    \[
    \prod_{\substack{s_\gamma v < v\\
                    s_\gamma v \not\geq w}}
    t_\gamma
    \quad\text{divides}\quad
    \spl{\sigma}_w(v).
    \]
    Since this product has degree $\ell(w)$ and $\spl{\sigma}_w$ is homogeneous of degree $\ell(w)$, the stated formula holds with $c^w(v)\in\C$.
    It remains to show that $c^w(v)\neq 0$ whenever $v\geq w$.
    
    The linear forms corresponding to distinct positive roots are pairwise non-proportional. If $v,s_\zeta v\geq w$ and $\spl{\sigma}_w(v)\neq 0$,
    then the edge joining these vertices lies within the upper order ideal. Therefore $t_\zeta$ is not proportional to any factor in either of the products
    \[
    \prod_{\substack{s_\gamma v < v \\ s_\gamma v \not\geq w}} t_\gamma \qquad \text{or}\qquad \prod_{\substack{s_\gamma s_\zeta v < s_\zeta v \\ s_\gamma s_\zeta v \not\geq w}} t_\gamma.
    \]
    Thus $t_\zeta$ does not divide $\spl{\sigma}_w(v)$. If $\spl{\sigma}_w(s_\zeta v)=0$, the spline condition would force $t_\zeta\mid\spl{\sigma}_w(v)$, a contradiction. Hence $c^w(s_\zeta v)\neq 0$ as well. Starting from $\spl{\sigma}_w(w)\neq 0$ and using connectedness of the upper order ideal gives $c^w(v)\neq 0$ for every $v\geq w$, proving the claim.
\end{proof}
\begin{proposition}\label{prop:schubert_basis}
The Schubert splines form a homogeneous free $\poly$-basis of
$\splines{\Phi^+}$, namely
\[
    \splines{\Phi^+}
    =
    \bigoplus_{w\in I_2(n)}\poly \spl{\sigma}_w ,
    \qquad
    \deg\spl{\sigma}_w=\ell(w).
\]
\end{proposition}
\begin{proof}
This follows immediately from Proposition~\ref{prop:schubert_flowup} and Lemma~\ref{lem:upper_triangular}.
\end{proof}

\begin{corollary}\label{cor:schubert_hilbert_series}
The Hilbert series of the full reflection spline module is
\[
    \Hilb(\splines{\Phi^+};q)
    =
    \frac{\sum_{w\in I_2(n)}q^{\ell(w)}}{(1-q)^2}
    =
    \frac{(1+q)(1+q+\cdots+q^{n-1})}{(1-q)^2}.
\]
\end{corollary}
\begin{proof}
By Proposition~\ref{prop:schubert_basis}, $\splines{\Phi^+}$ is a free
$\poly$-module with one homogeneous basis element of degree $\ell(w)$ for
each $w\in I_2(n)$. Since $\Hilb(\poly;q)=1/(1-q)^2$, this
gives
\[
    \Hilb(\splines{\Phi^+};q)
    =
    \frac{\sum_{w\in I_2(n)}q^{\ell(w)}}{(1-q)^2}.
\]
In $I_2(n)$, there is one element of length $0$, one element of length
$n$, and two elements of each length $1,\ldots,n-1$. Therefore
$\sum_{w\in I_2(n)}q^{\ell(w)}
=(1+q)(1+q+\cdots+q^{n-1})$.
\end{proof}

We call the basis from Proposition~\ref{prop:schubert_basis} the
\emph{Schubert basis}. The next result identifies the full reflection spline
module with the usual Borel presentation.

Recall that $\spl{x}_l(w)=x$ and $\spl{y}_l(w)=y$, while
$\spl{x}_r(w)=w(x)$ and $\spl{y}_r(w)=w(y)$.

\begin{lemma}\label{lem:borel_presentation_schubert_splines}
The map
\begin{equation}\label{eq:borel_presentation_schubert_splines}
    \Psi\colon \poly\otimes_{\poly^{I_2(n)}}\poly
        \longrightarrow
    \splines{\Phi^+},
    \qquad
    \Psi(f\otimes h)(w)=f\,w(h),
\end{equation}
is an isomorphism of graded $\poly$-algebras, where $\poly$ acts on the
source through the first tensor factor and on $\splines{\Phi^+}$ by the left
module structure. Under this isomorphism, $x\otimes 1$ and $y\otimes 1$
correspond to $\spl{x}_l$ and $\spl{y}_l$, while $1\otimes x$ and
$1\otimes y$ correspond to $\spl{x}_r$ and $\spl{y}_r$. Moreover, the
dot action is given by $u\cdot(f\otimes h)=u(f)\otimes h$.
\end{lemma}

\begin{proof}
The map is well-defined over $\poly^{I_2(n)}$: if
$p\in\poly^{I_2(n)}$, then
$\Psi(fp\otimes h)(w)=fp\,w(h)=f\,w(p)w(h)=\Psi(f\otimes ph)(w)$. It is
also multiplicative, since
$\Psi((f\otimes h)(f'\otimes h'))(w)=ff'w(hh')=\Psi(f\otimes h)(w)
\Psi(f'\otimes h')(w)$.

We next check that the image consists of splines. If $w\leftrightarrow
ws_\alpha$ is a right reflection edge, then
$\Psi(f\otimes h)(w)-\Psi(f\otimes h)(ws_\alpha)=
f\,w(h-s_\alpha(h))$. By Lemma~\ref{lem:reflect_divisible},
$h-s_\alpha(h)$ is divisible by $t_\alpha$, so this difference is
divisible by $t_{w(\alpha)}$, which is the edge label up to sign.

To prove injectivity, view both sides as left $\poly$-modules and localize
at the nonzero elements of $\poly$. Since $I_2(n)$ is a finite reflection
group, $\poly$ is free over $\poly^{I_2(n)}$
\cite[\S3.6]{HumphreysReflectionGroups}. Hence
$\poly\otimes_{\poly^{I_2(n)}}\poly$ is torsion-free as a left
$\poly$-module, so localization detects injectivity.

After localization, every edge label becomes a unit, so the localized spline
module is the full product
$\prod_{w\in I_2(n)}\operatorname{Frac}(\poly)$. The localized source is
initially $\operatorname{Frac}(\poly)\otimes_{\poly^{I_2(n)}}\poly$, but
the second tensor factor also localizes to $\operatorname{Frac}(\poly)$.
Indeed, if $a\in\poly$ is nonzero, then
$N(a)=\prod_{u\in I_2(n)}u(a)$ is a nonzero invariant polynomial. Since
$1\otimes N(a)=N(a)\otimes 1$ is invertible after localizing the first
factor, and $1\otimes N(a)=\prod_{u\in I_2(n)}(1\otimes u(a))$, each factor
$1\otimes u(a)$ is invertible in this commutative ring. In particular,
$1\otimes a$ is invertible.

Since
$\operatorname{Frac}(\poly)^{I_2(n)}=\operatorname{Frac}(\poly^{I_2(n)})$,
the localized source is
\[\operatorname{Frac}(\poly)\otimes_{\operatorname{Frac}(\poly)^{I_2(n)}}
\operatorname{Frac}(\poly).\] The field extension
$\operatorname{Frac}(\poly)/\operatorname{Frac}(\poly)^{I_2(n)}$ is Galois
with group $I_2(n)$, so the canonical map
\[
    \operatorname{Frac}(\poly)
        \otimes_{\operatorname{Frac}(\poly)^{I_2(n)}}
    \operatorname{Frac}(\poly)
        \longrightarrow
    \prod_{w\in I_2(n)}\operatorname{Frac}(\poly),
    \qquad\!\!
    a\otimes b\longmapsto \bigl(w\mapsto a\,w(b)\bigr),
\]
is an isomorphism. Thus the localized map is injective, and therefore
$\Psi$ is injective.

It remains to compare Hilbert series. By Chevalley's theorem for finite
reflection groups, $\poly^{I_2(n)}$ is a polynomial ring; for
$I_2(n)$, its basic invariant degrees are $2$ and $n$
\cite[\S\S3.5, 3.7, Table~1]{HumphreysReflectionGroups}. Hence
$\Hilb(\poly\otimes_{\poly^{I_2(n)}}\poly;q)=
(1-q^2)(1-q^n)/(1-q)^4=(1+q)(1+q+\cdots+q^{n-1})/(1-q)^2$. By
Corollary~\ref{cor:schubert_hilbert_series}, this is also the Hilbert series
of $\splines{\Phi^+}$. Since the graded pieces are finite-dimensional, the
injective graded map $\Psi$ is an isomorphism.

The identifications of the coordinate splines follow directly from
$\Psi(x\otimes 1)(w)=x$, $\Psi(y\otimes 1)(w)=y$,
$\Psi(1\otimes x)(w)=w(x)$, and $\Psi(1\otimes y)(w)=w(y)$. Finally, for
$u,v\in I_2(n)$, one has
$(u\cdot\Psi(f\otimes h))(v)=u(f\,(u^{-1}v)(h))=u(f)\,v(h)
=\Psi(u(f)\otimes h)(v)$, so the dot action corresponds to
$u\cdot(f\otimes h)=u(f)\otimes h$.
\end{proof}

Both submodules appearing in the quotients below are stable under the dot
action: $u\cdot\spl{x}_l$ and $u\cdot\spl{y}_l$ are linear combinations
of $\spl{x}_l$ and $\spl{y}_l$, while
$u\cdot\spl{x}_r=\spl{x}_r$ and $u\cdot\spl{y}_r=\spl{y}_r$.

\begin{proposition}\label{prop:left_schubert_rep}
    With respect to the left $\poly$-module structure on
    $\splines{\Phi^+}$, the quotient representation
    \[
        \lqot{\Phi^+}
            \deftobe
        \faktor{\splines{\Phi^+}}
        {\llangle \spl{x}_l,\spl{y}_l\rrangle\splines{\Phi^+}}
    \]
    is a direct sum of trivial representations, with one copy in degree
    $\ell(w)$ for each $w\in I_2(n)$. Equivalently,
    \[
        \lqot{\Phi^+}
            \cong
        \bigoplus_{w\in I_2(n)} q^{\ell(w)}\one
    \]
    as a graded $I_2(n)$-representation. Hence
    \[
        \lqot{\Phi^+}
            \cong
        (1+q)(1+q+\cdots+q^{n-1})\one.
    \]
\end{proposition}

\begin{proof}
    By Lemma~\ref{lem:borel_presentation_schubert_splines}, quotienting by
    $\llangle \spl{x}_l,\spl{y}_l\rrangle$ corresponds to quotienting
    $\poly\otimes_{\poly^{I_2(n)}}\poly$ by
    $\llangle x\otimes 1,y\otimes 1\rrangle$. Hence
    \[
        \lqot{\Phi^+}
            \cong
        \C\otimes_{\poly^{I_2(n)}}\poly,
    \]
    where $\C=\poly/\llangle x,y\rrangle$. Since the dot action on
    $\poly\otimes_{\poly^{I_2(n)}}\poly$ acts only on the first tensor
    factor, and that factor has been quotiented to constants, the dot action
    on $\lqot{\Phi^+}$ is trivial.

    The grading is determined by Proposition~\ref{prop:schubert_basis}: since
    $\splines{\Phi^+}$ is free over the left copy of $\poly$ with
    homogeneous basis $\{\spl{\sigma}_w\}_{w\in I_2(n)}$, the classes
    $[\spl{\sigma}_w]$ form a homogeneous vector-space basis for
    $\lqot{\Phi^+}$, with $\deg\spl{\sigma}_w=\ell(w)$. Thus there is one
    copy of the trivial representation in degree $\ell(w)$ for each
    $w\in I_2(n)$. Since
    $\sum_{w\in I_2(n)}q^{\ell(w)}=(1+q)(1+q+\cdots+q^{n-1})$, the claimed
    graded character follows.
\end{proof}

\begin{proposition}\label{prop:right_schubert_rep}
    With respect to the right $\poly$-module structure on
    $\splines{\Phi^+}$, the quotient representation
    \[
        \rqot{\Phi^+}
            \deftobe
        \faktor{\splines{\Phi^+}}
        {\llangle \spl{x}_r,\spl{y}_r\rrangle\splines{\Phi^+}}
    \]
    is isomorphic, as a graded $I_2(n)$-representation, to the coinvariant
    algebra
    \[
        \faktor{\poly}{\llangle(\poly^{I_2(n)})_+\rrangle}.
    \]
    In particular, since the action on $\poly_1$ is $V_1$,
    \[
        \operatorname{Tr}\bigl(g\mid \rqot{\Phi^+};q\bigr)
            =
        \frac{(1-q^2)(1-q^n)}{\det(I-qV_1(g))}
    \]
    for every $g\in I_2(n)$. Setting $g=e$ gives
    \[
        \dim_q \rqot{\Phi^+}
            =
        \frac{(1-q^2)(1-q^n)}{(1-q)^2}
            =
        (1+q)(1+q+\cdots+q^{n-1}).
    \]
\end{proposition}

\begin{proof}
    By Lemma~\ref{lem:borel_presentation_schubert_splines}, quotienting by
    $\llangle \spl{x}_r,\spl{y}_r\rrangle$ corresponds to quotienting
    $\poly\otimes_{\poly^{I_2(n)}}\poly$ by
    $\llangle 1\otimes x,1\otimes y\rrangle$. Hence
    \[
        \rqot{\Phi^+}
            \cong
        \poly\otimes_{\poly^{I_2(n)}}\C
            \cong
        \faktor{\poly}{\llangle(\poly^{I_2(n)})_+\rrangle},
    \]
    where $\C=\poly/\llangle x,y\rrangle$. Since the dot action on
    $\poly\otimes_{\poly^{I_2(n)}}\poly$ is
    $u\cdot(f\otimes h)=u(f)\otimes h$, the induced action on this quotient
    is the usual action of $I_2(n)$ on the coinvariant algebra.

    By Chevalley's theorem for finite reflection groups, $\poly^{I_2(n)}$
    is a polynomial ring; its basic invariant degrees
    are $2$ and $n$
    \cite[\S\S3.5, 3.7, Table~1]{HumphreysReflectionGroups}. Thus
    $\poly^{I_2(n)}\cong \C[F_2,F_n]$, with $\deg F_2=2$ and
    $\deg F_n=n$. Since $F_2,F_n$ form a homogeneous regular sequence and
    are fixed by $I_2(n)$, quotienting by them multiplies the graded
    character of $\poly$ by $(1-q^2)(1-q^n)$. The graded character of
    $\poly$ is $1/\det(I-qV_1(g))$, so
    \[
        \operatorname{Tr}\bigl(g\mid \rqot{\Phi^+};q\bigr)
            =
        \frac{(1-q^2)(1-q^n)}{\det(I-qV_1(g))}.
    \]
    The Hilbert series follows by setting $g=e$.
\end{proof}

\begin{remark}
    These two quotients have the same Hilbert series, and indeed the same
    underlying graded vector-space type, but different $I_2(n)$-actions. The
    dot action acts on the first tensor factor in
    $\poly\otimes_{\poly^{I_2(n)}}\poly$. Therefore the left quotient kills
    the factor on which $I_2(n)$ acts, giving a trivial representation,
    while the right quotient kills the other factor, leaving the usual coinvariant representation.
\end{remark}

\begin{remark}
    The elementary construction above is the rank-two $I_2(n)$ specialization
    of Kaji's finite-Coxeter equivariant Schubert calculus. Kaji denotes the
    double coinvariant presentation by $R_W[t;x]$
    \cite[Proposition~4.1]{KajiEquivariantSchubertCoxeter}; the equivariant
    Schubert classes form a free $R[t]$-basis of this ring
    \cite[Theorem~4.5]{KajiEquivariantSchubertCoxeter}. Kaji also defines the
    finite-Coxeter GKM ring
    \cite[Definition~4.7]{KajiEquivariantSchubertCoxeter} and proves that the
    localization map identifies it with $R_W[t;x]$
    \cite[Corollary~4.10]{KajiEquivariantSchubertCoxeter}. In this rank-two setting, we give an elementary proof using the Bruhat order, spline divisibility conditions, and basic invariant theory.
\end{remark}


\section{Divided Differences and Flow-up Bases}\label{sec:flowup} 

Throughout this section, all module generators, bases, and freeness
statements for $\Ms$ hold equally for both the left $\poly$-module structure defined in
\eqref{eqn:left_action} and the right $\poly$-action structure defined in \eqref{eqn:right_action}. We consider an arbitrary Hessenberg space $H\subseteq \Phi^+$, and show $\Ms$ has a flow-up basis (Definition \ref{def:flowup}) using a combination of the divided difference, the dot action, and Schubert splines.

\begin{lemma}\label{lem:generalized_divided_difference}
    Fix a Hessenberg space $H$, a spline $\spl{\rho} \in \Ms$, and $\alpha \in \Phi^+$. If $(v,s_\alpha v) \in \ce_H$ for all $v \in \supp(\spl{\rho})$, then 
    \begin{equation}\label{eq:divided_difference}
        \partial_\alpha\spl{\rho} \coloneqq  \frac{\spl{\rho} - s_\alpha \cdot \spl{\rho}}{t_\alpha} \in \Ms.
    \end{equation}
\end{lemma}
\begin{proof}
First we check that $\partial_\alpha \spl{\rho}$ is polynomial-valued at an arbitrary $v$. If neither $v$ nor $s_\alpha v$ are in the support, then divisibility holds immediately (the evaluation is $0$).
 By assumption, if either $v$  or $s_\alpha v$ are in the support of $\spl{\rho}$, then $\spl{\rho}(v)-\spl{\rho}(s_\alpha v)$ is divisible by $t_\alpha$. 
 Also $\spl{\rho}(s_\alpha v)-s_\alpha(\spl{\rho}(s_\alpha v))$ is divisible by $t_\alpha$ by Lemma~\ref{lem:reflect_divisible}. 
 Hence the numerator in \eqref{eq:divided_difference} evaluated at $v$, which can be written as 
\[
\spl{\rho}(v) - \spl{\rho}(s_\alpha v) + \spl{\rho}(s_\alpha v) - s_\alpha\spl{\rho}(s_\alpha v),
\]
is divisible by $t_\alpha$.

It remains to check the spline condition. Fix $\beta\in H$ and
$v\in I_2(n)$. Then 
\begin{align*}
    \partial_\alpha \spl{\rho}(v) - \partial_\alpha \spl{\rho}(vs_\beta) &= \frac{\spl{\rho}(v) - s_\alpha \cdot \spl{\rho}(v) - \spl{\rho}(vs_\beta) + s_\alpha \cdot \spl{\rho}(vs_\beta)}{t_\alpha}\\
    &= \frac{\left(\spl{\rho}(v)- \spl{\rho}(vs_\beta)\right) - \left(s_\alpha \cdot \spl{\rho}(v)  - s_\alpha \cdot \spl{\rho}(vs_\beta)\right)}{t_\alpha}.
\end{align*}
Both $\spl{\rho}$ and $s_\alpha\cdot\spl{\rho}$ are splines, the latter
by Lemma~\ref{lem:dot_action_preserves_splines}. Hence
the numerator is divisible by $v(t_\beta) = t_{v(\beta)}$. \par 
Now if $\alpha \neq  \pm v(\beta)$, the linear forms $t_\alpha$ and $t_{v(\beta)}$ are
relatively prime in $\poly$. Thus the numerator is divisible by $t_\alpha t_{v(\beta)}$, and
after dividing by $t_\alpha$ we get
$\partial_\alpha \spl{\rho}(v)-\partial_\alpha \spl{\rho}(vs_\beta)\in \llangle t_{v(\beta)}\rrangle$.

It remains only to handle $v(\beta)=\pm \alpha$, wherein $vs_\beta = s_\alpha v$ and the edge in question is $(v,s_\alpha v) \in \ce_H$. Since $s_\alpha(t_\alpha)=-t_\alpha$,
equation \eqref{eq:divided_difference} gives
\begin{align*}
    \partial_\alpha \spl{\rho}(s_\alpha v) &= \frac{\spl{\rho}(s_\alpha v) - s_\alpha \cdot \spl{\rho}(s_\alpha v)}{t_\alpha} \\
    &= \frac{\spl{\rho}(s_\alpha v) - s_\alpha \spl{\rho}(v)}{t_\alpha} \\
    &=  \frac{s_\alpha \left(s_\alpha \cdot \spl{\rho}(v) - \spl{\rho}(v)\right)}{t_\alpha} \\
    &=  s_\alpha\left(\frac{s_\alpha \cdot \spl{\rho}(v) - \spl{\rho}(v)}{-t_\alpha}\right) \\
    &=s_\alpha(\partial_\alpha \spl{\rho}(v)).
\end{align*}

Hence
$\partial_\alpha \spl{\rho}(v)-\partial_\alpha \spl{\rho}(s_\alpha v)
=
\partial_\alpha \spl{\rho}(v)-s_\alpha(\partial_\alpha \spl{\rho}(v))$, which is divisible
by $t_\alpha$ by Lemma~\ref{lem:reflect_divisible}.
\end{proof}

\begin{lemma}\label{lem:leftright_length}
    Let $\alpha$ be a simple root, $w \in I_2(n)$, and $\gamma$ the root such that $w < s_\alpha w = ws_\gamma$. Then 
    \[
    \ell(s_\gamma) = \begin{cases}
        2\ell(w) + 1 &\text{if }2\ell(w)+1 \leq n \\
        2(n-\ell(w)) - 1 &\text{otherwise.}
        \end{cases}
    \]
    If on the other hand $w > s_\alpha w = ws_\gamma$, then 
    \[
    \ell(s_\gamma) = \begin{cases}
        2\ell(w) - 1 &\text{if }2\ell(w)-1 \leq n \\
        2(n-\ell(w)) +1 &\text{otherwise.}
        \end{cases}
    \]
\end{lemma}
\begin{proof}
    Let $\beta$ be the other simple root. Let $w < s_\alpha w$ so that a reduced word for $w$ is $w = s_\beta s_\alpha \cdots s_r$ for $r \in \{\alpha,\beta\}$.
    
    The first case is straightforward, as 
    \[s_\gamma = s_r \cdots s_\alpha s_\beta (s_\alpha) s_\beta s_\alpha \cdots s_r\]
    is a reduced word if and only if its length is at most $n = \ell(w_0)$. If $2\ell(w) + 1 > n$, then the above expression for $s_\gamma$ can be expressed as $w_0u$ with $\ell(u) = 2\ell(w)+1 - n$. As $\ell(w) < n$ it follows that $\ell(u) < n$, thus $\ell(s_\gamma) = n - \ell(u) = 2(n-\ell(w)) - 1$.

    The result for $w > s_\alpha w$ can be achieved by replacing $\ell(w)$ with $\ell(s_\alpha w) = \ell(w) - 1$.
\end{proof}

\begin{lemma}\label{lem:simple_degree_changes}
Consider $w \in I_2(n)$, a simple root $\alpha$, and say that $w < s_\alpha w$. Let $H$ be a Hessenberg space. The following hold.
\begin{enumerate}
    \item If $(w,s_\alpha w) \notin \ce_H$, then $\hinv(w) = \hinv(s_\alpha w)$.
    \item If $(w,s_\alpha w) \in \ce_H$, then $\hlen(w) + 1 = \hlen(s_\alpha w)$.
    \item If $(w,s_\alpha w)\in\ce_H$, then one of the following holds:
    \begin{itemize}
        \item $(v,s_\alpha v)\in\ce_H$ for every $v$ with
        $\ell(v)\leq \ell(w)$, occurring when $2\ell(w) + 1 \leq n$, or
        \item $(v,s_\alpha v)\in\ce_H$ for every $v$ with
        $\ell(v) > \ell(w)$ otherwise.
    \end{itemize}
    \item If $(w,s_\alpha w) \in \ce_H$ where $(v,s_\alpha v) \in \ce_H$ for all $v$ where $\ell(v) \leq \ell(w)$ and $2\ell(w) + 1 \leq n$, then $\hlen(w) = \ell(w)$.
\end{enumerate} 
\end{lemma}
\begin{proof}
    (1) Since $w < s_\alpha w$, we know that $\alpha$ is a left descent of $s_\alpha w$. Since $(w,s_\alpha w)\notin \ce_H$, we know that $w^{-1}(\alpha) \notin H$. Recall that $N_R(w)$ is the set of right inversions of $w$, and  $N_R(s_\alpha w) = N_R(w) \cup \{w^{-1}(\alpha)\}$ by Lemma \ref{lem:rank_two_bruhat_facts}, and so the intersection $N_R(w) \cap H = \hinv(w)$ does not change. Essentially the exact same logic applies to prove (2), where $w^{-1}(\alpha) \in H$.\par 
    (3) If $s_\alpha w = ws_\gamma$, then $s_\gamma = w^{-1}s_\alpha w$. By Lemma \ref{lem:leftright_length}
    \[\ell(s_\gamma) = \begin{cases}
        2\ell(w) + 1 &\text{if }2\ell(w)+1 \leq n \\
        2(n-\ell(w)) - 1 &\text{otherwise.}
    \end{cases}\]
    Since $H$ is a lower order ideal for $I_2(n)$, $\gamma \in H$ implies $\gamma' \in H$ for all $\gamma'$ such that $\ell(s_{\gamma'}) < \ell(s_\gamma)$. Assume first that $\ell(s_\gamma) = 2\ell(w)+1$. If $\ell(v) \leq \ell(w)$ then
    \[\ell(v^{-1}s_\alpha v) \leq 2\ell(v) + 1 \leq 2\ell(w) + 1 \leq n.\]
    Thus $\ell(v^{-1}s_\alpha v) = 2\ell(v) \pm 1 \leq \ell(s_\gamma)$, with equality only when $v=w$,
    and so $(v,s_\alpha v) \in \ce_H$. 
    
    Now assume $2\ell(w)+1 > n$ and consider all $v$ where $\ell(v) > \ell(w)$. If $\ell(v) \geq \ell(w) + 2$ then \[\ell(v^{-1} s_\alpha v) = 2(n-\ell(v)) \pm 1 < 2(n-\ell(w))-1 = \ell(s_\gamma).\] 
    If $\ell(v) = \ell(w)+1$, there are two options. Either $v = s_\alpha w$, in which case $(v,s_\alpha v) \in \ce_H$ by assumption. If $v \neq s_\alpha w$ is the other element of length $\ell(w) + 1$, then $v < s_\alpha v$, in which case $\ell(v^{-1} s_\alpha v) = 2(n-\ell(v)) - 1$ and the required inequality holds. 

    (4) This also leverages Lemma \ref{lem:leftright_length}. If $\zeta\in N_L(w)$ and $\gamma \in \Phi^+$ is such that $s_\zeta w =ws_\gamma$, then 
    \[\ell(s_\gamma) = \ell(w^{-1}s_\zeta w) \leq \ell(w^{-1}) + \ell(s_\zeta w) < 2\ell(w) < \ell(w^{-1}s_\alpha w),\]
    and so $\gamma \in H$ and $(w,s_\zeta w) \in \ce_H$. Thus $N_R(w) \subseteq H$, so $\hlen(w)=\abs{N_R(w)}=\ell(w)$.
    \end{proof}

\begin{corollary}\label{cor:existence_upper_elts}
    Let $H$ be a Hessenberg space. If $w \in I_2(n)$ with $2\ell(w)+1 > n$, $\alpha$ is a simple root such that $w < s_\alpha w$ and $(v,s_\alpha v) \in \ce_H$ for all $v \geq w$, then there exists a flow-up spline at $w$.
\end{corollary}
\begin{proof}
We can assume that $w \notin \{e,w_0\}$. Consider an arbitrary $v \in I_2(n)$ such that $w \leq v < w_0$. We first show that if $s_\gamma$ is the left-ascent of $v$ with $\gamma \in \{\alpha, \beta\}$, then $(v,s_\gamma v) \in \ce_H$. If $\gamma = \alpha$, then the edge $(v,s_\gamma v) \in \ce_H$ by assumption. 

If $s_\beta v > v$, then $v \neq w$. In particular, $\ell(v) > \ell(w)$, so Lemma \ref{lem:leftright_length} gives us that $\ell(v^{-1}s_\beta v)$ is strictly less than $\ell(w^{-1}s_\alpha w)$. 
So if $s_\beta v = vs_\zeta$ and $s_\alpha w = ws_\nu$, it follows that $\ell(s_\zeta) < \ell(s_\nu)$ and thus $(v,s_\beta v) \in \ce_H$.

So $(v,s_\gamma v) \in \ce_H$ for all $s_\gamma v > v \geq w$ and $\gamma \in \{\alpha,\beta\}$.
We proceed by induction on $\ell(v)$ backwards from $\ell(w_0) = n$, where such a spline exists:
\begin{equation}\label{eq:top_spline}
    \spl{\rho}_{w_0}(v) \coloneqq \begin{cases}
    \ds \prod_{\alpha \in H} t_{w_0(\alpha)} & \text{if }v = w_0 \\ 
    0 & \text{otherwise.}
\end{cases}
\end{equation}
Say that such a spline exists for all $v'$ such that $\ell(v') > \ell(v)$. If $\gamma \in \{\alpha,\beta\}$ is such that $s_\gamma v > v$, our assumption guarantees that $(v,s_\gamma v) \in \ce_H$. Our induction assumption gives a spline $\spl{\rho}_{s_\gamma v}$ supported at $v' = s_\gamma v$ and above of degree $\hlen(s_\gamma v)$. If $u \in \supp(\spl{\rho}_{s_\gamma v})-\{s_\gamma v\}$, then $\ell(u) > \ell(s_\gamma v)$, so $\ell(u^{-1}s_\gamma u) < \ell(v^{-1} s_\gamma v)$ (even if $s_\gamma u < u$) implying $(u, s_\gamma u) \in \ce_H$.
So we can apply Lemma \ref{lem:generalized_divided_difference} to get a spline $\partial_\gamma \spl{\rho}_{s_\gamma v}$ of degree $\hlen(s_\gamma v) - 1$, which is precisely $\hlen(v)$ by Lemma \ref{lem:simple_degree_changes}(2). This spline $\partial_\gamma \spl{\rho}_{s_\gamma v}$ is nonzero at $v$ since \[\partial_\gamma \spl{\rho}_{s_\gamma v}(v)=-s_\gamma\bigl(\bar\rho_{s_\gamma v}(s_\gamma v)\bigr)/t_\gamma\neq 0.\] By assumption, $\spl{\rho}_{s_\gamma v}$ is supported only at $s_\gamma v$ and above. If $u\geq s_\gamma v$ and $u\neq s_\gamma v$, then $\ell(s_\gamma u)\ge\ell(v)+1$, so $s_\gamma u\geq v$ by Lemma~\ref{lem:rank_two_bruhat_facts}(2), and thus $s_\gamma \cdot \spl{\rho}_{s_\gamma v}$ is also supported only at $v$ and above. Therefore, the spline $\partial_\gamma \spl{\rho}_{s_\gamma v}$ is supported only at $v$ and above. The induction argument terminates at $v = w$ and $\gamma = \alpha$, and we have the claim.
\end{proof}

\begin{theorem}\label{thm:flowup}
    Fix a Hessenberg space $H \subseteq \Phi^+$. Then  $\Ms$ has a flow-up basis $\{\spl{\rho}_w \mid w \in I_2(n)\}$, given by
    \[\spl{\rho}_w = \begin{cases}
        \spl{\one} &\text{if }w = e \\
        \spl{\rho}_{w_0}\text{ from } \eqref{eq:top_spline} &\text{if }w = w_0 \\
        \spl{\sigma}_w \text{ from } \eqref{eq:schubert_recursion} & \text{if } (w,s_\alpha w)\in\ce_H \text{ and } 2\ell(w)+1\leq n,\\
    \partial_\alpha\spl{\rho}_{s_\alpha w} & \text{if } (w,s_\alpha w)\in\ce_H \text{ and } 2\ell(w)+1> n,\\
    s_\alpha\cdot\spl{\rho}_{s_\alpha w} & \text{if } (w,s_\alpha w)\notin\ce_H,
    \end{cases}\]
    where $\alpha$ is the unique left simple ascent of $w\notin \{e,w_0\}$.
\end{theorem}
\begin{proof}
    Consider $w \in I_2(n)$. The identity spline $\spl{\one}$ is a flow-up spline at $e$ and  $\spl{\rho}_{w_0}$ as in the proof of Corollary \ref{cor:existence_upper_elts} is a flow-up spline at $w_0$. So let $w \notin \{e,w_0\}$ and $\alpha$ be the simple root where $s_\alpha w > w$. 

    If $(w,s_\alpha w)\in\ce_H$, then Lemma~\ref{lem:simple_degree_changes}(3) gives $(v,s_\alpha v)\in\ce_H$ for either (a) all $v$ with $\ell(v)\leq \ell(w)$ when $2\ell(w)+1\le n$, or (b) for all $v$ with $\ell(v)>\ell(w)$ otherwise. If $(w,s_\alpha w) \in \ce_H$ and $2\ell(w)+1 \leq n$, then we take the Schubert spline, which has degree $\hlen(w)$ by Lemma \ref{lem:simple_degree_changes}(4). If $(w,s_\alpha w) \in \ce_H$ and $2\ell(w) + 1 > n$ then we have a flow-up spline at $w$ by Corollary \ref{cor:existence_upper_elts}.

    If $(w,s_\alpha w) \notin \ce_H$, then we can proceed by downwards induction, with the assumption that a flow-up spline exists at all $v > w$. In particular, there is a flow-up spline at $s_\alpha w$. Let $\spl{\rho}$ be the flow-up spline at $s_\alpha w$ and consider $s_\alpha \cdot \spl{\rho}$, which has the appropriate support since $\{s_\alpha v \mid v \geq s_\alpha w\} \subseteq \{v \mid v \geq w\}$ in strong Bruhat order when $s_\alpha w > w$, and the appropriate degree by Lemma \ref{lem:simple_degree_changes}(1), and is thus a flow-up spline at $w$.
\end{proof}

\begin{example}\label{ex:generators_exist}
    If $n = 5$ and $H = \{\alpha_1,\alpha_2,\alpha_3\}$, the labeled Cayley graph $\cg_H$ is given in Figure~\ref{fig:generators_exist}.
    Theorem \ref{thm:flowup} guarantees the existence of splines:
    \begin{equation*}
        \begin{matrix}
            \spl{\rho}_e &\text{supported on}& I_2(n) & \text{given by} & \spl{\sigma}_e = \spl{\one} \\
            \spl{\rho}_{s_1} &\text{supported weakly above}& s_1&\text{given by} & \spl{\sigma}_{s_1} \\
            \spl{\rho}_{s_2} &\text{supported weakly above}& s_2&\text{given by} & s_1s_2 \cdot \partial_1\partial_2 \spl{\rho}_{w_0}\\
            \spl{\rho}_{s_1s_2} &\text{supported weakly above}& s_1s_2&\text{given by} & s_2 \cdot \partial_1\partial_2 \spl{\rho}_{w_0} \\
            \spl{\rho}_{s_2s_1} &\text{supported weakly above}& s_2s_1&\text{given by} & s_1s_2 \cdot \partial_1\spl{\rho}_{w_0}\\
            \spl{\rho}_{s_1s_2s_1} &\text{supported weakly above}& s_1s_2s_1&\text{given by} & s_2 \cdot \partial_1 \spl{\rho}_{w_0}\\
            \spl{\rho}_{s_2s_1s_2} &\text{supported weakly above}& s_2s_1s_2&\text{given by} & \partial_1\partial_2 \spl{\rho}_{w_0}\\
            \spl{\rho}_{(s_1s_2)^2} &\text{supported weakly above}& (s_1s_2)^2&\text{given by} & \partial_2 \spl{\rho}_{w_0}\\
            \spl{\rho}_{(s_2s_1)^2} &\text{supported weakly above}& (s_2s_1)^2&\text{given by} & \partial_1 \spl{\rho}_{w_0}\\
            \spl{\rho}_{w_0} &\text{supported only at}& (s_1s_2)^2s_1 &\text{given by} & \spl{\rho}_{w_0}(w_0) = -t_{\alpha_1}t_{\alpha_2}t_{\alpha_4}.
        \end{matrix}
    \end{equation*}
    It follows that the rank generating function for $\Ms$ is $1+4q+4q^2+q^3$.
\end{example}

\begin{figure}
    \begin{tikzpicture}
        \begin{scope}[xshift=0in, scale=1.6]
            \tikzstyle{every node}=[shape=rectangle, inner sep=2pt]
            \draw  (0,0.2) node[draw] (v0) {$e$};
            \draw  (1.2,1) node[draw]  (v1) {$s_2$};
            \draw  (-1.2,1) node[draw] (v2) {$s_1$};
            \draw  (-1.5,2) node[draw] (v3) {$s_1 s_2$};
            \draw  (1.5,2) node[draw] (v4) {$s_2 s_1$};
            \draw (1.5,3) node[draw] (v5) {$s_2 s_1 s_2$};
            \draw (-1.5,3) node[draw] (v6) {$s_1 s_2s_1$};
            \draw (-1.2,4)  node[draw]  (v7) {$(s_1 s_2)^2$};
            \draw (1.2,4)  node[draw]  (v8) {$(s_2 s_1)^2$};
            \draw (0,4.8)  node[draw]  (v9) {$(s_1 s_2)^2s_1$};
    
            \draw[thick] (v0)--(v1)node[midway,right] {\textcolor{blue}{ $\;\;\alpha_2$}};
            \draw[thick] (v1)--(v4)node[midway,right] {\textcolor{blue}{ $\alpha_4$}};
            \draw[thick] (v4)--(v5)node[midway,right] {\textcolor{blue}{ $\alpha_5$}};
            \draw[thick] (v5)--(v8)node[midway,right] {\textcolor{blue}{ $\alpha_3$}};
            \draw[thick] (v8)--(v9)node[midway,right] {\textcolor{blue}{ $\;\alpha_1$}};
            \draw[thick] (v9)--(v7)node[midway,left] {\textcolor{blue}{ $\alpha_2\;$}};
            \draw[thick] (v7)--(v6)node[midway,left] {\textcolor{blue}{ $\;\;\;\alpha_4$}};
            \draw[thick] (v6)--(v3)node[midway,left] {\textcolor{blue}{ $\;\;\;\alpha_5$}};
            \draw[thick] (v3)--(v2)node[midway,left] {\textcolor{blue}{ $\;\;\;\alpha_3$}};
            \draw[thick] (v2)--(v0)node[midway,left] {\textcolor{blue}{ $\alpha_1\;$}};
            
            \draw[thick] (v0) to[bend right=15] (v6);
            \draw (-0.5,2) node (ll1) {\textcolor{blue}{$\alpha_3$}};
        
            \draw[thick] (v1) to[bend left=30] (v8);
            \draw (0.5,2.5) node (ll1) {\textcolor{blue}{$\alpha_5$}};
            
            \draw[thick] (v2) to (v4);
            \draw (0.3,1.35) node (ll1) {\textcolor{blue}{$\alpha_2$}};
            
            \draw[thick] (v3) to[bend right=20] (v9);
            \draw (-0.4,3) node (ll1) {\textcolor{blue}{$\alpha_4$}};
            
            \draw[thick] (v5) to (v7);
            \draw (0.3,3.6) node (ll1) {\textcolor{blue}{$\alpha_1$}};
        
        \end{scope}
    \end{tikzpicture}
    \caption{The labeled Cayley graph $\cg_H$ in Example~\ref{ex:generators_exist}.}
    \label{fig:generators_exist}
\end{figure}
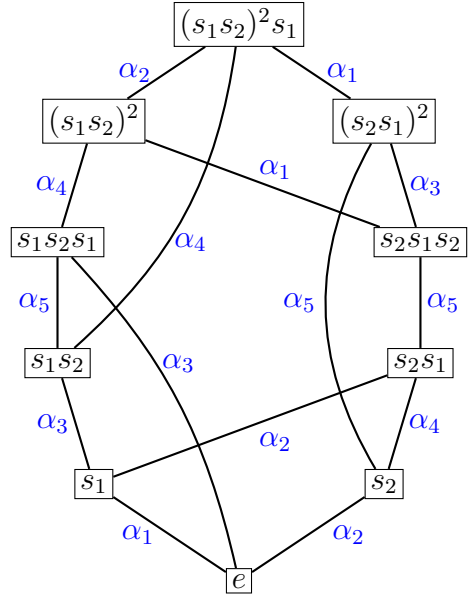

\begin{corollary}\label{cor:rank_gf}
    The rank generating function for $\Ms$ is palindromic and given by \[P_H(q) = \sum_{w \in I_2(n)} q^{\hlen(w)}.\]

    If $\abs{H} = 2h$ then 
    \begin{align*}
        P_H(q) &= 
        \ds 1 + 2q + \cdots + 2q^{h-1} + (2n-4h+2)q^h + 2q^{h+1} + \cdots + 2q^{2h-1} + q^{2h} \\
        &= (1+q)(1+q+\cdots+q^{2h-1}) + (2n-4h)q^h.
    \end{align*}
    If $\abs{H} = 2h+1$ then 
    \begin{align*}
        P_H(q) &= 1 + 2q + \cdots + 2q^{h-1} + (n-2h+1)\left(q^{h} + q^{h+1}\right) + 2q^{h+2} + \cdots  + 2q^{2h} + q^{2h+1}\\
        &= (1+q)(1+q+\cdots+q^{2h}) + (n-2h-1)(q^h+q^{h+1})
    \end{align*}
    Furthermore, the Hilbert series of $\Ms$ is 
    \[\Hilb(\Ms;q) = \frac{P_H(q)}{(1-q)^2}.\]
\end{corollary}
\begin{proof}
    The first formulation follows directly from Theorem \ref{thm:flowup}. As for being palindromic, the longest element satisfies $w_0(\Phi^+) = \Phi^-$, so $\hinv(w_0w) = H \setminus \hinv(w)$ for all $w \in I_2(n)$ and $w \mapsto w_0w$ gives an $H$-inversion reversing involution.

    Now we compute the first $h$ coefficients of $P_H(q)$. Clearly $P_H(0) = 1$. Let $w \in I_2(n)$ be a non-identity element and say without loss of generality that a reduced word for $w$ ends in $s_\alpha$. Then $w(s_\alpha s_\beta)^ks_\alpha < w$ for all $0 \leq k < \ell(w)$. Specifically, if $\ell(w) < h$ then $\ell(w) = \hlen(w)$ and if $\ell(w) \geq h$ then $\hlen(w)\geq h$. Thus the coefficient of $q^k$ for $k \in [h-1]$ is $2$, the number of $w \in I_2(n)$ with length $k$, and so the initial segment of $h$ coefficients is $(1,2,\ldots,2)$. The polynomials above are the unique palindromic polynomials of degree $\abs{H}$ with coefficient sum $2n$ that have this initial segment of coefficients.
\end{proof}

\section{Properties of the Flow-up Basis}\label{sec:properties}

This section will prove two main results. The first is a concrete description of $\spl{\rho}_w(v)$ for all $v$ and $w$. The second is a linear independence result on particular orbits $I_2(n) \cdot \spl{\rho}_w$ that will be critical for our computations in Sections \ref{sec:left_rep} and \ref{sec:right_rep}.

\begin{lemma}\label{lem:regular}
    Consider $G = \{w,ws_1,ws_1s_2,\ldots,w(s_1s_2)^{k}s_1\}$. If $\alpha_{2k+1} \in H$ then the subgraph of $\cg_H$ induced by $G$ is regular.\par 
    Similarly, if $G = \{w,ws_2,\ldots,w(s_2s_1)^{k}s_2\}$ and $\alpha_{2k+2} \in H$, then the subgraph of $\cg_H$ induced by $G$ is regular.
\end{lemma}
\begin{proof}
    The root $\alpha_{2k+1}$ is associated to the reflection $(s_1s_2)^{k}s_1$. If $\alpha_{2k+1} \in H$ then $\alpha_{i} \in H$ for all $i \leq 2k$. In particular, an edge exists in $\cg$ for right multiplication by \emph{any} reflection of length at most $2k-1$.\par 
    List the elements of $G$ as such:
    \[(v_0,v_1,\ldots,v_{2k+1}) \coloneqq (w,ws_1,\ldots,w(s_1s_2)^{k}s_1).\]
    Then $(v_0,v_j)$ is an edge in $\cg$ for every odd $j \in \{0,\ldots,2k+1\}$. Similarly, for any even $i$ and odd $j$ in $[2k+1]$, the pair $(v_i,v_j)$ is an edge in $\cg$ via: \begin{itemize}
        \item right multiplication by some $(s_1s_2)^ls_1$ for $l \leq k$ if $i<j$, or
        \item right multiplication by some $(s_2s_1)^ls_2$ for $l < k$ if $j<i$.
    \end{itemize}
    In particular, every even-indexed $v_i$ in $G$ is connected to every odd-indexed $v_j$ in $G$, and vice versa. Since $\{0,\ldots,2k+1\}$ has the same number of odd and even integers (and two elements in $G$ of the same parity index cannot be connected), we get that $G$ is regular.\par 
    The proof for the $G = \{w,\ldots,w(s_2s_1)^ks_2\}$ is the same with slightly different indices, and in any case follows by symmetry.
\end{proof}
So the number of edges \emph{leaving} $G$ in $\cg_H$ is also the same for each vertex in $G$. In particular, we get the following corollary.

\begin{corollary}\label{cor:rationally_smooth}
    Fix $w \in I_2(n)$. If $\alpha_{2k+1} \in H$ then any nonzero homogeneous spline with support contained in \[G = \{w,ws_1,ws_1s_2,\ldots,w(s_1s_2)^{k}s_1\},\]
    respectively $\alpha_{2k+2} \in H$ and with support contained in
    \[G = \{w,ws_2,\ldots,w(s_2s_1)^{k}s_2\},\]
    of degree $N = \abs{\{\alpha \in H \mid ws_\alpha \notin G\}}$  must be of the form 
    \[\spl{\sigma}(v) = c^w(v)\prod_{\substack{\alpha \in H \\ vs_\alpha \notin G}}t_{v(\alpha)}\]
    where $c^w(v) \in \C^*$ when $v \in G$ and $\spl{\sigma}(v) = 0$ otherwise. Moreover, if such a spline exists, there is a unique one
normalized by $c^w(w)=1$.
\end{corollary}
\begin{proof}
   The product $\ds \prod_{\substack{\alpha \in H \\ vs_\alpha \notin G}}t_{v(\alpha)}$ divides $\spl{\sigma}(v)$, as each $v(\alpha)$ is a pairwise non-proportional linear form and $\spl{\sigma}(vs_\alpha) = 0$ when $vs_\alpha \notin G$ by assumption. Lemma \ref{lem:regular} ensures that this product is degree $N$. As such, we get the above form for $c^w(v) \in \C$ and it remains to prove that $c^w(v) \neq 0$ for all $v \in G$.

   Since $\spl{\sigma} \neq 0$, certainly $c^w(v) \neq 0$ for \emph{some} $v \in G$. If $v_1,v_2 \in G$ are connected by an edge in $\cg_H$ labeled by $\beta$ and $c^{w}(v_1) \neq 0$ , then $t_\beta$ \emph{cannot} appear in the products for $\spl{\sigma}(v_1)$ or $\spl{\sigma}(v_2)$. Since these roots are pairwise non-proportional, $t_\beta$ does not divide $\spl{\sigma}(v_1)$.
    If $c^w(v_2)=0$, then $\spl{\sigma}(v_2)=0$, and the spline condition would force $t_\beta\mid\spl{\sigma}(v_1)$, a contradiction. Thus $c^w(v_2)\neq 0$. Since $G$ is connected, $c^w(v) \neq 0$ for all $v \in G$.

   The moreover part of the proof will proceed in much the same way. Say that $c^w(v_1)$ is uniquely determined from $c^w(w) = 1$ and $v_2$ is adjacent to $v_1$ with edge label $\beta$. Consider $\ds \left[\spl{\sigma}(v_1)\right], \left[\spl{\sigma}(v_2)\right] \in \faktor{\C[x,y]}{\llangle t_\beta \rrangle}$. This quotient is (non-canonically) isomorphic to the polynomial ring $\C[t]$. We have established that neither class is zero, and so choosing any particular graded isomorphism to $\C[t]$ gives $[\spl{\sigma}(v_1)] \mapsto c^w(v_1) c_1t^N$ and $[\spl{\sigma}(v_2)] \mapsto c^w(v_2) c_2t^N$ for some $c_1,c_2 \in \C^*$. Since $[\spl{\sigma}(v_1) - \spl{\sigma}(v_2)] = 0$, we get that $c^w(v_2) = \frac{c_1}{c_2}c^w(v_1)$. So we have determined the value of $c^w(v_2)$, and since $G$ is connected we have the claim.
\end{proof}
Now we turn these results towards the flow-up basis $\{\spl{\rho}_w \mid w \in I_2(n)\}$ that we have constructed with Schubert splines, the dot action, and the divided difference operator. We have now as well the tool to define their supports.

\begin{proposition}\label{prop:flowup_explicit_setup}
    Let $w \in I_2(n)$. Then either $\spl{\rho}_w$ is a Schubert spline, $w = w_0$ and $\spl{\rho}_{w_0}$ is given, or $\spl{\rho}_w$ satisfies the conditions in Corollary \ref{cor:rationally_smooth}.
\end{proposition}
\begin{proof}
    If $\spl{\rho} \in \Ms$, then $\supp(v \cdot \spl{\rho}) = \{vu \mid u \in \supp(\spl{\rho})\}$. Furthermore, since the dot action preserves degree, we have that if $\spl{\rho}$ satisfies the conditions in Corollary \ref{cor:rationally_smooth} then so does $v \cdot \spl{\rho}$ for all $v \in I_2(n)$.

    As such, it remains to check those flow-up splines $\spl{\rho}_w$ that are achieved by repeatedly applying the divided difference operator to $\spl{\rho}_{w_0}$. We take $\spl{\rho}_{w_0}$ as the base case, with its support and values given directly by \eqref{eq:top_spline}, and proceed by induction.
    
    Take such a $\spl{\rho}_w$ where $\alpha$ is a left ascent of $w$, $(w,s_\alpha w) \in \ce_H$, and $2\ell(w)+1>n$. By Theorem \ref{thm:flowup}, we have $\partial_\alpha \spl{\rho}_{s_\alpha w} = \spl{\rho}_{w}$. Now $\spl{\rho}_{w}$ is supported on the upper order ideal generated by $w$, so we let $G = \{v \in I_2(n) \mid w \leq v\}$. By Corollary \ref{cor:existence_upper_elts}, the degree of $\spl{\rho}_{w}$ is $\hlen(w) = \abs{\{\gamma \in H \mid w s_\gamma \notin G\}}$. Furthermore, if $s_1$ is the unique right ascent of $w$ then \[G = \left\{w, ws_1, ws_1s_2, \ldots,w_0,\ldots, w (s_1s_2)^ks_1 = s_\alpha w \right\}, \]
    where $\alpha_{2k+1} \in H$. If $s_2$ is the unique right ascent, then the same occurs with $\alpha_{2k+2}$, and thus $\spl{\rho}_{w}$ satisfies the conditions in Corollary \ref{cor:rationally_smooth}.
\end{proof}

When combined with Theorem \ref{thm:schubert_rational}, we have an explicit description of each flow-up spline $\spl{\rho}_w$. The support $G_w$ is either the upper order ideal generated by $w$, or the set given in Corollary \ref{cor:rationally_smooth}.

\begin{proposition}\label{prop:linear_independence}
    Fix an integer $k\geq 0$ with $2k+1<n$, and define
    \[
        H \coloneqq
        \{\gamma\in\Phi^+ \mid \ell(s_\gamma)\leq 2k+1\}
        =\{\alpha_1,\ldots,\alpha_{2k+2}\}.
    \]
    For each distinct set of the form
    \[
        G=\{w,ws_1,\ldots,w(s_1s_2)^ks_1\}
        \qquad\text{or}\qquad
        G=\{w,ws_2,\ldots,w(s_2s_1)^ks_2\},
    \]
    choose constants $c^G(v)\in\C^\times$ for $v\in G$,
    and define $\spl{\sigma}_G\colon I_2(n)\to\C[x,y]$ by
    \[
        \spl{\sigma}_G(v)=
        \begin{cases}
            \displaystyle c^G(v)
            \prod_{\substack{\alpha\in H\\
                            vs_\alpha\notin G}}
            t_{v(\alpha)}, & v\in G,\\[6pt]
            0, & v\notin G.
        \end{cases}
    \]
    Then the set of functions $\{\spl{\sigma}_G\}$ for all
    such $G$ is linearly independent over $\C$.
\end{proposition}

\begin{proof}
    Since there are duplicates in the descriptions of $G$,
    it will be convenient to choose an indexing for $I_2(n)$
    within this proof that removes this ambiguity. Let
    $v_0=e$, $v_1=s_1$, and more generally
    \begin{align*}
        v_{2j} &= (s_1s_2)^j,\\
        v_{2j+1} &= (s_1s_2)^js_1,
    \end{align*}
    up to $v_{2n-1}=s_2$. Read all vertex indices modulo $2n$.
    Then let
    \[
        G_i=\{v_i,v_{i+1},\ldots,v_{i+2k+1}\},
    \]
    and denote $\spl{\sigma}_{G_i}$ by $\spl{\sigma}_i$,
    for all $i\in\{0,\ldots,2n-1\}$.

    Consider a linear relation
    $\ds\sum_{i=0}^{2n-1}c_i\spl{\sigma}_i=0$,
    where $c_i\in\C$. Since this is a linear relation of
    functions, we may evaluate it at any $w\in I_2(n)$
    and it must evaluate to $0$. In particular,
    \[
        \sum_{i=0}^{2n-1}c_i\spl{\sigma}_i(e)=0
        \qquad\text{and}\qquad
        \sum_{i=0}^{2n-1}c_i\spl{\sigma}_i(v_1)=0.
    \]
    Many of these evaluations are $0$ because $e$ or $v_1$
    are not in the support, so we consider the subset
    \[
        \left\{
            \spl{\sigma}_{2n-2k-1},
            \spl{\sigma}_{2n-2k},\ldots,
            \spl{\sigma}_{2n-1},
            \spl{\sigma}_0,\spl{\sigma}_1
        \right\}
    \]
    whose supports contain at least one of $e$ and $v_1$.
    Consider the sequence
    $S=(2n-2k-1,\ldots,2n-1,0,1)$, so that
    \[
        \sum_{i=0}^{2n-1}c_i\spl{\sigma}_i(e)
        =\sum_{i\in S}c_i\spl{\sigma}_i(e)
        \qquad\text{and}\qquad
        \sum_{i=0}^{2n-1}c_i\spl{\sigma}_i(v_1)
        =\sum_{i\in S}c_i\spl{\sigma}_i(v_1).
    \]
    We will prove that $c_i=0$ for all $i\in S$.

    For contradiction, let $m$ be the earliest-appearing
    index in $S$ where $c_m\neq 0$. If $m=0$, then
    \[
        \sum_{i\in S}c_i\spl{\sigma}_i(e)
        =c_0\spl{\sigma}_0(e)+c_1\spl{\sigma}_1(e)
        =c_0\spl{\sigma}_0(e)
        =c_0c^{G_0}(e)
          \prod_{\substack{\alpha\in H\\s_\alpha\notin G_0}}
          t_\alpha
        \neq 0,
    \]
    a contradiction. Similarly, if $m=1$, then
    \[
        \sum_{i\in S}c_i\spl{\sigma}_i(v_1)
        =c_1\spl{\sigma}_1(v_1)\neq 0.
    \]
    So $m\in\{2n-2k-1,\ldots,2n-1\}$.

    If $m=2n-2q-1$ for $q\in\{0,\ldots,k\}$, then
    $v_m$ and $v_0$ differ by a reflection of length at most
    $2k+1$. By the definition of $H$, this reflection
    corresponds to a root $\gamma\in H$.
    Now $v_m\notin G_i$ for all
    \[
        i\in J_m\coloneqq\{m+1,\ldots,2n-1,0,1\},
    \]
    so $t_\gamma$ divides $\spl{\sigma}_i(e)$ for every
    such $i$, whereas it does not divide
    $\spl{\sigma}_m(e)$.

    Now we consider evaluation of the polynomials themselves.
    Since distinct positive roots are non-proportional,
    any nonzero point $p_\gamma\in\ker t_\gamma$ satisfies
    $t_\alpha(p_\gamma)\neq 0$ for every
    $\alpha\in\Phi^+\setminus\{\gamma\}$. As such,
    \begin{align*}
        \left[\sum_{i\in S}c_i\spl{\sigma}_i(e)\right]
            (p_\gamma)
        &=
        \left[
            \sum_{i\in\{m\}\cup J_m}c_i\spl{\sigma}_i(e)
        \right](p_\gamma)\\
        &=
        \sum_{i\in\{m\}\cup J_m}
            c_i[\spl{\sigma}_i(e)](p_\gamma)\\
        &=
        c_m[\spl{\sigma}_m(e)](p_\gamma).
    \end{align*}
    Since $[\spl{\sigma}_m(e)](p_\gamma)\neq 0$, the
    relation forces $c_m=0$, giving a contradiction.

    If, on the other hand, $m=2n-2q$ for
    $q\in\{1,\ldots,k\}$, then $v_m$ and $v_1$ differ
    by right multiplication by a reflection of length at
    most $2k+1$, so they are joined by an edge of $\cg_H$.
    The remaining contradiction proceeds exactly as it did
    for $m=2n-2q-1$, this time evaluating the functions at
    $v_1$ and then evaluating the resulting polynomials at
    a nonzero point in the kernel of the label of this edge.
    The incident labels at $v_1$ are pairwise
    non-proportional, so the same argument applies.

    Thus no such $m$ exists, and $c_i=0$ for all $i\in S$.
    For each $j\in\{0,\ldots,n-1\}$, left multiplication
    by $(s_1s_2)^{-j}$ sends $v_{2j}$ to $e$, $v_{2j+1}$
    to $v_1$, and $G_i$ to $G_{i-2j}$, with indices
    read modulo $2n$. It preserves $\cg_H$ and transforms
    its labels by an invertible linear map. Applying the
    same argument at $v_{2j}$ and $v_{2j+1}$ therefore
    gives $c_{i+2j}=0$ for every $i\in S$. Since $0,1\in S$,
    this gives $c_{2j}=c_{2j+1}=0$ for every $j$.
    Hence all coefficients vanish, proving linear
    independence.
\end{proof}
\begin{corollary}\label{cor:admissible_block_independence}
    Let $H\subsetneq\Phi^+$ be a Hessenberg space with
    $|H|=2k+1$.
    Consider the sets
    \[
        G=\{w,ws_1,\ldots,w(s_1s_2)^ks_1\}
        \qquad\text{when }\alpha_{2k+1}\in H,
    \]
    and
    \[
        G=\{w,ws_2,\ldots,w(s_2s_1)^ks_2\}
        \qquad\text{when }\alpha_{2k+2}\in H.
    \]
    For each distinct such $G$, choose constants
    $c^G(v)\in\C^\times$ for $v\in G$, and define
    $\spl{\sigma}_G\colon I_2(n)\to\C[x,y]$ by
    \[
        \spl{\sigma}_G(v)=
        \begin{cases}
            \displaystyle c^G(v)
            \prod_{\substack{\gamma\in H\\
                            vs_\gamma\notin G}}
            t_{v(\gamma)}, & v\in G,\\[6pt]
            0, & v\notin G.
        \end{cases}
    \]
    Then these functions are linearly independent over $\C$.
\end{corollary}
\begin{proof}
    Write
    \[
        H^*=\{\alpha_1,\ldots,\alpha_{2k+2}\}
            =H\cup\{\delta\}.
    \]
    For each set $G$ under consideration, the proof of
    Lemma \ref{lem:regular} shows that the subgraph induced
    by $G$ in $\cg_H$ is complete bipartite, with parts
    given by length parity. Consequently,
    $vs_\delta\notin G$ for every $v\in G$: otherwise
    $v$ and $vs_\delta$ would be opposite-parity vertices
    in $G$ whose connecting edge does not belong to
    $\cg_H$.

    Define
    \[
        \spl{\tau}_G(v)
        \coloneqq t_{v(\delta)}\spl{\sigma}_G(v).
    \]
    Since the $\delta$-edge leaves $G$ at every $v\in G$,
    these functions satisfy
    \[
        \spl{\tau}_G(v)=
        \begin{cases}
            \displaystyle c^G(v)
            \prod_{\substack{\gamma\in H^*\\
                            vs_\gamma\notin G}}
            t_{v(\gamma)}, & v\in G,\\[6pt]
            0, & v\notin G.
        \end{cases}
    \]
    Because $2k+1=|H|<n$, these functions form a subfamily
    of the linearly independent family in
    Proposition \ref{prop:linear_independence}, applied
    to $H^*$.

    Any relation $\sum_G a_G\spl{\sigma}_G=0$,
    with $a_G\in\C$, gives
    $\sum_G a_G\spl{\tau}_G=0$ after pointwise
    multiplication by $v\mapsto t_{v(\delta)}$.
    Hence every $a_G$ is zero.
\end{proof}


\section{The Left Representation}\label{sec:left_rep}
In this section, we will compute the dot action representation $\lqot{H}$.

For non-negative integer $h$ at most $(n-1)/2$ for $n$ odd and at most $n/2-1$ for $n$ even, set the following notation: 
\begin{equation}\label{eq:Ph_representation}
    \mathsf{P}_h \coloneqq \one \oplus \bigoplus_{j=1}^h V_j.
\end{equation}
\begin{remark}
    We record three facts about $\mathsf{P}_h$. First, the dimension is $2h+1$. Second, the virtual representation $C_\gamma - \mathsf{P}_h$ is in fact a representation for all $h$. Finally, when $n$ is odd then $C_\gamma \cong \mathsf{P}_{(n-1)/2}$, but when $n$ is even  $C_\gamma \cong \eta_\gamma \oplus \mathsf{P}_{n/2-1}$.
\end{remark}
\begin{definition}\label{def:H_dominant}
    Let $H$ be a Hessenberg space with simple roots $\alpha,\beta$, and let $h = \lfloor \abs{H} / 2 \rfloor$. The simple root $\alpha$ is \emph{$H$-dominant} if there is a $w = s_\beta s_\alpha\cdots \in I_2(n)$ with $\ell_H(w) = h$ such that $(w,s_\alpha w) \in \ce_H$  but $(w,s_\beta w) \notin \ce_H$.
\end{definition}
\begin{remark}
    Assume $H\subsetneq\Phi^+$. When $\abs{H}$ is even, both simple roots satisfy the definition for $H$-dominant. When $\abs{H}$ is odd, one of the simple roots will be $H$-dominant, and for the other no such $w$ exists. The main consequence is that, in view of Theorem \ref{thm:left_representation} below, when $\abs{H}$ is odd and $n$ is even, the representation $\lqot{H}$ does in fact depend on which of the two Hessenberg spaces of size $2h+1$ are chosen to be $H$.
\end{remark}

\begin{theorem}\label{thm:left_representation}
    Let $H \subsetneq \Phi^+$ be a Hessenberg space, let $\{\alpha,\beta\}$ be the simple roots, and set $h=\lfloor |H|/2\rfloor$. If $\abs{H}$ is odd, let $\alpha$ be the $H$-dominant simple root. Then, as an element of the graded representation ring of $I_2(n)$,
    \begin{equation}\label{eq:left_hessenberg_representation}
    \lqot{H}  \cong
    \begin{cases}
        \ds
        (1+q)(1+q+\cdots+q^{2h-1})\one
        \oplus
        q^h\left(C_{\alpha} \oplus C_\beta - \mathsf{P}_h - \mathsf{P}_{h-1}\right),
            &\text{if } \abs{H}=2h,\\[10pt]
        \ds
        (1+q)(1+q+\cdots+q^{2h})\one
        \oplus
        (q^h+q^{h+1})\left(C_\alpha - \mathsf{P}_h\right),
            &\text{if } \abs{H}=2h+1.
    \end{cases}
    \end{equation}
    In particular, the graded representation $\lqot{H}$ is palindromic.
\end{theorem}

\begin{remark}
Taking dimensions in \eqref{eq:left_hessenberg_representation} recovers the rank generating function in Corollary~\ref{cor:rank_gf}. 
\end{remark}
\begin{example}
    When $H=\Phi^+$, the theorem as written does not apply. Its formula extends to this case when $n$ is odd, since then $C_\alpha \cong \mathsf{P}_h$. When $n$ is even, $h=n/2$ lies outside the defined range of $\mathsf{P}_h$. In either case, the representation is
\[
    \lqot{\Phi^+}
        \cong
    (1+q)(1+q+\cdots+q^{n-1})\one,
\]
which is Proposition~\ref{prop:left_schubert_rep}.
\end{example}

\begin{example}
When $H=\{\alpha_1,\alpha_2\}$, one has $h=1$, so
\begin{align*}
    C_{\alpha_1}\oplus C_{\alpha_2} - \mathsf{P}_h - \mathsf{P}_{h-1}
        &\cong
    C_{\alpha_1}\oplus C_{\alpha_2} - V_1 - \one^{\oplus 2}
\end{align*}
It follows that
\[
    \lqot{H}
        \cong
        \begin{cases}
    \ds \one
    \oplus
    \left(
        \one^{\oplus 2}
        \oplus V_1 \oplus 
        \bigoplus_{j=2}^{(n-1)/2}V_j^{\oplus 2}
    \right)q
    \oplus
    \one q^2 &n\text{  odd}\\
    \ds \one
    \oplus
    \left(
        \one^{\oplus 2}
        \oplus
        \eta_0
        \oplus
        \eta_1
        \oplus
        V_1 \oplus 
        \bigoplus_{j=2}^{n/2-1}V_j^{\oplus 2}
    \right)q
    \oplus
    \one q^2 &n\text{  even}
    \end{cases}
\]

\end{example}

\begin{example}
    When $H = \Phi^+ \setminus \{\alpha_{n}\}$ and $n$ is even, it follows that $h = n/2-1$ and if $\alpha$ is the $H$-dominant simple root then
    \[
    \lqot{H} \cong (1+q)(1+q+\cdots+q^{n-2})\one
        \oplus
        (q^{n/2-1}+q^{n/2})\eta_\alpha.
    \]
    When $n$ is odd, it is the case that $h = (n-1)/2$ and 
    \[
    \lqot{H} \cong (1+q)(1+q+\cdots+q^{n-2})\one
        \oplus
        (q^{(n-1)/2})V_{(n-1)/2}.
    \]
\end{example}
The remainder of this section is dedicated to proving Theorem \ref{thm:left_representation}.  The key tool is \emph{equivariant freeness}. Assume throughout that $H\neq\Phi^+$, since the full-reflection case has already been proved. In particular, if $\abs{H}=2h$ or $2h+1$, then $h<n/2$. Write $\poly_d$ for the homogeneous degree-$d$ piece of
$\poly$. The quotient map
\[
    \Ms\longrightarrow\lqot{H}
\]
is graded and $I_2(n)$-equivariant. Since $\Ms$ is free over $\poly$ by
Theorem~\ref{thm:flowup}, by Maschke's theorem we can choose a graded equivariant section of this quotient. Together with graded Nakayama, this gives a noncanonical isomorphism
\begin{equation}
\label{eq:left_equivariant_freeness}
    \Ms
    \cong
    \poly\otimes_{\C}\lqot{H}
\end{equation}
of graded $I_2(n)$-representations, where $I_2(n)$ acts diagonally on the
right-hand side. Consequently,
\[
    (\Ms)_d
    \cong
    \bigoplus_{j=0}^{d}
    \poly_j\otimes(\lqot{H})_{d-j}.
\]

The following proposition proves that $\lqot{H}$ is palindromic, which reduces the Theorem \ref{thm:left_representation} to calculating $(\lqot{H})_d$ for $d \leq h$.

\begin{proposition}
\label{prop:left_quotient_duality}
The top graded piece of $\lqot{H}$ is the trivial representation. For every
$0\leq d\leq\abs{H}$, multiplication in the graded algebra $\lqot{H}$
restricts to a perfect $I_2(n)$-invariant pairing
\[
    (\lqot{H})_d
    \otimes
    (\lqot{H})_{\abs{H}-d}
    \longrightarrow
    (\lqot{H})_{\abs{H}}
    \cong
    \one.
\]
Consequently,
\begin{equation}
\label{eq:left_representation_duality}
    (\lqot{H})_{\abs{H}-d}
        \cong
    (\lqot{H})_d^*
        \cong
    (\lqot{H})_d
\end{equation}
as $I_2(n)$-representations.
\end{proposition}

\begin{proof}
Let 
$\{\spl{\rho}_w\mid w\in I_2(n)\}$ be as in
Theorem~\ref{thm:flowup}, so that its $w_0$-element is given by \eqref{eq:top_spline}. 
For each $w\in I_2(n)$, define the corresponding flow-down spline by
$\spl{\rho}^{\downarrow}_w
\coloneqq w_0\cdot\spl{\rho}_{w_0w}$.

If $v\in\supp(\spl{\rho}^{\downarrow}_w)$, then
$\spl{\rho}_{w_0w}(w_0v)\neq0$, and hence $w_0w\leq w_0v$. Left
multiplication by $w_0$ reverses Bruhat order, so $v\leq w$. Therefore
\[
    \supp\left(\spl{\rho}^{\downarrow}_w\right)
        \subseteq
    \{v\in I_2(n)\mid v\leq w\}.
\]
Moreover,
$\spl{\rho}^{\downarrow}_w(w)
=w_0\bigl(\spl{\rho}_{w_0w}(w_0w)\bigr)\neq0$, so this spline is genuinely
flow-down from $w$.

For $\alpha\in H$,
\[
\begin{aligned}
    \alpha\in\hinv(w_0w)
    &\Longleftrightarrow
    \ell(w_0ws_\alpha)<\ell(w_0w)\\
    &\Longleftrightarrow
    n-\ell(ws_\alpha)<n-\ell(w)\\
    &\Longleftrightarrow
    \alpha\notin\hinv(w).
\end{aligned}
\]
Therefore $\hinv(w_0w)=H\setminus\hinv(w)$, and
$\deg\spl{\rho}^{\downarrow}_w
=\abs{H}-\hlen(w)$.

The dot action is semilinear for the left $\poly$-module structure:
$w_0\cdot(f\spl{\rho})
=w_0(f)(w_0\cdot\spl{\rho})$.
Since $f\mapsto w_0(f)$ is an automorphism of $\poly$, the splines
$\{\spl{\rho}^{\downarrow}_w\mid w\in I_2(n)\}$ form another homogeneous
left $\poly$-basis of $\Ms$. The action of $w_0$ also preserves
$\llangle x,y\rrangle$, so reducing either basis modulo
$\llangle x,y\rrangle\Ms$ gives a homogeneous basis of $\lqot{H}$.

The class $[\spl{\rho}_{w_0}]$ is one of these quotient basis elements and
is therefore nonzero. Corollary~\ref{cor:rank_gf} gives
$\dim(\lqot{H})_{\abs{H}}=1$, so this class spans the top degree. For each
simple reflection $s_i$, conjugation by $w_0$ gives
$w_0s_iw_0=s_j$ for some $j\in\{1,2\}$, and hence
$s_iw_0=w_0s_j$. Since $\alpha_j\in H$, the vertices $w_0$ and
$s_iw_0$ are joined by an edge of $\mathcal G_H$. The support of
$\spl{\rho}_{w_0}$ is the singleton $\{w_0\}$, so
Lemma~\ref{lem:generalized_divided_difference} gives
$\spl{\rho}_{w_0}-s_i\cdot\spl{\rho}_{w_0}
=t_{\alpha_i}\spl{\tau}_i$ for some $\spl{\tau}_i\in\Ms$.
The right-hand side vanishes in the left quotient. Thus both simple
reflections fix $[\spl{\rho}_{w_0}]$, and
$(\lqot{H})_{\abs{H}}\cong\one$.

Let $\lambda\colon(\lqot{H})_{\abs{H}}\to\C$ be the invariant linear map
satisfying $\lambda([\spl{\rho}_{w_0}])=1$. The submodule
$\llangle x,y\rrangle\Ms$ is a homogeneous $I_2(n)$-stable ideal of the
spline algebra. Thus $\lqot{H}$ is a graded algebra, and the dot action
descends to an action by algebra automorphisms. Multiplication and
$\lambda$ define a pairing
\[
    \langle-,-\rangle_d\colon
    (\lqot{H})_d\otimes(\lqot{H})_{\abs{H}-d}
    \longrightarrow
    \C,
    \qquad
    \langle a,b\rangle_d
        \coloneqq
    \lambda(ab).
\]
For $g\in I_2(n)$, the algebra-automorphism property and the invariance of
$\lambda$ give
$\langle g\cdot a,g\cdot b\rangle_d
=\lambda(g\cdot(ab))
=\lambda(ab)
=\langle a,b\rangle_d$.
Hence the pairing is $I_2(n)$-invariant.

Fix $d$, and list the elements $u_1,\ldots,u_r$ satisfying
$\hlen(u_i)=d$ in a linear extension of Bruhat order, so that
$u_i\leq u_j$ implies $i\leq j$. The classes
$[\spl{\rho}_{u_1}],\ldots,[\spl{\rho}_{u_r}]$ form a basis of
$(\lqot{H})_d$, while
$[\spl{\rho}^{\downarrow}_{u_1}],\ldots,
[\spl{\rho}^{\downarrow}_{u_r}]$ form a basis of
$(\lqot{H})_{\abs{H}-d}$.

Set
$M_{ij}
\coloneqq
\langle[\spl{\rho}_{u_i}],
[\spl{\rho}^{\downarrow}_{u_j}]\rangle_d$.
The product
$\spl{\rho}_{u_i}\spl{\rho}^{\downarrow}_{u_j}$ is supported in
$\{v\in I_2(n)\mid u_i\leq v\leq u_j\}$. If $i>j$, then
$u_i\nleq u_j$ by the choice of linear extension, so this interval is
empty and $M_{ij}=0$. Thus $M$ is upper triangular.

If $\alpha\in\hinv(u_i)$, then $u_is_\alpha<u_i$, so
$\spl{\rho}_{u_i}(u_is_\alpha)=0$. The spline condition along this edge
gives $t_{u_i(\alpha)}\mid\spl{\rho}_{u_i}(u_i)$. If
$\alpha\in H\setminus\hinv(u_i)$, then $u_is_\alpha>u_i$, so
$\spl{\rho}^{\downarrow}_{u_i}(u_is_\alpha)=0$ and
$t_{u_i(\alpha)}
\mid\spl{\rho}^{\downarrow}_{u_i}(u_i)$.

Since $u_i$ permutes the roots and $H\subseteq\Phi^+$ contains at most one
root from each root line, the linear forms $t_{u_i(\alpha)}$, for
$\alpha\in H$, are pairwise nonassociate. The ring $\poly$ is a unique
factorization domain, so the corresponding products divide the two spline
values. Their degrees equal the degrees of the respective splines, and
both splines are nonzero at $u_i$. There are therefore
$a_i,b_i\in\C^\times$ such that
\[
    \spl{\rho}_{u_i}(u_i)
        =
    a_i\prod_{\alpha\in\hinv(u_i)}t_{u_i(\alpha)},
    \qquad
    \spl{\rho}^{\downarrow}_{u_i}(u_i)
        =
    b_i\prod_{\alpha\in H\setminus\hinv(u_i)}
        t_{u_i(\alpha)}.
\]

Any vertex in the supports of both $\spl{\rho}_{u_i}$ and
$\spl{\rho}^{\downarrow}_{u_i}$ must satisfy
$u_i\leq v\leq u_i$, and hence must equal $u_i$. Their product is
therefore supported only at $u_i$, where its value is
$a_ib_i\prod_{\alpha\in H}t_{u_i(\alpha)}$.

The translated top spline $(u_iw_0)\cdot\spl{\rho}_{w_0}$ has support
$(u_iw_0)\{w_0\}=\{u_i\}$. Its value at $u_i$ is
\[
\begin{aligned}
    \bigl((u_iw_0)\cdot\spl{\rho}_{w_0}\bigr)(u_i)
    &=
    u_iw_0\left(
        \spl{\rho}_{w_0}\bigl((u_iw_0)^{-1}u_i\bigr)
    \right)\\
    &=
    u_iw_0\left(\spl{\rho}_{w_0}(w_0)\right)\\
    &=
    u_iw_0\left(
        \prod_{\alpha\in H}t_{w_0(\alpha)}
    \right)\\
    &=
    \prod_{\alpha\in H}t_{u_i(\alpha)}.
\end{aligned}
\]
The two point-supported splines consequently satisfy
\[
    \spl{\rho}_{u_i}\spl{\rho}^{\downarrow}_{u_i}
        =
    a_ib_i\,(u_iw_0)\cdot\spl{\rho}_{w_0}.
\]
Since the top quotient representation is trivial,
$[(u_iw_0)\cdot\spl{\rho}_{w_0}]
=[\spl{\rho}_{w_0}]$, and hence $M_{ii}=a_ib_i\neq0$.

The matrix $M$ is upper triangular with nonzero diagonal, so the pairing
is perfect. It induces an $I_2(n)$-equivariant isomorphism
$(\lqot{H})_{\abs{H}-d}\cong(\lqot{H})_d^*$. Every irreducible
representation of $I_2(n)$ is self-dual, since every element of
$I_2(n)$ is conjugate to its inverse. This proves
\eqref{eq:left_representation_duality}.
\end{proof}
The following Lemma shows how many copies of the trivial representation appear.
\begin{lemma}\label{lem:trivial_left_subreps}
    The representation $\lqot{H}$ contains a sub-representation isomorphic to
    \[(1+q)(1+q+\cdots+q^{\abs{H}-1})\,\one.\]
\end{lemma}
\begin{proof}
    Consider different graded pieces of the injection $\splines{\Phi^+} \hookrightarrow \Ms$. For $0 \leq k \leq h-1$, the coefficient of $q^k$ in $\Hilb(\Ms;q)$ is the same as in $\Hilb(\splines{\Phi^+};q)$. So this injection must be a surjection, and so $\left(\splines{\Phi^+}\right)_k = \left(\Ms\right)_k$ for all $0 \leq k \leq h-1$. Since every element of $J$ is degree at least one,  $\left(J\splines{\Phi^+}\right)_k = \left(J\Ms\right)_k$ for all $1 \leq k \leq h$. Thus, when passing to the quotient, there is an injection
    \[\left(\lqot{\Phi^+}\right)_k \hookrightarrow \left(\lqot{H}\right)_k \qquad \text{for all}\;\; 0 \leq k \leq h.\]

    So $\lqot{H}$ contains a sub-representation isomorphic to $\left(1 + 2q + \cdots +2q^h\right)\one$. Proposition \ref{prop:left_quotient_duality} ensures that $\lqot{H}$ also contains $\left(q^{|H|} + 2q^{|H|-1} + \cdots + 2q^{|H|-h}\right)\one$ as a sub-representation, and the claim follows.
\end{proof}

\begin{proof}[Proof of Theorem \ref{thm:left_representation}]
    Set $h=\lfloor |H|/2\rfloor$. We first construct
    permutation subrepresentations of $(\Ms)_h$, then
    compare their multiplicities with those
    of $(J\Ms)_h$ to determine $(\lqot{H})_h$.

    To construct these subrepresentations, we choose
    degree-$h$ basis splines whose supports allow us
    to determine their stabilizers and prove independence
    of their dot-action orbits.
    Let $z$ be an alternating word of length $t$ ending
    in $s_i$, where $i\in\{1,2\}$ and $2t\leq n$.
    Taking the reflections associated to successive
    suffixes of the reduced word for $z$ gives
    \[
        N_R(z)=\{\alpha_i,\alpha_{i+2},\ldots,
                         \alpha_{i+2t-2}\}.
    \]
    Since $w_0$ interchanges positive and negative roots,
    \[
        N_R(w_0z)=\Phi^+\setminus N_R(z).
    \]

    Suppose first that $|H|=2h$. Take the two elements $z$
    of length $h$. Their right inversion sets are both
    contained in $H=\{\alpha_1,\ldots,\alpha_{2h}\}$.
    The corresponding elements $w_0z$ have length $n-h$
    and $H$-length $h$. Denote them by $w$ and $v$, with
    $s_\alpha$ the simple left ascent of $w$ and
    $s_\beta$ the simple left ascent of $v$.

    If $|H|=2h+1$, then $H$ contains exactly one of
    $\alpha_{2h+1}$ and $\alpha_{2h+2}$, together with
    $\alpha_1,\ldots,\alpha_{2h}$. Take the unique element
    $z$ of length $h+1$ whose right inversion set is
    contained in $H$, and set $w=w_0z$. Then
    \[
        \ell(w)=n-h-1,
        \qquad
        \hlen(w)=|H|-|N_R(z)|=h.
    \]
    Write $s_\alpha$ for the simple left ascent of $w$,
    and let $\beta$ be the other simple root. We check
    below that $\alpha$ is $H$-dominant.

    In either case, write $w=w_0z$ and $t=\ell(z)$.
    The reflection $w^{-1}s_\alpha w$ is the reflection
    of length $2t-1$ corresponding to the last root in
    the displayed description of $N_R(z)$. In particular,
    it corresponds to a root in $H$, so
    $(w,s_\alpha w)\in\ce_H$.

    On the other hand, if $2t<n$, then
    Lemma \ref{lem:leftright_length}, applied to the
    corresponding simple left ascent of $z$, gives
    \[
        \ell(w^{-1}s_\beta w)=2t+1.
    \]
    This is larger than the length of every reflection
    corresponding to a root in $H$. If $2t=n$, which
    can occur here only in the odd case, then
    $w^{-1}s_\beta w$ is the other reflection of length
    $2t-1=2h+1$, whose root is not in $H$. Thus
    $(w,s_\beta w)\notin\ce_H$ in every case.
    This verifies the conditions in
    Definition \ref{def:H_dominant}. In the even case,
    the same argument applies to $v$, with $\alpha$
    and $\beta$ interchanged.

    Since $H$ is proper, the lengths of these vertices
    satisfy $2\ell(w)+1>n$ and, in the even case,
    $2\ell(v)+1>n$. Thus Theorem \ref{thm:flowup} gives
    \[
        \spl{\rho}_w
        =\partial_\alpha\spl{\rho}_{s_\alpha w},
    \]
    and similarly
    $\spl{\rho}_v=\partial_\beta\spl{\rho}_{s_\beta v}$
    in the even case. These splines have degree $h$.

    Let
    \[
        U_w=\{u\in I_2(n)\mid w\leq u\}.
    \]
    The underlying undirected graph of
    $\cg_{\{\alpha,\beta\}}$ is a cycle of length $2n$.
    By the rank-two description of Bruhat order, $U_w$
    consists of $2t$ consecutive vertices in this cycle,
    with endpoints $w$ and $s_\alpha w$. The endpoint reflection
    $w^{-1}s_\alpha w$ corresponds to a root in $H$,
    so Lemma \ref{lem:regular} applies. Its proof shows that every vertex of even length
    in $U_w$ is adjacent to every vertex of odd length,
    with $t$ vertices of each parity. Thus each vertex
    has $t$ neighbors within $U_w$, and hence
    $|H|-t=h$ neighbors outside $U_w$.

    The spline $\spl{\rho}_w$ is nonzero and supported
    within $U_w$, and has degree $h$. Therefore
    Corollary \ref{cor:rationally_smooth} shows that its
    support is exactly $U_w$ and gives the product formula
    for its nonzero values. The same
    conclusions hold for $\spl{\rho}_v$ in the even case.

    Any element stabilizing $\spl{\rho}_w$ must stabilize
    $U_w$. Since the subgraph induced by $U_w$ in the cycle
    $\cg_{\{\alpha,\beta\}}$ is a path, any such element
    permutes its two endpoints. Left multiplication acts freely on
    vertices, so the stabilizer of $U_w$ has order at
    most two. Moreover,
    \[
        s_\alpha\cdot
        \bigl(\partial_\alpha\spl{\rho}_{s_\alpha w}\bigr)
        =
        \partial_\alpha\spl{\rho}_{s_\alpha w},
    \]
    so $s_\alpha$ stabilizes $\spl{\rho}_w$. Hence
    \[
        \operatorname{Stab}(\spl{\rho}_w)
        =\operatorname{Stab}(U_w)
        =\llangle s_\alpha\rrangle.
    \]
    In the even case, the corresponding stabilizers
    for $\spl{\rho}_v$ and $U_v$ are
    $\llangle s_\beta\rrangle$.

    Left translates of these splines satisfy the
    product formula in
    Corollary \ref{cor:rationally_smooth} on their
    translated supports.
    Since $\spl{\rho}_w$ and its support have the same
    stabilizer, distinct splines in its orbit have
    distinct supports. The analogous statement holds
    for $\spl{\rho}_v$.
    In the even case, the elements $w$ and $v$ have
    the same length and different simple right ascents.
    Starting from $w$ or $v$, the path through its
    upper order ideal begins with an edge given by
    right multiplication by its simple right ascent.
    Each path has $2h-1$ edges, given alternately by
    right multiplication by the two simple reflections.
    Thus its first and last edges are given by the
    same simple reflection, and this reflection differs
    between the two paths.
    Left multiplication sends an edge $\{u,us_i\}$ to
    $\{gu,(gu)s_i\}$, so it preserves the simple
    reflection defining the edge. Therefore no left
    translate of $U_w$ equals $U_v$, and their orbits
    under left multiplication are disjoint.

    Proposition \ref{prop:linear_independence}, with
    $k=h-1$, now shows that the dot-action orbits of
    $\spl{\rho}_w$ and $\spl{\rho}_v$ in the even case
    are jointly linearly independent. In the odd case,
    Corollary \ref{cor:admissible_block_independence},
    with $k=h$, shows that the dot-action orbit of
    $\spl{\rho}_w$ is linearly independent. Their stabilizers therefore give
    \[
        C_\alpha
        \cong
        \spann_\C\bigl(I_2(n)\cdot\spl{\rho}_w\bigr),
    \]
    and, in the even case,
    \[
        C_\beta
        \cong
        \spann_\C\bigl(I_2(n)\cdot\spl{\rho}_v\bigr).
    \]
    In particular, we have inclusions
    \[
        \begin{cases}
            C_\alpha\oplus C_\beta
            \hookrightarrow(\Ms)_h,
                & |H|=2h,\\
            C_\alpha\hookrightarrow(\Ms)_h,
                & |H|=2h+1.
        \end{cases}
    \]

    For $1\leq k<h$, Lemma \ref{lem:trivial_left_subreps}
    and Corollary \ref{cor:rank_gf} give
    $(\lqot{H})_k\cong\one^{\oplus 2}$, while
    $(\lqot{H})_0\cong\one$. Thus equivariant freeness
    gives, in the representation ring of $I_2(n)$,
    \begin{equation}\label{eq:left_middle_subtraction}
        [(\lqot{H})_h]
        =
        [(\Ms)_h]
        -
        [\poly_h]
        -
        2\sum_{d=1}^{h-1}[\poly_d].
    \end{equation}
    For $1\leq d<n/2$, the symmetric powers of the
    defining representation decompose as
    \[
        \poly_d
        \cong
        \bigoplus_{j=0}^{\lfloor(d-1)/2\rfloor}V_{d-2j}
        \oplus
        \begin{cases}
            0, & d\text{ is odd},\\
            \one, & d\text{ is even}.
        \end{cases}
    \]
    Since $H$ is proper, $h<n/2$, so
    \begin{align*}
        \C[x,y]_h
        \oplus 2\bigoplus_{d=1}^{h-1}\C[x,y]_d
        &\cong
        \left(\bigoplus_{d=1}^h\C[x,y]_d\right)
        \oplus
        \left(\bigoplus_{d=1}^{h-1}\C[x,y]_d\right)\\
        &\cong
        \one^{\oplus(h-1)}
        \oplus
        \bigoplus_{j=1}^h V_j^{\oplus(h+1-j)}.
    \end{align*}
    This is the decomposition of the subrepresentation
    $(J\Ms)_h$ removed in passing to the quotient.
    In particular, it contains no $V_j$ with $j>h$,
    no copy of either $\eta_0$ or $\eta_1$ when $n$
    is even, and only one copy of $V_h$.

    Since $(J\Ms)_h$ contains no $V_j$ with $j>h$
    and no $\eta_i$, all copies of these constituents
    in the permutation subrepresentations constructed
    above survive in the quotient. In the even case,
    $C_\alpha\oplus C_\beta$ also supplies two copies
    of $V_h$, whereas $(J\Ms)_h$ contains only one,
    so at least one survives. Thus
    \eqref{eq:reflection_coset_decomposition} and
    \eqref{eq:Ph_representation} show that the
    constituents of
    \[
        \begin{cases}
            C_\alpha\oplus C_\beta
                -\mathsf{P}_h-\mathsf{P}_{h-1},
                & |H|\text{ even},\\
            C_\alpha-\mathsf{P}_h,
                & |H|\text{ odd},
        \end{cases}
    \]
    survive in $(\lqot{H})_h$ with at least the
    indicated multiplicities.

    Lemma \ref{lem:trivial_left_subreps} gives two trivial
    constituents in degree $h$, and
    Corollary \ref{cor:rank_gf} leaves no room for any
    other irreducible summands. Thus, when $|H|$ is even,
    \[
        (\lqot{H})_h
        \cong
        \one^{\oplus 2}
        \oplus
        \left(
            C_\alpha\oplus C_\beta
            -\mathsf{P}_h-\mathsf{P}_{h-1}
        \right).
    \]
    When $|H|$ is odd,
    \[
        (\lqot{H})_h
        \cong
        \one^{\oplus 2}
        \oplus
        \left(C_\alpha-\mathsf{P}_h\right),
    \]
    and Proposition \ref{prop:left_quotient_duality}
    gives the same decomposition in degree $h+1$.

    Finally, Proposition \ref{prop:left_quotient_duality}
    determines the remaining upper degrees from the
    lower degrees already computed. In particular,
    $(\lqot{H})_k\cong\one^{\oplus 2}$ for
    $|H|-h<k<|H|$, and
    $(\lqot{H})_{|H|}\cong\one$.
    This proves the stated formula.
\end{proof}

\begin{corollary}\label{cor:not_permutation}
    The left representation $\lqot{H}$ is not always a permutation representation. Moreover, this can fail in the first degree.
\end{corollary}
\begin{proof}
    For $n=5$ and $H=\{\alpha_1,\alpha_2,\alpha_3\}$, Theorem \ref{thm:left_representation} gives, after forgetting the grading,
    \[
        \lqot{H}\cong\one^{\oplus 6}\oplus V_2^{\oplus 2}.
    \]
    Its character value at $r$ is $6+4\cos(4\pi/5)=5-\sqrt5$. Since permutation characters take nonnegative integer values, $\lqot{H}$ is not a permutation representation. For the same representation we also get that $(\lqot{H})_1\cong \one^{\oplus 2}\oplus V_2$, whose character value at $r$ is $2+2\cos(4\pi/5)= \frac{1}{2}(3-\sqrt{5})$, which is also irrational.
\end{proof}


\section{The Right Representation}\label{sec:right_rep}
This section will compute the right action representation $\rqot{H}$ in Theorem \ref{thm:right_representation}. The proof will leverage equivariant freeness and Theorem \ref{thm:left_representation}, and follow a similar structure as the previous section.

Since $\Ms$ is free over the right action as well as the left, there is a noncanonical isomorphism
\begin{equation}\label{eq:right_equivariant_freeness}
    \Ms \cong
    \C[\spl{x}_r,\spl{y}_r] \otimes_\C \rqot{H}
    \cong 
    \left(\bigoplus_{i \geq 0} \one^{\oplus i+1}q^i\right) \otimes \rqot{H}
\end{equation}
of graded $I_2(n)$-representations, where $I_2(n)$ acts diagonally on the right-hand side. Consequently,
\[
\left(\Ms\right)_d \cong \bigoplus_{j=0}^d \one^{\oplus j+1} \otimes (\rqot{H})_{d-j} \cong \bigoplus_{j=0}^d (\rqot{H})_{d-j}^{\oplus j+1}.
\]

\begin{proposition}
\label{prop:right_quotient_duality}
The top graded piece of $\rqot{H}$ is the sign representation. For every
$0\leq d\leq\abs{H}$, multiplication in the graded algebra $\rqot{H}$
restricts to a perfect $I_2(n)$-equivariant pairing 
\[
    (\rqot{H})_d
    \otimes
    (\rqot{H})_{\abs{H}-d}
    \longrightarrow
    (\rqot{H})_{\abs{H}}
    \cong
    \sgn.
\]

Consequently,
\begin{equation}
\label{eq:right_representation_duality}
    \sgn \otimes (\rqot{H})_{\abs{H}-d}
        \cong
    (\rqot{H})_d
\end{equation}
as $I_2(n)$-representations.
\end{proposition}
\begin{proof}
    We will compute $(\rqot{H})_{|H|}$ directly. The rest of the statement follows exactly as it does in the proof of Proposition \ref{prop:left_quotient_duality} with the $\sgn$-tensor instead of the trivial tensor.

    Let $\alpha$ be a simple root and $\spl{r}_\gamma$ be the spline in $\C[\spl{x}_r,\spl{y}_r]$ whose evaluation at $w$ is $t_{w(\gamma)}$. The following table gives evaluations of splines at $w_0$ and $s_\alpha w_0$ for several splines in $\Ms$. Each follows directly from their definitions.

    \begin{center}
    {\renewcommand{\arraystretch}{1.8}
    \begin{tabular}{c|c|c}
     & $s_\alpha w_0$ & $w_0$ \\\hline
     $\spl{\rho}_{w_0}$ & $0 $& $\ds \prod_{\gamma \in H} t_{w_0(\gamma)}$ \\\hline
     $s_\alpha \cdot \spl{\rho}_{w_0}$ & $\ds \prod_{\gamma \in H} s_\alpha t_{w_0(\gamma)}$ & $0$ \\\hline
     $\spl{\rho}_{s_\alpha w_0}$ & $\ds \frac{-1}{t_\alpha} \prod_{\gamma \in H} t_{s_\alpha w_0(\gamma)}$ & $\ds \frac{1}{t_\alpha} \prod_{\gamma \in H} t_{ w_0(\gamma)}$ \\\hline
     $\spl{r}_{w_0(\alpha)}$& $-t_\alpha$ & $t_\alpha$
    \end{tabular}
    }
    \end{center}
So $s_\alpha \cdot \spl{\rho}_{w_0} = \spl{r}_{w_0(\alpha)} \spl{\rho}_{s_\alpha w_0} - \spl{\rho}_{w_0}$. Since the same holds replacing $\alpha$ with the other simple root $\beta$ and $(\rqot{H})_{\abs{H}}$ is one-dimensional, we get $(\rqot{H})_{\abs{H}} \cong \sgn$.
\end{proof}

\begin{lemma}\label{lem:coinv_right_subreps}
    Let $H \neq \Phi^+$. The representation $\rqot{H}$ contains a sub-representation isomorphic to 
    \begin{equation}\label{eq:h_coinvariant}
    \mathsf{Co}_H \coloneqq \begin{cases}
        \one + V_1\,q  + \cdots + V_h\,q^h + V_{h-1}\,q^{h+1} + \cdots + V_1\,q^{2h-1} + \sgn q^{2h} & \text{if }\abs{H} = 2h \\
        \one + V_1\,q +\cdots + V_h\,q^h + V_{h}\,q^{h+1} + \cdots + V_1\,q^{2h} + \sgn q^{2h+1} & \text{if }\abs{H} = 2h+1.
    \end{cases}
    \end{equation}
\end{lemma}
\begin{proof}
    The proof that 
    \[
    \one + V_1q + \cdots + V_hq^h
    \]
    appears as a sub-representation is exactly the same as that of Lemma \ref{lem:trivial_left_subreps}, replacing $\lqot{H}$ with $\rqot{H}$. Note that since $H \neq \Phi^+$, we never see the middle degree of the coinvariant algebra appear when $n$ is even.

    The remaining graded sub-representations appear as a consequence of Proposition \ref{prop:right_quotient_duality}, since $\sgn \otimes V_i \cong V_i$ and $\sgn \otimes \one \cong \sgn$.
\end{proof}

\begin{theorem}\label{thm:right_representation}
    Let $H \subsetneq \Phi^+$ be a Hessenberg space, let $\{\alpha,\beta\}$ be the simple roots, and set $h=\lfloor |H|/2\rfloor$. If $\abs{H}$ is odd, let $\alpha$ be the $H$-dominant simple root. Recall the definition of $\mathsf{Co}_H$ in Lemma \ref{lem:coinv_right_subreps} as well as $\mathsf{P}_h$ in Section \ref{sec:left_rep}.
    Then, as an element of the graded representation ring of $I_2(n)$,
    \begin{equation}\label{eq:right_hessenberg_representation}
    \rqot{H}  \cong
    \begin{cases}
        \ds
        \mathsf{Co}_H
        \oplus
        q^h\left(C_{\alpha} \oplus C_\beta - \mathsf{P}_h - \mathsf{P}_{h-1}\right),
            &\text{if } \abs{H}=2h,\\[10pt]
        \ds
        \mathsf{Co}_H
        \oplus
        q^h\left(C_\alpha - \mathsf{P}_h\right)
        \oplus 
        q^{h+1}\left(C_\beta - \mathsf{P}_h\right),
            &\text{if } \abs{H}=2h+1.
    \end{cases}
    \end{equation}
\end{theorem}
\begin{proof}
    Comparing equations \eqref{eq:left_equivariant_freeness} and \eqref{eq:right_equivariant_freeness} with Lemmas \ref{lem:trivial_left_subreps} and \ref{lem:coinv_right_subreps} at degree $h$, one gets 
    \[
    (\lqot{H})_h \oplus (h-1)\,\one \oplus \bigoplus_{j=1}^h (h+1-j) V_j \cong (\rqot{H})_h \oplus (h+1)\,\one \oplus \bigoplus_{j=1}^{h-1}(h+1-j) V_j.
    \]
    Canceling, we get 
    \[ (\lqot{H})_h \oplus V_h \cong (\rqot{H})_h \oplus 2\one.\]
    Comparing with Theorem \ref{thm:left_representation} $(\rqot{H})_h$ computes as claimed. If $(\rqot{H})_{h+1}$ is needed, then Proposition \ref{prop:right_quotient_duality} applies, noting that when $n$ is even then $\eta_\alpha \otimes \sgn = \eta_\beta$. The remaining graded pieces are handled by Lemma \ref{lem:coinv_right_subreps} for dimension reasons. 
\end{proof}

\printbibliography

\end{document}